\documentclass{amsart}%
\usepackage{amsfonts}
\usepackage{amsmath}
\usepackage{amssymb}
\usepackage{graphicx}%
\providecommand{\U}[1]{\protect\rule{.1in}{.1in}}
\newtheorem{theorem}{Theorem}
\theoremstyle{plain}

\newtheorem{corollary}{Corollary}

\newtheorem{lemma}{Lemma}

\newtheorem{proposition}{Proposition}
\newtheorem{remark}{Remark}

\numberwithin{equation}{section}
\begin{document}
\title[Askey--Wilson scheme]{On the families of polynomials forming a part of the Askey--Wilson scheme and
their probabilistic applications.}
\author{Pawe\l \ J. Szab\l owski}
\address{Emeritus in the Department of Mathematics and Information Sciences, \\
Warsaw University of Technology\\
ul. Koszykowa 75, 00-662 Warsaw, Poland}
\email{pawel.szablowski@gmail.com}
\thanks{The author is deeply grateful to the unknown referee for his positive opinion
of the paper as well as for pointing out numerous misprints, small mistakes as
well as many remarks improving its style.}
\date{April 2020}
\dedicatory{To the memory of Richard Askey}
\begin{abstract}
We review the properties of six families of orthogonal polynomials that form
the main bulk of the collection called the Askey--Wilson scheme of
polynomials. We give connection coefficients between them as well as the
so-called linearization formulae and other useful important finite and
infinite expansions and identities. An important part of the paper is the
presentation of probabilistic models where most of these families of
polynomials appear. These results were scattered within the literature in
recent 15 years. We put them together to enable an overall outlook on these
families and understand their crucial r\^{o}le in the attempts to generalize
Gaussian distributions and find their bounded support generalizations. The
paper is based on 65 items in the predominantly recent literature.

\end{abstract}
\maketitle

\part{Introduction\label{intr}}

The aim of this paper is to review the basic properties of the orthogonal
polynomials of one variable, which constitute a part of the so-called
Askey-Wilson (briefly AW) scheme. The AW scheme of orthogonal polynomials is a
large family of orthogonal polynomials that can be divided into 5 subsets of
orthogonal polynomials and each of these subsets is characterized by the
common number of parameters that define it. Even though we do not analyze and
review all families of polynomials that form AW scheme, but only a part of it
for the sake of brevity, we will refer to this part as to the AW scheme. The
whole AW scheme is presented in, e.g., \cite{KLS}. In fact, among parameters,
there is one playing a special r\^{o}le. It will be traditionally denoted by
the letter $q$ and, usually, it will be assumed that $q\in(-1,1)$. All other
parameters defining appropriate families of polynomials will be denoted by
$a,b,c,d$. Hence we will discuss the family characterized by $1$ parameter
$q$. In fact, in the one-parameter family, there are two families of
polynomials that have some importance. Namely, one of them constitutes the
so-called Rogers-Szeg\"{o} family that are not orthogonal and closely related
with them the so-called $q-$Hermite polynomials. The next are the families
with $2$ parameters, i.e.,$q$ and $a$. Here we will have the so-called big
$q-$Hermite polynomials and the so-called Rogers or $q-$ultraspherical
polynomials. Then we have $3$ parameter families of polynomials. Here are the
so-called Al-Salam--Chihara polynomials. After that, we have $4$ parameter
families including, the so-called continuous dual Hahn polynomials. Finally,
we have the so-called Askey-Wilson polynomials a large family of polynomials
depending on $5$ parameters $q,a,b,c,d$.

Apart from these families, we will consider the so-called Chebyshev
polynomials of the first and the second kind, whose some linear combinations
are important as the special cases for $q\allowbreak=\allowbreak0$ of the
families of polynomials, considered in this paper. We will also consider the
so-called Hermite (more precisely the so-called probabilistic Hermite
polynomials), that is monic polynomials that are orthogonal with respect to
$\exp(-x^{2}/2)/\sqrt{2\pi},$ i.e.,the density of the Normal $N(0,1)$ (or
Gaussian) distribution.

Why these families of polynomials are important?. Well, because firstly, they
have nice applications in mathematical analysis (see e.g., monograph of B.
Simon \cite{Sim98}), combinatorics (see e.g., papers of et al. Corteel, in
particular, \cite{Corteel10} and the reference there) and what is surprising
in representation theory of compact Lie groups (see e.g., monograph
\cite{Vilen95} and the later papers like e.g., \cite{Nest10}), probability
theory, providing examples of compactly supported families of distributions
and stochastic processes. We will present the most known of these
probabilistic applications, since they seem to be less obvious and popular
among the probabilists as well as among specialists in the so-called
$q-$series theory.

The polynomials that we are going to present are of two types that can be
distinguished on the basis of possible applications. They differ, in fact, in
so to say, re-scaling. That is, if one transforms linearly the support of the
measure, that makes a particular family of polynomials orthogonal, then we
switch between these two types of polynomials. The main difference between
them, as stated above, lies in the possible applications. Namely, the
polynomials of the first type have applications in the theory of orthogonal
polynomials and combinatorics. The polynomials of the second type appear
mostly in stochastic applications since the linear transformation of the
support enables consideration of the case $q\rightarrow1^{-}$ which results in
making the Gaussian case included in all applications. That enables some
comparisons of the properties of the probabilistic models built with the help
of AW polynomials with the properties of the Gaussian case. For probabilists
such comparisons are very important.

Basically, the paper consists of three parts: the first one, i.e., Part
\ref{intr} is an introduction as well as is dedicated to the notation and the
introduction of some auxiliary families of polynomials. The second one, i.e.,
Part \ref{sup[-1,1]} is devoted to the polynomials of the first type. The
third one, i.e., Part \ref{prob} deals with the stochastic application and
description of new interesting families of distributions and stochastic
processed that emerge here. This apart presents some infinite expansions
called kernels. The examination, if such kernel is nonnegative for a given
range of unknowns, is important in probability theory leading to the so-called
Lancaster kernels.

Part \ref{sup[-1,1]} will be divided into sections referring to the number of
parameters the polynomials presented in this subsection depend upon. Since all
polynomials, except for some auxiliary ones, will depend on the parameter $q$
we will not count it. Hence the sections will present polynomials depending on
$0,1,2,3,4$ parameters.

All polynomials, that we are going to present here, belong to the $q-$series
theory, a part of mathematical analysis. This branch of analysis developed
rapidly in the last $20$ years of the previous century. The new insight into
this theory and some surprising applications can be found in the works of
Lusztig and his followers like, say, F. Brenti or L. J. Billera and many
others. These works indicate connections of $q-$series theory with Macdonald
polynomials and counting points of various algebraic varieties that arise in
group representation theory and/or representation of Weil group.

As far as the probabilistic applications of AW polynomials are concerned,
which is the main application presented in this review, about $20$ years ago,
there appeared the paper of W. Bryc \cite{bryc1} where the links between
$q-$series theory and the theory of stochastic processes came into the light
for the first time. Since then, several papers were published, where
practically all families of orthogonal polynomials mentioned in this paper
have their probabilistic interpretation (after proper re-scaling, described
below). In Section, \ref{prob} we will present all known, so far, of these applications.

To define these polynomials and briefly describe their properties one has to
adopt the notation used in $q$-series theory. Moreover, the terminology
concerning these polynomials is not fixed and under the same name appears
sometimes different, but related to one another families of polynomials. Thus,
one has to be aware of these differences.

That is why the next section of the paper is devoted to notations, definitions
and the next one \ref{pom} to the presentation of some auxiliary families of
polynomials. The following Part \ref{sup[-1,1]} is dedicated to the
definitions and basic properties of the families of polynomials that form AW
scheme and the measures that make them orthogonal are supported on $[-1,1]$.
We present there also some 'finite expansions' formulae establishing
relationships between these families of polynomials, including listing known
the so-called 'connection coefficients' and 'linearization' formulae. The last
Part \ref{prob} is devoted to the probabilistic applications, that is, more
precisely, to the description of probabilistic models where polynomials from
AW scheme appear including infinite expansions, in particular, kernels
involving discussed polynomials. It consists of two subsections the first of
which is devoted to different generalizations of the Mehler expansion formula,
the second one to some useful infinite expansions including reciprocals of
some kernels that have an auxiliary meaning.

\section{Notations and definitions}

$q$ is a parameter. We will assume that $-1<q\leq1$ unless otherwise stated.
The case $q\allowbreak=\allowbreak1$ may not always be considered directly,
but sometimes as left-hand side limit ( i.e.,$q\longrightarrow1^{-}$). We will
point out these cases.

We will use traditional notations of the $q-$series theory i.e.,%
\[
\left[  0\right]  _{q}\allowbreak=\allowbreak0,~\left[  n\right]
_{q}\allowbreak=\allowbreak1+q+\ldots+q^{n-1}\allowbreak,\left[  n\right]
_{q}!\allowbreak=\allowbreak\prod_{j=1}^{n}\left[  j\right]  _{q},\text{with
}\left[  0\right]  _{q}!\allowbreak=1\text{,}%
\]%
\[%
%TCIMACRO{\QATOPD{[}{]}{n}{k}}%
%BeginExpansion
\genfrac{[}{]}{0pt}{}{n}{k}%
%EndExpansion
_{q}=\left\{
\begin{array}
[c]{ccc}%
\frac{\left[  n\right]  _{q}!}{\left[  n-k\right]  _{q}!\left[  k\right]
_{q}!} & , & n\geq k\geq0\\
0 & , & \text{otherwise}%
\end{array}
\right.  \text{.}%
\]
$\binom{n}{k}$ will denote the ordinary, well known binomial coefficient.

It is useful to use the so-called $q-$Pochhammer symbol for $n\geq1:$%
\[
\left(  a|q\right)  _{n}=\prod_{j=0}^{n-1}\left(  1-aq^{j}\right)  ,~~\left(
a_{1},a_{2},\ldots,a_{k}|q\right)  _{n}\allowbreak=\allowbreak\prod_{j=1}%
^{k}\left(  a_{j}|q\right)  _{n}\text{,}%
\]
with $\left(  a|q\right)  _{0}\allowbreak=\allowbreak1$.

Often $\left(  a|q\right)  _{n}$ as well as $\left(  a_{1},a_{2},\ldots
,a_{k}|q\right)  _{n}$ will be abbreviated to $\left(  a\right)  _{n}$ and
\newline$\left(  a_{1},a_{2},\ldots,a_{k}\right)  _{n},$ if it will not cause misunderstanding.

We will also use the following symbol $\left\lfloor n\right\rfloor $ to denote
the largest integer not exceeding $n$.

It is worth to mention the following two formulae, that are well known.
Namely, the following formulae are true for $\left\vert t\right\vert <1,$
$\left\vert q\right\vert <1$ (derived already by Euler, see \cite{Andrews1999}
Corollary 10.2.2)
\begin{align}
\frac{1}{(t)_{\infty}}\allowbreak &  =\allowbreak\sum_{k\geq0}\frac{t^{k}%
}{(q)_{k}}\text{,}\label{binT}\\
(t)_{\infty}\allowbreak &  =\allowbreak\sum_{k\geq0}(-1)^{k}q^{\binom{k}{2}%
}\frac{t^{k}}{(q)_{k}}\text{.} \label{obinT}%
\end{align}

It is easy to notice that $\left(  q\right)  _{n}=\left(  1-q\right)
^{n}\left[  n\right]  _{q}!$ and that%
\[%
%TCIMACRO{\QATOPD{[}{]}{n}{k}}%
%BeginExpansion
\genfrac{[}{]}{0pt}{}{n}{k}%
%EndExpansion
_{q}\allowbreak=\allowbreak\allowbreak\left\{
\begin{array}
[c]{ccc}%
\frac{\left(  q\right)  _{n}}{\left(  q\right)  _{n-k}\left(  q\right)  _{k}}
& , & n\geq k\geq0\\
0 & , & \text{otherwise}%
\end{array}
\right.  \text{.}%
\]
\newline The above-mentioned formula is just an example where direct setting
$q\allowbreak=\allowbreak1$ is senseless however, the passage to the limit
$q\longrightarrow1^{-}$ makes sense.

Notice that, in particular,
\begin{equation}
\left[  n\right]  _{1}\allowbreak=\allowbreak n,~\left[  n\right]
_{1}!\allowbreak=\allowbreak n!,~%
%TCIMACRO{\QATOPD{[}{]}{n}{k}}%
%BeginExpansion
\genfrac{[}{]}{0pt}{}{n}{k}%
%EndExpansion
_{1}\allowbreak=\allowbreak\binom{n}{k},~(a)_{1}\allowbreak=\allowbreak
1-a,~\left(  a|1\right)  _{n}\allowbreak=\allowbreak\left(  1-a\right)  ^{n}
\label{q1}%
\end{equation}
and
\begin{equation}
\left[  n\right]  _{0}\allowbreak=\allowbreak\left\{
\begin{array}
[c]{ccc}%
1 & \text{if} & n\geq1\\
0 & \text{if} & n=0
\end{array}
\right.  ,~\left[  n\right]  _{0}!\allowbreak=\allowbreak1,~%
%TCIMACRO{\QATOPD{[}{]}{n}{k}}%
%BeginExpansion
\genfrac{[}{]}{0pt}{}{n}{k}%
%EndExpansion
_{0}\allowbreak=\allowbreak1,~\left(  a|0\right)  _{n}\allowbreak
=\allowbreak\left\{
\begin{array}
[c]{ccc}%
1 & \text{if} & n=0\\
1-a & \text{if} & n\geq1
\end{array}
\right.  . \label{q2}%
\end{equation}

$i$ will denote imaginary unit, unless otherwise clearly stated. Let us define also:%

\begin{align}
\left(  ae^{i\theta},ae^{-i\theta}\right)  _{\infty}  &  =\prod_{k=0}^{\infty
}v\left(  x|aq^{k}\right)  \text{,}\label{rozklv}\\
\left(  te^{i\left(  \theta+\phi\right)  },te^{i\left(  \theta-\phi\right)
},te^{-i\left(  \theta-\phi\right)  },te^{-i\left(  \theta+\phi\right)
}\right)  _{\infty}  &  =\prod_{k=0}^{\infty}w\left(  x,y|tq^{k}\right)
\text{,}\label{rozklw}\\
\left(  ae^{2i\theta},ae^{-2i\theta}\right)  _{\infty}  &  =\prod
_{k=0}^{\infty}l\left(  x|aq^{k}\right)  \text{,} \label{rozkll}%
\end{align}
where,
\begin{align}
v(x|a)\allowbreak &  =\allowbreak1-2ax+a^{2},\label{vxa}\\
l(x|a)  &  =(1+a)^{2}-4x^{2}a,\label{lsa}\\
w(x,y|a)  &  =(1-a^{2})^{2}-4xya(1+a^{2})+4a^{2}(x^{2}+y^{2}) \label{wxya}%
\end{align}
and, as usually in the $q-$series theory, $x\allowbreak=\allowbreak\cos\theta$
and $y=\cos\phi$.

We will use through the paper the following way of obtaining infinite
expansions of the type
\[
\sum_{j\geq0}a_{n}p_{n}(x),
\]
that are convergent almost everywhere on some subset of $\mathbb{R}$. Namely,
in sight of \cite{Szab4} the following setting is considered. Say we have two
measures on $\mathbb{R}$ both having densities $f$ and $g$. Suppose that, we
know that $\int(f/g)^{2}gdx$ is finite. Further suppose also that we know two
families of orthogonal polynomials $\left\{  \alpha_{n}\right\}  $ and
$\left\{  \beta_{n}\right\}  $, such that the first one is orthogonal with
respect to the measure having the density $f$ and the other is orthogonal with
respect to the measure having the density $g$. Then we know that $f/g$ can be
expanded in an infinite series
\[
\sum_{n\geq0}a_{n}\beta_{n}(x),
\]
that is convergent in $L^{2}(\mathbb{R},g)$. We know in particular that
$\sum_{n\geq0}\left\vert a_{n}\right\vert ^{2}<\infty$. If additionally
\[
\sum_{n\geq0}\left\vert a_{n}\right\vert ^{2}\log^{2}(n+1)<\infty,
\]
then by the Rademacher--Meshov theorem, we deduce that the series in question
converges not only in $L^{2},$ but also almost everywhere with respect to $g$.

The point is that, in the sequel, all considered densities will be supported
only on the segment $[-1,1],$ for the proper values of some additional
parameters all these densities are bounded and hence their ratios will be
square integrable. Thus we will get the condition $\sum_{n\geq0}\left\vert
a_{n}\right\vert ^{2}<\infty$ satisfied for free. Moreover, in all cases we
will have $\left\vert a_{n}\right\vert ^{2}\leq r^{n}$ for some $r<1$. Hence
the condition $\sum_{n\geq0}\left\vert a_{n}\right\vert ^{2}\log
^{2}(n+1)<\infty$ is also naturally satisfied. If one knows the connection
coefficients between the families $\left\{  \alpha_{n}\right\}  $ and
$\left\{  \beta_{n}\right\}  ,$ i.e.,a set of coefficients $\left\{
c_{k,n}\right\}  _{n\geq1,0\leq k\leq n}$ satisfying%
\[
\beta_{n}\allowbreak(x)=\allowbreak\sum_{k=0}^{n}c_{k,n}\alpha_{k}(x),
\]
then $a_{n}\allowbreak=\allowbreak c_{0,n}/\int\beta_{n}^{2}gdx$.

We will refer to this type of reasoning as D(ensity) E(expansion) I(idea)
(*,*) (that is DEI(*,*)), where the stars point out to the connection
coefficient formulae and the formula for $\int\beta_{n}^{2}gdx$.

Apart from these expansions that lead to new, so to say, generating functions
of some families of polynomials we will consider, together with some families
of orthogonal polynomials say, $\left\{  p_{n}(x|q)\right\}  $ depending on
the parameter $\left\vert q\right\vert <1$, another auxiliary family is
defined by:
\[
\hat{p}_{n}(x|q)\allowbreak=\allowbreak\left\{
\begin{array}
[c]{ccc}%
(-1)^{n}q^{\binom{n}{2}}p_{n}(x|q^{-1}) & \text{if} & q\neq0\\
(-1)^{n}\lim_{q\rightarrow0}q^{\binom{n}{2}}p_{n}(x|q^{-1}) & \text{if} & q=0
\end{array}
\right.  .
\]
These polynomials also satisfy a three-term recurrence related to the one
satisfied by the polynomials $\left\{  p_{n}\right\}  $ however, they are not
orthogonal with respect to a nonnegative measure. They will play an important,
auxiliary r\^{o}le. It will turn out, that they are equal to some finite sums,
whose values and properties can be easily found and examined due to the simple
form of a three-term recurrence they satisfy. We will try to stick to the
convention, that these auxiliary polynomials will be denoted by a "hut" above
the name of the orthogonal polynomials, with a few exceptions where the names
of these auxiliary polynomials are already set.

The following convention will help in orderly listing of the properties of the
discussed families of polynomials. Namely, the family of polynomials whose
names start with say a letter $A$ will be referred to as $A$ (similarly for
the lower case $a)$.

Let us also define the following sets of polynomials and present their
generating functions and measures with respect to which these polynomials are
orthogonal if these measures are positive.

\section{Important auxiliary families of orthogonal polynomials\label{pom}}

\subsection{Hermite polynomials}

The Hermite polynomials are defined by the following three-term recurrence
(\ref{_1}), below:
\begin{equation}
xH_{n}\left(  x\right)  =H_{n+1}\left(  x\right)  +nH_{n-1}, \label{_1}%
\end{equation}
with $H_{0}\left(  x\right)  \allowbreak=\allowbreak H_{1}\left(  x\right)
\allowbreak=\allowbreak1$. They slightly differ from the Hermite polynomials
$h_{n}$ considered in most of the books on special functions. Namely,
\[
2xh_{n}\left(  x\right)  =h_{n+1}(x)+2nh_{n-1}(x),
\]
with $h_{-1}\left(  x\right)  =0,$ $h_{0}\left(  x\right)  \allowbreak
=\allowbreak1$.

It is known that polynomials $\left\{  h_{n}\right\}  $ are orthogonal with
respect to $\exp\left(  -x^{2}\right)  $ while the polynomials $\left\{
H_{n}\right\}  $ with respect to $\exp\left(  -x^{2}/2\right)  $. Moreover
$H_{n}\left(  x\right)  =h_{n}\left(  x/\sqrt{2}\right)  /\left(  \sqrt
{2}\right)  ^{n}$. Besides, we have (see \cite{Andrews1999} 6.1.7)
\begin{align}
\exp\left(  xt-t^{2}/2\right)  \allowbreak &  =\allowbreak\sum_{k\geq0}%
\frac{t^{k}}{k!}H_{k}\left(  x\right)  ,\label{_gH}\\
\exp\left(  2xt-t^{2}\right)  \allowbreak &  =\allowbreak\sum_{k\geq0}%
\frac{t^{k}}{k!}h_{k}\left(  x\right)  . \label{_gh}%
\end{align}

\subsection{Chebyshev polynomials}

They are of two kinds. The Chebyshev polynomials of the first kind $\left\{
T_{n}\right\}  _{n\geq-1}$ as well as Chebyshev polynomials of the second kind
$\left\{  U_{n}\right\}  _{n\geq-1}$ are defined by the same following
three-term recursion%
\begin{equation}
2xT_{n}\left(  x\right)  \allowbreak=\allowbreak T_{n+1}\left(  x\right)
+T_{n-1}\left(  x\right)  , \label{czeb}%
\end{equation}
for $n\geq1,$ with different, however, initial conditions$:$ namely with
$T_{-1}(x)=x,$ $U_{-1}(x)=0,$ $T_{0}\left(  x\right)  \allowbreak=\allowbreak
U_{0}(x)\allowbreak=\allowbreak1$.

One can define also these two families of polynomials in the following way for
$n\geq0$:%
\begin{align}
T_{n}\left(  \cos\theta\right)   &  =\cos\left(  n\theta\right)
,\label{repCh}\\
U_{n}\left(  \cos\theta\right)  \allowbreak &  =\allowbreak\frac{\sin\left(
n+1\right)  \theta}{\sin\theta}. \label{repCh2}%
\end{align}

We have (see e.g., \cite{IA} or \cite{Andrews1999})
\begin{align*}
\int_{-1}^{1}T_{n}\left(  x\right)  T_{m}\left(  x\right)  \frac{dx}{\pi
\sqrt{1-x^{2}}}  &  =\left\{
\begin{array}
[c]{ccc}%
1 & \text{if} & m=n=0\\
1/2 & \text{if} & m=n\neq0\\
0 & \text{if} & m\neq n
\end{array}
\right.  ,\\
\int_{-1}^{1}U_{n}\left(  x\right)  U_{m}\left(  x\right)  \frac
{2\sqrt{1-x^{2}}}{\pi}dx  &  =\left\{
\begin{array}
[c]{ccc}%
1 & \text{if} & m=n\\
0 & \text{if} & m\neq n
\end{array}
\right.  ,
\end{align*}
and for $\left\vert t\right\vert \leq1$
\begin{align}
\sum_{k=0}^{\infty}t^{k}T_{k}\left(  x\right)   &  =\frac{1-tx}{v(x|t)}%
,\label{_gT}\\
\sum_{k=0}^{\infty}t^{k}U_{k}\left(  x\right)   &  =\frac{1}{v(x|t)},
\label{_gU}%
\end{align}
where $v(x|t)$ is given by (\ref{vxa}). It has to be remarked that the
Chebyshev polynomials of the second kind, i.e.,polynomials $\left\{
U_{n}\right\}  _{n\geq-1}$ are related to polynomials $\left\{  h_{n}%
(x|q)\right\}  _{n\geq-1}$ defined below and play an extremely important
r\^{o}le in the so-called free probability, an intensively developing theory
that was basically first presented in papers of Voiculescu \cite{Voi} and
\cite{Voi2}.

\subsection{Rogers-Szeg\"{o} polynomials}

These polynomials are defined by the equality:
\[
s_{n}\left(  x|q\right)  \allowbreak=\allowbreak\sum_{k=0}^{n}%
%TCIMACRO{\QATOPD{[}{]}{n}{k}}%
%BeginExpansion
\genfrac{[}{]}{0pt}{}{n}{k}%
%EndExpansion
_{q}x^{k},
\]
for $n\geq0$ and $s_{-1}\left(  x|q\right)  \allowbreak=\allowbreak0$. They
will be playing here an important, auxiliary r\^{o}le. One can easily notice
(by changing the order of summation and using (\ref{binT})) that for all
$\left\vert t\right\vert <1,$ $q\in\lbrack-1,1)$ we have:%
\begin{equation}
\sum_{k\geq0}\frac{t^{k}}{(q)_{k}}s_{k}(x|q)=\frac{1}{(tx)_{\infty}%
(t)_{\infty}}. \label{gen_R-S}%
\end{equation}

In the sequel, the following identities discovered by Carlitz (see Exercise
12.2(b) and 12.2(c) of \cite{IA}), are true for $\left\vert q\right\vert
,\left\vert t\right\vert <1$ of which the first one is trivial in the face of
(\ref{gen_R-S}):
\begin{equation}
\sum_{k=0}^{\infty}\frac{s_{k}\left(  1|q\right)  t^{k}}{\left(  q\right)
_{k}}\allowbreak=\allowbreak\frac{1}{\left(  t\right)  _{\infty}^{2}}%
,~~\sum_{k=0}^{\infty}\frac{s_{k}^{2}\left(  1|q\right)  t^{k}}{\left(
q\right)  _{k}}\allowbreak=\allowbreak\frac{\left(  t^{2}\right)  _{\infty}%
}{\left(  t\right)  _{\infty}^{4}}, \label{Car_id}%
\end{equation}
which will allow us to show the convergence of many series considered in the
sequel. Basically, these polynomials appeared for the first time in the three
old papers of Rogers \cite{Rogers1}, \cite{Rogers2} and \cite{Rogers3}. One
has to add that many useful and complicated formulae involving
Rogers-Szeg\"{o} polynomials were presented in works of Carlitz
\cite{Carlitz56}, \cite{Carlitz57}, \cite{Carlitz72}. These polynomials were
generalized further, like, for example, \cite{Chen08}.

\part{Askey-Wilson scheme supported on $[-1,1]\label{sup[-1,1]}$}

\section{$0-$Parameter families}

\subsection{$q$-Hermite polynomials}

\subsubsection{3 term recurrence}

The $q-$Hermite polynomials are defined by the following three-term
recurrence:
\begin{equation}
2xh_{n}(x|q)=h_{n+1}(x|q)+(1-q^{n})h_{n-1}(x|q), \label{q-cont}%
\end{equation}
for $n\geq1$ with $h_{-1}(x|q)=0,$ $h_{0}(x|q)=1$.

In particular, one shows (see e.g., \cite{IA}(13.1.7)) that:
\begin{equation}
h_{n}\left(  x|q\right)  \allowbreak=\allowbreak e^{in\theta}s_{n}\left(
e^{-2i\theta}|q\right)  , \label{cH}%
\end{equation}
where $x\allowbreak=\allowbreak\cos\theta,$ and that:
\[
\sup_{\left\vert x\right\vert \leq1}\left\vert h_{n}\left(  x|q\right)
\right\vert \leq s_{n}\left(  1|q\right)  .
\]

The polynomials $h_{n}$ are often called the continuous $q-$Hermite
polynomials. Since the terminology is not fixed, we will use the name
$q-$Hermite polynomials for brevity.

\subsubsection{Orthogonality relation}

We have%

\begin{equation}
\int_{-1}^{1}h_{n}\left(  x|q\right)  h_{m}\left(  x|q\right)  f_{h}\left(
x|q\right)  dx=\left\{
\begin{array}
[c]{ccc}%
\left(  q\right)  _{n} & \text{if} & m=n\\
0 & \text{if} & m\neq n
\end{array}
\right.  , \label{inth^2}%
\end{equation}
where we denoted
\begin{equation}
f_{h}\left(  x|q\right)  =\frac{2\left(  q\right)  _{\infty}\sqrt{1-x^{2}}%
}{\pi}\prod_{k=1}^{\infty}l\left(  x|q^{k}\right)  , \label{fh}%
\end{equation}
and where $l\left(  x|a\right)  \allowbreak$ is given by (\ref{lsa}). It is
worth to mention, that $f_{h}$ can be expanded in an infinite series. Namely,
we have%
\begin{equation}
f_{h}(x|q)\allowbreak=\allowbreak\frac{2\sqrt{1-x^{2}}}{\pi}\sum_{k=0}%
^{\infty}(-1)^{k}q^{\binom{k+1}{2}}U_{2k}(x), \label{exp-fh}%
\end{equation}
that was proved in \cite{Szab-qGauss}, basically for the linearly transformed
$x,$ but that can be easily brought in the above mentioned form.

\subsubsection{Generating function}

Following (\ref{cH}) and (\ref{gen_R-S}), one can easily derive the following
formula for the generating function of $q-$Hermite polynomials (see also
\cite{KLS}(14.26.11))%
\begin{equation}
\sum_{j=0}^{\infty}\frac{t^{j}}{\left(  q\right)  _{j}}h_{j}\left(
x|q\right)  =\frac{1}{\prod_{k=0}^{\infty}v\left(  x|tq^{k}\right)  },
\label{_gqh}%
\end{equation}
where $v(x|t)\allowbreak=\allowbreak1-2tx+t^{2}$ and the series is convergent
for $\left\vert x|\leq1,|t\right\vert <1$. In the sequel, we will use denote
the right hand side of (\ref{_gqh}) by $\varphi_{h}(x|t,q)$. That is, we will
denote%
\[
\varphi_{h}(x|t,q)=\frac{1}{\prod_{k=0}^{\infty}v\left(  x|tq^{k}\right)  }.
\]
In fact, one can give two more generating functions for these polynomials.
They are given by the formulae (3.26.12) and (3.26.13) in \cite{KLS-arX} ) and
are less legible. In order to finish the topic of generating functions of
$q-$Hermite polynomials let us mention the following two.

Using DEI((\ref{Cnah}),(\ref{inth^2})), we get for all $\left\vert
q\right\vert ,\left\vert \beta\right\vert ,\left\vert x\right\vert <1,$%
\begin{equation}
\sum_{k\geq0}\frac{\beta^{k}}{(q)_{k}(\beta)_{k+1}}h_{2k}(x|q)=\frac
{(\beta^{2})_{\infty}}{(\beta)_{\infty}^{2}\prod_{j=0}^{\infty}l(x|\beta
q^{j})}. \label{genh2}%
\end{equation}
Similarly, basing on (\ref{hhQ}) and the ideas of \cite{Szab4}, we get for
$\max(\left\vert x\right\vert ,\left\vert a\right\vert ,\left\vert
b\right\vert ,\left\vert q\right\vert )<1$:%
\[
(ab)_{\infty}\varphi_{h}(x|a,q)\varphi_{h}(x|b,q)\allowbreak=\allowbreak
\sum_{j\geq0}\frac{(ab)^{j/2}}{(q)_{j}}h_{j}\left(  \frac{(a+b)}{2(ab)^{1/2}%
}|q\right)  h_{j}(x|q),
\]
which is nothing else but the famous Poisson-Mehler expansion formula.

\subsubsection{Connection coefficients, linearization formulae and other
finite expansion formulae}

One has to mention one other important formula, so-called "change of base
formula" relating polynomials $h_{n}(x|p)$ and $h_{n}(x|q)$. Namely, we have%
\begin{align}
h_{n}(x|p)  &  =\sum_{k=0}^{\left\lfloor n/2\right\rfloor }h_{n-2k}%
(x|q)\sum_{j=0}^{k}\left(  -1\right)  ^{j}p^{k-j}q^{j\left(  j+1\right)  /2}%
%TCIMACRO{\QATOPD{[}{]}{n-2k+j}{j}}%
%BeginExpansion
\genfrac{[}{]}{0pt}{}{n-2k+j}{j}%
%EndExpansion
_{q}\allowbreak\label{chbase}\\
&  \times\left(
%TCIMACRO{\QATOPD{[}{]}{n}{k-j}}%
%BeginExpansion
\genfrac{[}{]}{0pt}{}{n}{k-j}%
%EndExpansion
_{p}-p^{n-2k+2j+1}%
%TCIMACRO{\QATOPD{[}{]}{n}{k-j-1}}%
%BeginExpansion
\genfrac{[}{]}{0pt}{}{n}{k-j-1}%
%EndExpansion
_{p}\right)  ,\nonumber
\end{align}
that was proved say in \cite{IS03}, \cite{bressoud} or \cite{GIS99} (formula 7.2).

Remembering that $h_{n}(x|0)\allowbreak=\allowbreak U_{n}(x)$ and then setting
$p\allowbreak=\allowbreak0$ and then, later, setting $p\allowbreak=\allowbreak
q$ and $q\allowbreak=\allowbreak0$ in \ref{chbase}, we arrive at the following
two formulae:%

\begin{gather}
U_{n}\left(  x\right)  =\sum_{j=0}^{\left\lfloor n/2\right\rfloor }\left(
-1\right)  ^{j}q^{j\left(  j+1\right)  /2}%
%TCIMACRO{\QATOPD{[}{]}{n-j}{j}}%
%BeginExpansion
\genfrac{[}{]}{0pt}{}{n-j}{j}%
%EndExpansion
_{q}h_{n-2j}\left(  x|q\right)  ,\label{Unah}\\
h_{n}\left(  y|q\right)  =\sum_{k=0}^{\left\lfloor n/2\right\rfloor }%
\frac{q^{k}-q^{n-k+1}}{1-q^{n-k+1}}%
%TCIMACRO{\QATOPD{[}{]}{n}{k}}%
%BeginExpansion
\genfrac{[}{]}{0pt}{}{n}{k}%
%EndExpansion
_{q}U_{n-2k}\left(  y\right)  . \label{hnaU}%
\end{gather}

We have also the following, very useful, the so-called 'linearization
formula', that is true for $n,m\geq0$ and all $x$ and $q$.%
\begin{equation}
h_{n}\left(  x|q\right)  h_{m}\left(  x|q\right)  =\sum_{j=0}^{\min\left(
n,m\right)  }%
%TCIMACRO{\QATOPD{[}{]}{m}{j}}%
%BeginExpansion
\genfrac{[}{]}{0pt}{}{m}{j}%
%EndExpansion
_{q}%
%TCIMACRO{\QATOPD{[}{]}{n}{j}}%
%BeginExpansion
\genfrac{[}{]}{0pt}{}{n}{j}%
%EndExpansion
_{q}(q)_{j}h_{n+m-2j}\left(  x|q\right)  . \label{lihh}%
\end{equation}
This formula can be extended, but is not as legible as the one above:%

\begin{gather}
h_{n}\left(  x|q\right)  h_{m}\left(  x|q\right)  h_{k}\left(  x|q\right)
=\label{linhhh}\\
\sum_{j=0}^{\left\lfloor (k+m+n)/2\right\rfloor }\left(  \sum_{r=\max
(j-k,0)}^{\min(m,n,m+n-j)}%
%TCIMACRO{\QATOPD{[}{]}{m}{r}}%
%BeginExpansion
\genfrac{[}{]}{0pt}{}{m}{r}%
%EndExpansion
_{q}%
%TCIMACRO{\QATOPD{[}{]}{n}{r}}%
%BeginExpansion
\genfrac{[}{]}{0pt}{}{n}{r}%
%EndExpansion
_{q}%
%TCIMACRO{\QATOPD{[}{]}{k}{j-r}}%
%BeginExpansion
\genfrac{[}{]}{0pt}{}{k}{j-r}%
%EndExpansion
_{q}%
%TCIMACRO{\QATOPD{[}{]}{m+n-2r}{j-r}}%
%BeginExpansion
\genfrac{[}{]}{0pt}{}{m+n-2r}{j-r}%
%EndExpansion
_{q}(q)_{r}(q)_{j-r}\right) \nonumber\\
\times h_{n+m+k-2j}\left(  x|q\right)  .\nonumber
\end{gather}

It is interesting, that one can easily get the generating function of the
numbers%
\[
A_{j,k,n,m}\allowbreak=\allowbreak\int_{-1}^{1}h_{j}(x)h_{k}(x)h_{n}%
(x)h_{m}(x)f_{h}(x)dx.
\]
Namely, following \cite{Szab4}(Remark 2.2) we have%
\[
\sum_{j,k,n,m=0}^{\infty}\frac{a^{j}b^{k}c^{n}d^{m}}{\left(  q\right)
_{j}\left(  q\right)  _{k}\left(  q\right)  _{n}\left(  q\right)  _{m}%
}A_{i,j,k,l}=\frac{\left(  abcd\right)  _{\infty}}{\left(
ab,ac,ad,bc,bd,cd\right)  _{\infty}}.
\]
Hence, setting $d\allowbreak=\allowbreak0$ in the above-mentioned formula, we
get%
\[
\sum_{j,k,n=0}^{\infty}\frac{a^{j}b^{k}c^{n}}{\left(  q\right)  _{j}\left(
q\right)  _{k}\left(  q\right)  _{n}}\int_{-1}^{1}h_{j}(x)h_{k}(x)h_{n}%
(x)f_{h}(x)dx=\frac{1}{\left(  ab,ac,bc\right)  _{\infty}}.
\]
It is worth to mention that the first important results on $q-$Hermite
appeared in \cite{ISV87}, \cite{IsMas94}, \cite{IRS99} and \cite{IsStan200}.

\subsection{$q^{-1}-$Hermite polynomials}

This is an auxiliary, however important family of polynomials. It is related
to the family of $q$-Hermite polynomials by the formula:
\begin{subequations}
\begin{equation}
b_{n}\left(  x|q\right)  \allowbreak=\allowbreak\left\{
\begin{array}
[c]{ccc}%
\left(  -1\right)  ^{n}q^{\binom{n}{2}}h_{n}\left(  x|q^{-1}\right)  . &
\text{if} & q\neq0\\
\left(  -1\right)  ^{n}\lim_{q->0}q^{\binom{n}{2}}h_{n}\left(  x|q^{-1}\right)
& \text{if} & q=0
\end{array}
\right.  . \label{bbb}%
\end{equation}

\subsubsection{3 term recurrence}

It is easy to note that these polynomials satisfy the following three-term
recurrence:
\end{subequations}
\begin{equation}
b_{n+1}\left(  x|q\right)  =-2q^{n}xb_{n}\left(  y|q\right)  +q^{n-1}%
(1-q^{n})b_{n-1}\left(  x|q\right)  , \label{maleB}%
\end{equation}
with $b_{-1}\left(  x|q\right)  \allowbreak=\allowbreak0,$ $b_{0}\left(
x|q\right)  \allowbreak=\allowbreak1$. From Favard's theorem it follows that
polynomials defined by (\ref{maleB}) cannot be orthogonal with respect to a
nonnegative measure for $\left\vert q\right\vert <1$.

Notice that $b_{-1}(x|0)\allowbreak=\allowbreak0$, $b_{0}(x|0)\allowbreak
=\allowbreak b_{2}(x|0)\allowbreak=\allowbreak1$, $b_{1}(x|0)\allowbreak
=\allowbreak-2x$, $b_{n}(x|0)\allowbreak=\allowbreak0$ for $n\geq3$.

Polynomials $\left\{  b_{n}\right\}  $ will help us to define new families of
orthogonal polynomials.

\subsubsection{Generating function}

To see the nice r\^{o}le of these polynomials let us recall that recently (see
\cite{SzabCheb} Lemma 2(a)) it was shown that for $\left\vert t\right\vert
,\left\vert q\right\vert <1:$%
\[
1/\varphi_{h}(x|t,q)=\prod_{k=0}^{\infty}v\left(  x|tq^{k}\right)
=\sum_{n\geq0}\frac{t^{n}}{(q)_{n}}b_{n}(x|q).
\]

\subsection{Connection coefficients, linearization formulae and finite
expansion formulae involving $q$-Hermite and $q^{-1}$-Hermite polynomials}

To appreciate the importance of polynomials $\left\{  b_{n}\right\}  $, let us
recall several formulae that expose the interrelation between them and the
$q$-Hermite polynomials. They are related by the following formula:
\begin{equation}
b_{n}\left(  x|q\right)  \allowbreak=\allowbreak\left(  -1\right)
^{n}q^{\binom{n}{2}}\sum_{k=0}^{\left\lfloor n/2\right\rfloor }%
%TCIMACRO{\QATOPD{[}{]}{n}{k}}%
%BeginExpansion
\genfrac{[}{]}{0pt}{}{n}{k}%
%EndExpansion
_{q}%
%TCIMACRO{\QATOPD{[}{]}{n-k}{k}}%
%BeginExpansion
\genfrac{[}{]}{0pt}{}{n-k}{k}%
%EndExpansion
_{q}(q)_{k}q^{k(k-n)}h_{n-2k}\left(  x|q\right)  , \label{bnah}%
\end{equation}
that was proved in \cite{Szab6} Lemma 2 assertion i. Further, for all
$n,m\geq0$ we have
\begin{equation}
\sum_{k=0}^{n}%
%TCIMACRO{\QATOPD{[}{]}{n}{k}}%
%BeginExpansion
\genfrac{[}{]}{0pt}{}{n}{k}%
%EndExpansion
_{q}b_{n-k}\left(  x|q\right)  h_{k+m}\left(  x|q\right)  \allowbreak=\left\{
\begin{array}
[c]{ccc}%
0 & \text{if} & n>m\\
(-1)^{n}q^{\binom{n}{2}}\frac{(q)_{m}}{(q)_{m-n}}h_{m-n}\left(  x|q\right)  &
\text{if} & n\leq m
\end{array}
\right.  , \label{simpbh}%
\end{equation}
that was proved also in \cite{Szab6} Lemma 2 assertion iii.

We also have for all $\forall n,m\geq0$ (in \cite{Szab6} Lemma 2 assertion
ii)
\begin{equation}
h_{m}\left(  x|q\right)  b_{n}\left(  x|q\right)  =(-1)^{n}q^{\binom{n}{2}%
}\sum_{k=0}^{\left\lfloor (n+m)/2\right\rfloor }%
%TCIMACRO{\QATOPD{[}{]}{n}{k}}%
%BeginExpansion
\genfrac{[}{]}{0pt}{}{n}{k}%
%EndExpansion
_{q}%
%TCIMACRO{\QATOPD{[}{]}{n+m-k}{k}}%
%BeginExpansion
\genfrac{[}{]}{0pt}{}{n+m-k}{k}%
%EndExpansion
_{q}(q)_{k}q^{-k(n-k)}h_{n+m-2k}\left(  x|q\right)  . \label{linbh}%
\end{equation}

\section{$1$-Parameter families}

\subsection{Big continuous $q-$Hermite polynomials}

\subsubsection{3 term recurrence}

This family of polynomials satisfies the following three-term recurrence:%
\begin{equation}
(2x-aq^{n})h_{n}\left(  x|a,q\right)  =h_{n+1}\left(  x|a,q\right)
+(1-q^{n})h_{n-1}\left(  x|a,q\right)  , \label{big}%
\end{equation}
with $h_{-1}(x|a,q)\allowbreak=\allowbreak0$ and $h_{0}(x|a,q)\allowbreak
=\allowbreak1$.

\subsubsection{Orthogonality relation}

As far as orthogonality relations are concerned, one has to distinguish two
cases $\left\vert a\right\vert <1$ and and $a>1$. In the first situation,
there exists an absolutely continuous measure with the density equal to%
\begin{equation}
f_{bh}(x|a,q)=f_{h}(x|q)\varphi_{h}(x|a,q), \label{ffbh}%
\end{equation}
where $f_{h}$ is given by (\ref{fh}). The fact that then $f_{bh}$ integrates
to $1$ is obvious since we have expansion (\ref{_gqh}). In case of $a>1$ we
have to assume that $q>0$ and then the orthogonalizing measure has the density
given by (\ref{ffbh}) but it has also $\#\left\{  k:1<aq^{k}<a\right\}  $
atoms at the points
\begin{equation}
x_{k}\allowbreak=\allowbreak(aq^{k}+a^{-1}q^{-k})/2, \label{atoms}%
\end{equation}
with the weights
\[
\hat{w}_{k}\allowbreak=\allowbreak\frac{(1-a^{2}q^{2k})(a^{-2})_{\infty}%
(a^{2})_{k}}{(1-a^{2})(q)_{k}}q^{-(3k^{2}+k)/2}(\frac{-1}{a^{4}})^{k}.
\]

For $\left\vert a\right\vert <1$ we have%

\begin{equation}
\int_{-1}^{1}h_{n}\left(  x|a,q\right)  h_{m}\left(  x|a,q\right)
f_{bh}\left(  x|q\right)  dx=\left\{
\begin{array}
[c]{ccc}%
\left(  q\right)  _{n} & \text{if} & m=n\\
0 & \text{if} & m\neq n
\end{array}
\right.  . \label{ortbh}%
\end{equation}

\begin{remark}
Notice that for $\left\vert a\right\vert <1,$ $f_{bh}$ can be expanded in the
following way:%
\[
f_{bh}(x|a,q)=f_{h}(x|q)\sum_{k\geq0}\frac{a^{k}}{(q)_{k}}h_{k}(x|q).
\]
This will be later generalized in two ways.
\end{remark}

\subsubsection{Connection coefficients, linearization formulae and finite
expansion formulae}

One can easily show (by calculating the generating function and comparing it
with (\ref{gfha}), below, and then applying (\ref{rozklv})) that%
\[
h_{n}\left(  x|a,q\right)  =\sum_{k=0}^{n}%
%TCIMACRO{\QATOPD{[}{]}{n}{k}}%
%BeginExpansion
\genfrac{[}{]}{0pt}{}{n}{k}%
%EndExpansion
_{q}\left(  ae^{i\theta}\right)  _{k}e^{i\left(  n-2k\right)  \theta},
\]
where, as usual, $x\allowbreak=\allowbreak\cos\theta$. From this equality
follow immediately the following two formulae:%
\begin{align}
h_{n}\left(  x|a,q\right)  \allowbreak &  =\allowbreak\sum_{k=0}^{n}%
%TCIMACRO{\QATOPD{[}{]}{n}{k}}%
%BeginExpansion
\genfrac{[}{]}{0pt}{}{n}{k}%
%EndExpansion
_{q}(-1)^{k}q^{\binom{k}{2}}a^{k}h_{n-k}\left(  x|q\right)  ,\label{bigh}\\
h_{n}(x|q)  &  =\sum_{k=0}^{n}%
%TCIMACRO{\QATOPD{[}{]}{n}{k}}%
%BeginExpansion
\genfrac{[}{]}{0pt}{}{n}{k}%
%EndExpansion
_{q}a^{k}h_{n-k}(x|a,q). \label{hnabh}%
\end{align}

For more connection coefficients see Section \ref{rem}, below.

\subsubsection{Generating function}

We have (see e.g., \cite{KLS}(14.18.13)) that:%
\begin{equation}
\sum_{j=0}^{\infty}\frac{t^{j}}{\left(  q\right)  _{j}}h_{j}\left(
x|a,q\right)  =\frac{\left(  at\right)  _{\infty}}{\prod_{k=0}^{\infty
}v\left(  x|tq^{k}\right)  }=(at)_{\infty}\varphi_{h}(x|t,q), \label{gfha}%
\end{equation}
for $\left\vert at\right\vert <1$ and $\varphi_{h}$ denotes as before the
right-hand side of (\ref{_gqh}). Again, one can give two more formulae for
slightly differently defined generating functions of these polynomials, namely
(3.18.14) and (3.18.15) in \cite{KLS-arX}.

Again using DEI((\ref{hnabh}),(\ref{ortbh})) we end up with the following
expansion:%
\begin{equation}
1/\varphi_{h}(x|a,q)=\prod_{j=0}^{\infty}v(x|aq^{j})\allowbreak=\allowbreak
\sum_{k=0}^{\infty}q^{\binom{k}{2}}\frac{(-a)^{k}}{(q)_{k}}h_{n}(x|a,q).
\label{genbh2}%
\end{equation}

\begin{remark}
Notice that for $q\allowbreak=\allowbreak0$ (\ref{big}) is identical with
(\ref{czeb}) for $n\geq1$. Besides, we have $h_{n}(x|a,0)\allowbreak
=\allowbreak2x-a\allowbreak=\allowbreak U_{1}(x)\allowbreak-\allowbreak
aU_{0}(x),$ hence we deduce that for $n\geq-1$ we have%
\[
h_{n}(x|a,0)=U_{n}(x)-aU_{n-1}(x).
\]

\end{remark}

\subsection{Big continuous $q^{-1}-$Hermite polynomials}

Like in the case of $q-$Hermite polynomials, we define an auxiliary family of
polynomials $\left\{  \hat{h}_{n}\right\}  _{n\geq-1}$ that we will call big
continuous $q^{-1}-$Hermite polynomials. They are defined by the relationship%
\[
\hat{h}_{n}(x|a,q)=(-1)^{n}q^{\binom{n}{2}}h_{n}(x|a,q^{-1}),
\]
for all $n\geq-1$.

\subsubsection{Three term recurrence}

It is elementary to check that they satisfy the following three-term
recurrence:%
\[
\hat{h}_{n+1}(x|a,q)=(2xq^{n}-a)\hat{h}_{n}(x|a,q)-q^{n-1}(1-q^{n})\hat
{h}_{n-1}(x|a,q),
\]
with $\hat{h}_{-1}(x|a,q)\allowbreak=\allowbreak0$ and $\hat{h}_{0}%
(x|a,q)\allowbreak=\allowbreak1$. They are obviously not orthogonal as it
follows from Favard's theorem. However, they play an important, auxiliary
r\^{o}le, similar to the r\^{o}le played by the polynomials $\left\{
b_{n}\right\}  $. Namely, for $n\geq0$ we have
\begin{gather}
\hat{h}_{n}(x|a,q)=\sum_{k=0}^{n}%
%TCIMACRO{\QATOPD{[}{]}{n}{k}}%
%BeginExpansion
\genfrac{[}{]}{0pt}{}{n}{k}%
%EndExpansion
_{q}a^{k}b_{n-k}(x|q),\label{Id1}\\
\sum_{j=0}^{n}%
%TCIMACRO{\QATOPD{[}{]}{n}{j}}%
%BeginExpansion
\genfrac{[}{]}{0pt}{}{n}{j}%
%EndExpansion
_{q}\hat{h}_{j}(x|a,q)h_{n-j}(x|a,q)=0. \label{id3}%
\end{gather}

(\ref{Id1}) follows directly from the definition of polynomials $b_{n}$ and
(\ref{bigh}). (\ref{id3}) follows from (\ref{Id1}), (\ref{hnabh}), change the
order of summation and finally from (\ref{simpbh}).

\begin{remark}
It is easy to note that $\hat{h}_{1}(x|a,0)\allowbreak=\allowbreak a-2x,$
$\hat{h}_{n}(x|a,0)\allowbreak=\allowbreak a^{n-2}v(x|a),$ for $n\geq2$.
\end{remark}

\subsubsection{Generating function}

We have for $\left\vert at\right\vert <1:$%

\[
1/\sum_{j=0}^{\infty}\frac{t^{j}}{\left(  q\right)  _{j}}h_{j}\left(
x|a,q\right)  =\frac{1}{(at)_{\infty}}\prod_{j=0}^{\infty}v(x|tq^{j}%
)=\sum_{k\geq0}\frac{t^{k}}{(q)_{k}}\hat{h}_{k}(x|a,q).
\]
Firstly we apply (\ref{id3}) and then change of the order of summation.

\subsection{Continuous $q-$utraspherical polynomials}

\subsubsection{3 term recurrence}

It turns out that the polynomials $\left\{  h_{n}\right\}  _{n\geq-1}$ are
also related to another family of orthogonal polynomials $\left\{
C_{n}\left(  x|\beta,q\right)  \right\}  _{n\geq-1},$ which was considered by
Rogers in 1894 (see \cite{Rogers2}, \cite{Rogers3}). Now, they are called the
continuous $q-$utraspherical polynomials. Polynomials $\left\{  C_{n}\right\}
$ can be defined through their $3-$recurrence (see \cite{KLS}(14.10.19)) in
the following way:%

\begin{equation}
2(1-\beta q^{n})xC_{n}(x|\beta,q)=(1-q^{n+1})C_{n+1}\left(  x|\beta,q\right)
\allowbreak+(1-\beta^{2}q^{n-1})C_{n-1}\left(  x|\beta,q\right)  , \label{3tC}%
\end{equation}
for $n\geq0,$ with $C_{-1}\left(  x|\beta,q\right)  \allowbreak=\allowbreak0,$
$C_{0}\left(  x|\beta,q\right)  \allowbreak=\allowbreak1,$ where $\beta$ is a
real parameter such that $\left\vert \beta\right\vert <1$. One shows (see
e.g., \cite{IA}(13.2.1)) that for $\left\vert q\right\vert ,\left\vert
\beta\right\vert <1,$ $\forall n\in\mathbb{N}:$
\begin{equation}
C_{n}\left(  x|\beta,q\right)  \allowbreak=\allowbreak\sum_{k=0}^{n}%
\frac{\left(  \beta\right)  _{k}\left(  \beta\right)  _{n-k}}{\left(
q\right)  _{k}\left(  q\right)  _{n-k}}e^{i\left(  n-2k\right)  \theta},
\label{cS}%
\end{equation}
where $x=\cos\theta$.

\subsubsection{Orthogonality relation}

We have (see e.g., \cite{IA}(13.2.4))
\begin{equation}
\int_{-1}^{1}C_{n}\left(  x|\beta,q\right)  C_{m}\left(  x|\beta,q\right)
f_{C}\left(  x|\beta,q\right)  dx=\left\{
\begin{array}
[c]{ccc}%
0 & \text{if} & m\neq n\\
\frac{\left(  \beta^{2}\right)  _{n}}{\left(  1-\beta q^{n}\right)  \left(
q\right)  _{n}} & \text{if} & m=n
\end{array}
\right.  , \label{Intb^2}%
\end{equation}
where
\begin{equation}
f_{C}(x|\beta,q)\allowbreak=\allowbreak\frac{(\beta^{2})_{\infty}}%
{(1-\beta)(\beta,\beta q)_{\infty}}f_{h}\left(  x|q\right)  /\prod
_{j=0}^{\infty}l\left(  x|\beta q^{j}\right)  , \label{fC}%
\end{equation}
with, as before, $l(x|a)\allowbreak=\allowbreak(1+a)^{2}-4x^{2}a$.

\subsubsection{Connection coefficients, linearization formulae and finite
expansion formulae}

We have the celebrated connection coefficient formula for the Rogers
polynomials see \cite{IA},(13.3.1).%

\begin{equation}
C_{n}\left(  x|\gamma,q\right)  \allowbreak=\allowbreak\sum_{k=0}%
^{\left\lfloor n/2\right\rfloor }\frac{\beta^{k}\left(  \gamma/\beta\right)
_{k}\left(  \gamma\right)  _{n-k}\left(  1-\beta q^{n-2k}\right)  }%
{(q)_{k}\left(  \beta q\right)  _{n-k}\left(  1-\beta\right)  }C_{n-2k}\left(
x|\beta,q\right)  . \label{CnaC}%
\end{equation}

As special cases (once we take $\gamma\allowbreak=\allowbreak\beta,$
$\beta\allowbreak=\allowbreak0$ and (\ref{Ch}) to get (\ref{Cnah}) and then
$\gamma\allowbreak=\allowbreak0$ and (\ref{Ch}) to get (\ref{hnaC}) ) we get
the following useful formulae true for $\left\vert \beta\right\vert
,\left\vert \gamma\right\vert <1$:
\begin{align}
C_{n}\left(  x|\beta,q\right)  \allowbreak &  =\allowbreak\sum_{k=0}%
^{\left\lfloor n/2\right\rfloor }(-\beta)^{k}\frac{q^{\binom{k}{2}}\left(
\beta\right)  _{n-k}}{(q)_{k}(q)_{n-2k}}h_{n-2k}\left(  x|q\right)
,\label{Cnah}\\
h_{n}\left(  x|q\right)  \allowbreak &  =\allowbreak\sum_{k=0}^{\left\lfloor
n/2\right\rfloor }\beta^{k}\frac{(q)_{n}\left(  1-\beta q^{n-2k}\right)
}{(q)_{k}(1-\beta)\left(  \beta q\right)  _{n-k}}C_{n-2k}\left(
x|\beta,q\right)  . \label{hnaC}%
\end{align}

\subsubsection{Generating function}

It is also known (see \cite{KLS-arX} (3.10.25)), that
\begin{equation}
\sum_{k=0}^{\infty}t^{k}C_{k}\left(  x|\beta,q\right)  =\prod_{k=0}^{\infty
}\frac{v\left(  x|\beta tq^{k}\right)  }{v\left(  x|tq^{k}\right)  }.
\label{fgC}%
\end{equation}
Again one can give several other (less legible) generating functions involving
these polynomials given by (3.10.26)-(3.10.31) in \cite{KLS-arX}.

\begin{remark}
We have (following formulae (\ref{cS}) and (\ref{cH})), that
\begin{equation}
C_{n}\left(  x|0,q\right)  \allowbreak=\allowbreak\frac{h_{n}\left(
x|q\right)  }{\left(  q\right)  _{n}}. \label{Ch}%
\end{equation}

\end{remark}

\begin{remark}
Notice that taking $\beta\allowbreak=\allowbreak q$ we get $C_{n}%
(x|q,q)\allowbreak=\allowbreak U_{n}(x)$. To see this put $\beta
\allowbreak=\allowbreak q$ in (\ref{3tC}) and cancel out by $q^{n+1}$.
\end{remark}

\begin{remark}
For $q\allowbreak=\allowbreak0$ we have (by (\ref{fgC})%
\[
\sum_{k=0}^{\infty}t^{k}C_{k}\left(  x|\beta,0\right)  =\frac{1-2\beta
tx+\beta^{2}t^{2}}{1-2tx+t^{2}}\text{,}%
\]
from which we deduce that for $n\geq1$ we have
\[
C_{n}(x|\beta,0)\allowbreak=\allowbreak(1-\beta)U_{n}(x)-\beta(1-\beta
)U_{n-2}(x).
\]

\end{remark}

\section{$2-$Parameter family}

\subsection{Al-Salam--Chihara polynomials}

\subsubsection{3 term recurrence}

In the literature connected with the special functions, the following
polynomials defined recursively function as the Al-Salam--Chihara (ASC)
polynomials:%
\begin{equation}
(2x-(a+b)q^{n})Q_{n}\left(  x|a,b,q\right)  =Q_{n+1}\left(  x|a,b,q\right)
\allowbreak+(1-abq^{n-1})(1-q^{n})Q_{n-1}(x|a,b,q), \label{AlSC1}%
\end{equation}
with $Q_{-1}\left(  x|a,b,q\right)  \allowbreak=\allowbreak0,$ $Q_{0}\left(
x|a,b,q\right)  \allowbreak=\allowbreak1$.

\begin{remark}
Notice that polynomials $Q_{n}$ depend only on $s=a+b$ and $p\allowbreak
=\allowbreak ab$. These are their real parameters. Notice also, that for any
set of values of $s$ and $p$ we get two sets of values of $a$ and $b$
(possibly complex). Further notice, that if $a+b$ and $ab$ are real then each
polynomial $Q_{n}$ has real coefficients. This, of course, may happen if
either both $a$ and $b$ are real or they form a conjugate pair.
\end{remark}

\begin{remark}
It follows from Favard's theorem (\cite{IA}), that if $\left\vert
ab\right\vert \leq1$, then there exists a positive measure with respect to
which polynomials $Q_{n}$ are orthogonal.
\end{remark}

Comparing (\ref{AlSC1}) with (\ref{big}) one can notice that when
$b\allowbreak=\allowbreak0$, then we deal with big $q$-Hermite polynomials.

Sometimes, it turns out to be useful to change parameters; in the case when
$ab\geq0$ we will sometimes introduce parameters $\rho$ and $y$ defined by the
following relationships:
\[
a+b\allowbreak=\allowbreak2\rho y,~~ab\allowbreak=\allowbreak\rho^{2}.
\]
Polynomials $Q_{n}$ with these parameters will be denoted by $p_{n}$. More
precisely, for $\rho$ and $y$ defined by the above mentioned relationship, we
define%
\begin{equation}
p_{n}(x|y,\rho,q)=Q_{n}(x|a,b,q). \label{pn}%
\end{equation}

\subsubsection{Orthogonal relations}

For details see \cite{KLS}(14.8.3). We have the following orthogonality
relationships (see \cite{KLS}(14.8.2)) satisfied for $\left\vert a\right\vert
,\left\vert b\right\vert <1$:%
\begin{equation}
\int_{-1}^{1}Q_{n}\left(  x|a,b,q\right)  Q_{m}\left(  x|a,b,q\right)
f_{Q}\left(  x|a,b,q\right)  dx=\left\{
\begin{array}
[c]{ccc}%
0 & \text{if} & n\neq m\\
\left(  q\right)  _{n}\left(  ab\right)  _{n} & \text{if} & m=m
\end{array}
\right.  , \label{QnQm}%
\end{equation}
where
\begin{equation}
f_{Q}\left(  x|a,b,q\right)  =(ab)_{\infty}f_{h}(x|h)\varphi_{h}%
(x|a,q)\varphi_{h}(x|b,q). \label{gASC}%
\end{equation}
As in the case of big $q-$Hermite polynomials, if one of the parameters $a$
and $b$ is greater than $1$, then the measure that makes ASC polynomials
orthogonal, has $\#\{k:1<aq^{k}<a\}$ atoms located at points $x_{k}$ defined
by (\ref{atoms}) with weights given by:%
\[
\hat{w}_{k}\allowbreak=\allowbreak\frac{(a^{-2})_{\infty}(1-a^{2}q^{2k}%
)(a^{2},ab)_{k}}{(b/a)_{\infty}(1-a^{2})(q,aq/b)_{k}}q^{-k^{2}}(\frac{1}%
{a^{3}b})^{k}.
\]

\begin{remark}
\label{simpl}Notice that following (\cite{Szabl-intAW}, (2.7) ) for
$\left\vert a\right\vert ,\left\vert b\right\vert <1$, the density $f_{Q}$ can
expanded in the following way:%
\begin{equation}
f_{Q}(x|a,b,q)=f_{h}\left(  x|q\right)  \sum_{j=0}^{\infty}\frac{S_{j}%
^{(2)}\left(  a,b\right)  }{\left(  q\right)  _{j}}h_{j}\left(  x|q\right)  ,
\label{mehler}%
\end{equation}
where $S_{j}^{(2)}\left(  a,b\right)  \allowbreak=\allowbreak\sum_{n=0}^{j}%
%TCIMACRO{\QATOPD{[}{]}{j}{n}}%
%BeginExpansion
\genfrac{[}{]}{0pt}{}{j}{n}%
%EndExpansion
_{q}a^{n}b^{j-n}$ and also (following \cite{Szab-bAW}, (2.13)) in the
following way:%
\[
f_{Q}(x|a,b,c)=f_{bh}(x|b,q)\sum_{j\geq0}\frac{a^{j}}{(q)_{j}}h_{j}(x|b,q).
\]
Notice also that
\begin{align*}
w\left(  x,\frac{a+b}{2(ab)^{1/2}}|(ab)^{1/2}\right)   &  =(1+a^{2}%
)(1+b^{2})-2x(a+b)(1+ab)+4x^{2}ab\\
&  =v(x|a)v(x|b),
\end{align*}
$w(x,y|a)$ is given by (\ref{wxya}). So the density $f_{Q}$ can be written as
\begin{equation}
f_{Q}(x|a,b,q)\allowbreak=\allowbreak f_{h}(x|q)\frac{(ab)_{\infty}}%
{\prod_{j=0}^{\infty}w\left(  x,\frac{a+b}{2(ab)^{1/2}}|(ab)^{1/2}%
q^{j}\right)  }. \label{n_form}%
\end{equation}
Again, this observation will be important in the next part dedicated mostly to
probabilistic interpretations.
\end{remark}

\subsubsection{Connection coefficients and other finite expansions}

In \cite{bms} (see Remark 1 following Theorem 1) there has been shown (for
some specific values of $a$ and $b$ but that can be easily extended to the
general case ) the following formula combining ASC, $q$-Hermite and $q^{-1}%
$-Hermite $\forall n\geq-1,x,a,b,q\in\mathbb{C}$:%
\begin{equation}
Q_{n}(x|a,b,q)=\sum_{k=0}^{n}%
%TCIMACRO{\QATOPD{[}{]}{n}{k}}%
%BeginExpansion
\genfrac{[}{]}{0pt}{}{n}{k}%
%EndExpansion
_{q}(ab)^{(n-k)/2}b_{n-k}\left(  \frac{(a+b)}{2\sqrt{ab}}|q\right)
h_{k}(x|q). \label{Qbh}%
\end{equation}

We have below, the other formula combining ASC and $q$-Hermite polynomials,
namely for $\forall n\geq-1,x,a,b,q$:%
\begin{equation}
h_{n}(x|q)=\sum_{k=0}^{n}%
%TCIMACRO{\QATOPD{[}{]}{n}{k}}%
%BeginExpansion
\genfrac{[}{]}{0pt}{}{n}{k}%
%EndExpansion
_{q}(ab)^{(n-k)/2}h_{n-k}\left(  \frac{(a+b)}{2\sqrt{ab}}|q\right)
Q_{k}(x|a,b,q). \label{hhQ}%
\end{equation}

As far as the relationships with big $q-$Hermite polynomials are concerned we
have for $n\geq0:$%
\[
Q_{n}(x|a,b,q)=\sum_{j}^{n}%
%TCIMACRO{\QATOPD{[}{]}{n}{j}}%
%BeginExpansion
\genfrac{[}{]}{0pt}{}{n}{j}%
%EndExpansion
_{q}(-1)^{j}b^{j}q^{\binom{j}{2}}h_{n-j}(x|a,q),
\]
which can be easily justified given the forms of generating functions of the
ASC and big $q-$Hermite polynomials as well as (\ref{obinT}).

We have also the following formula that combines ASC polynomials with
different parameters, namely for all $n\geq-1$ and all $x,a,b,c,d,q$ we have%
\begin{equation}
Q_{n}(x|a,b,q)=\sum_{j=0}^{n}%
%TCIMACRO{\QATOPD{[}{]}{n}{j}}%
%BeginExpansion
\genfrac{[}{]}{0pt}{}{n}{j}%
%EndExpansion
_{q}(cd)^{(n-j)/2}Q_{j}(x|c,d,q)Q_{n-j}\left(  \frac{c+d}{2(cd)^{1/2}}%
|\frac{a}{(cd)^{1/2}},\frac{b}{(cd)^{1/2}},q\right)  . \label{abcd}%
\end{equation}
That was proved in \cite{SzablKer} for some complex values of $a,b,c,d$, but
that can be easily extended to the above-mentioned form since on both sides of
this formula we deal with the rational function of polynomials in $x,a,b,c,d$.

Notice that if say $cd\geq0$, then defining parameters $y$ and $\rho$ by the
relationships $cd\allowbreak=\allowbreak\rho^{2}$ and $2\rho y\allowbreak
=\allowbreak c+d$ and introducing polynomials
\begin{equation}
p_{n}(x|y,\rho,q)=Q_{n}(x|c,d,q), \label{pby Q}%
\end{equation}
and denoting the density $f_{Q}$ with new parameters by $f_{p}$ i.e.,%
\[
f_{Q}(x|c,d,q)\allowbreak=\allowbreak f_{p}(x|y,\rho,q),
\]
the identity (\ref{abcd}) takes now the more friendly form for all $n\geq0$
and $\rho\neq0$:%
\[
Q_{n}(x|a,b,q)\allowbreak=\allowbreak\sum_{j=0}^{n}%
%TCIMACRO{\QATOPD{[}{]}{n}{j}}%
%BeginExpansion
\genfrac{[}{]}{0pt}{}{n}{j}%
%EndExpansion
_{q}\rho^{n-j}p_{j}(x|y,\rho,q)Q_{n-j}(y|a/\rho,b/\rho,q).
\]

Below, we present yet another formula combining two families of
ASC\ polynomials with so to say a reversed r\^{o}le of argument and parameters:%

\begin{gather}
\frac{Q_{n}(x|a,b,q)}{(ab)_{n}}=\sum_{j=0}^{n}%
%TCIMACRO{\QATOPD{[}{]}{n}{j}}%
%BeginExpansion
\genfrac{[}{]}{0pt}{}{n}{j}%
%EndExpansion
_{q}(-1)^{j}q^{\binom{j}{2}}(ab)^{j/2}h_{n-j}(x|q)\label{odwr}\\
\times\frac{Q_{j}\left(  \frac{a+b}{2(ab)^{1/2}}|(ab)^{1/2}(x+(x^{2}%
-1)^{1/2}),(ab)^{1/2}(x-(x^{2}-1)^{1/2}),q\right)  }{(ab)_{j}}.\nonumber
\end{gather}

Similarly, if $ab\geq0$ and we introduce parameters $\rho$ and $y$ by the
relationship $2\rho y\allowbreak=\allowbreak a+b$ and $\rho^{2}\allowbreak
=\allowbreak ab$ and the polynomials $p_{n}$, then (\ref{odwr}) takes the more
friendly form:%
\[
p_{n}(x|y,\rho,q)/(\rho^{2})_{n}=\sum_{j=0}^{n}%
%TCIMACRO{\QATOPD{[}{]}{n}{j}}%
%BeginExpansion
\genfrac{[}{]}{0pt}{}{n}{j}%
%EndExpansion
_{q}(-1)^{j}q^{\binom{j}{2}}\rho^{j}h_{n-j}(x|q)p_{j}(y|x,\rho,q)/\left(
\rho^{2}\right)  _{j}.
\]

Again, this formula was proved in \cite{Szab6} (Corollary 3) for complex
conjugate parameters $a$ and $b$ and $ab\geq0$. Since, however, on both sides
of the above-mentioned formula, we have rational functions in all arguments
$x,y,\rho,q$ we can naturally extend it to complex values of arguments.

\begin{remark}
From (\ref{hhQ}) follows directly the following integration formula that is
very important for the probabilistic interpretation:%
\[
\int_{-1}^{1}h_{n}(x|q)f_{Q}(x|a,b,q)\allowbreak dx=\allowbreak(ab)^{n/2}%
h_{n}\left(  \frac{a+b}{2(ab)^{1/2}}|q\right)  ,
\]
for all $n\geq1$. While from (\ref{odwr}), after introducing new parameters
defined by $\rho^{2}\allowbreak=\allowbreak ab$ and $2y\rho\allowbreak
=\allowbreak a+b$ and polynomials $p_{n}$ and assuming $ab\geq0,$ follows the
other more readable form of the formula above as well as the following
integration formula also important for the probabilistic interpretation:%
\begin{align*}
\allowbreak\int_{-1}^{1}h_{n}(x|q)f_{p}(x|y,\rho,q)\allowbreak dx  &
=\allowbreak\rho^{n}h_{n}(y|q),\\
\int_{-1}^{1}p_{n}(x|y,\rho,q)f_{p}(y|x,\rho,q)dy  &  =\allowbreak(\rho
^{2})_{n}h_{n}(x|q).
\end{align*}

for $\max(\left\vert x\right\vert ,\left\vert \rho\right\vert ,\left\vert
y\right\vert )\leq1$.
\end{remark}

All formulae (\ref{Qbh}), (\ref{hhQ}), (\ref{abcd}) and (\ref{odwr}) look much
more friendly when applied to parameters that form conjugate pairs and, what
is more, have nice probabilistic interpretations allowing to construct some
Markov processes compactly supported. See their interpretations presented in
Part \ref{prob}.

\subsubsection{Generating function}

We have
\[
\sum_{k=0}^{\infty}\frac{t^{k}}{\left(  q\right)  _{k}}Q_{k}\left(
x|a,b,q\right)  =\left(  at,bt\right)  _{\infty}\varphi_{h}(x|t,q),
\]
for all $\left\vert t\right\vert ,\left\vert q\right\vert ,\left\vert
ab\right\vert <1$.

Again, there exist many more formulae for the generating functions of ASC
polynomials. They are given say by the formulae (3.8.14)-(3.816) in
\cite{KLS-arX}.

We have, however, another two ones that are based on DEI((2.16),((2.9) of
\cite{Szab-bAW}) and DEI((3,11),(2.11) \cite{Szab-bAW}). Namely, for all
$\max(\left\vert x\right\vert ,\left\vert a\right\vert ,\left\vert
b\right\vert ,\left\vert c\right\vert ,\left\vert d\right\vert ,\left\vert
q\right\vert )<1$ we have%
\begin{gather*}
\sum_{j\geq0}\frac{(cd)^{j/2}}{(abcd,q)_{j}}Q_{j}(x|a,b,q)Q_{j}\left(
\frac{c+d}{2(cd)^{1/2}}|a(cd)^{1/2},b(cd)^{1/2},q\right) \\
=\frac{(ac,ad,bc,bd,cd)_{\infty}}{(abcd)_{\infty}}\varphi_{h}(x|c,q)\varphi
_{h}(x|d,q),\\
\sum_{k\geq0}\frac{(ab)^{k/2}}{(q)_{k}(ab)_{k}}b_{k}\left(  \frac{a+b}%
{2\sqrt{ab}}|q\right)  Q_{k}(x|a,b,q)=\frac{1}{(ab)_{\infty}}\prod
_{j=0}^{\infty}v(x|aq^{j})v(x|bq^{j}).
\end{gather*}

\begin{remark}
Comparing formulae (2.16) and (3.11) of \cite{Szab-bAW}, we get for free the
following, otherwise hard to prove, summation formula:%
\begin{equation}
\sum_{k=0}^{n}%
%TCIMACRO{\QATOPD{[}{]}{n}{k}}%
%BeginExpansion
\genfrac{[}{]}{0pt}{}{n}{k}%
%EndExpansion
_{q}c^{n-k}d^{k}\frac{(ac,bc)_{k}}{(abcd)_{k}}=\frac{(cd)^{n/2}}{(abcd)_{n}%
}Q_{n}\left(  \frac{c+d}{2(cd)^{1/2}}|a(cd)^{1/2},b(cd)^{1/2},q\right)  ,
\label{id_sum}%
\end{equation}
that is true for all $n\geq0$ and all complex $a,b,c,d,$ $q$ such $~$that
$abcdq^{k}\neq1$ for all $k\geq0$. This is so, since first we check the
validity of (\ref{id_sum}) for $\max(\left\vert a\right\vert ,\left\vert
b\right\vert ,\left\vert c\right\vert ,\left\vert d\right\vert ,\left\vert
q\right\vert )<1$ and then extend it, because on both sides we have rational functions.
\end{remark}

\begin{remark}
When $q\allowbreak=\allowbreak0$ then, (\ref{AlSC1}) is reduced to%
\[
2xQ_{n}(x|a,b,0)=Q_{n+1}(x|a,b,0)+Q_{n-1}(x|a,b,0),
\]
for $n\geq2,$ with $Q_{-1}(x|a,b,0)\allowbreak=\allowbreak0,$ $Q_{0}%
(x|a,b,0)\allowbreak=\allowbreak1$, $Q_{1}(x|a,b,0)\allowbreak=\allowbreak
U_{1}(x)\allowbreak-\allowbreak(a+b)U_{0}(x),$ $U_{2}(x|a,b,0)\allowbreak
=\allowbreak U_{2}(x)-(a+b)U_{1}(x)+abU_{0}(x),$ while the formula
(\ref{gASC}) to
\begin{equation}
f_{Q}(x|a,b,0)\allowbreak=\allowbreak\frac{2(1-ab)\sqrt{1-x^{2}}}%
{\pi(1-2xa+a^{2})(1-2xb+b^{2})}. \label{gASC0}%
\end{equation}
Following \cite{Szab4} and the fact that $\frac{2\sqrt{1-x^{2}}}{\pi}$ is the
density with respect to which Chebyshev polynomials of the second kind
$\left\{  U_{n}\right\}  _{n\geq-1}$ are orthogonal, we deduce that $n$-th
polynomial of the family must have the form of the linear combination of the
last $3$ (i.e.,$n$-th, $n-1$-th and $n-2$-th) Chebyshev polynomials of the
first kind. Taking into account the fact that polynomials of the form
\[
U_{n}(x)-(a+b)U_{n-1}(x)+abU_{n-2}(x)
\]
satisfy (\ref{gASC0}) for $n\geq2$ and for $n\allowbreak=\allowbreak-1,0,1,2$
these polynomials are polynomials $Q_{n}$ given above, hence we deduce that
\[
Q_{n}(x|a,b,0)\allowbreak=\allowbreak U_{n}(x)-(a+b)U_{n-1}(x)+abU_{n-2}(x),
\]
for all $n\geq-1$.
\end{remark}

\subsection{$q^{-1}-$ Al-Salam-Chihara}

As before, we consider an auxiliary family of polynomials, that we call
$q^{-1}$ Al-Salam-Chihara ( $q^{-1}-$ASC) polynomials. They are defined, as
before, by
\begin{equation}
\hat{Q}_{n}(x|a,b,q)=(-1)^{n}q^{\binom{n}{2}}Q_{n}(x|a,b,q^{-1}),
\label{odwrQ}%
\end{equation}
for $q\neq0$ and as a limit when $q\rightarrow0$ when $q\allowbreak
=\allowbreak0$.

\subsubsection{Three-term recurrence}

It is elementary to check, that they satisfy the following three-term
recurrence:%
\[
\hat{Q}_{n+1}(x|a,b,q)=-2(xq^{n}-(a+b))\hat{Q}_{n}\left(  x|a,b,q\right)
-(1-q^{n})(ab-q^{n-1})\hat{Q}_{n-1}\left(  x|y,\rho,q\right)  ,
\]
with $\hat{Q}_{-1}(x|a,b,q)\allowbreak=\allowbreak0,$ $\hat{Q}_{0}%
(x|a,b,q)\allowbreak=\allowbreak1$.

In \cite{Szab-bAW}( Proposition 3.1(iv)) it has been shown (for complex
conjugate parameters, but this can easily be extended to all cases of
parameters) that for $\forall n\geq1:$%
\begin{equation}
\sum_{j=0}^{n}%
%TCIMACRO{\QATOPD{[}{]}{n}{j}}%
%BeginExpansion
\genfrac{[}{]}{0pt}{}{n}{j}%
%EndExpansion
_{q}Q_{j}(x|a,b,q)\hat{Q}_{n-j}(x|a,b,q)=0. \label{ident}%
\end{equation}
From this identity almost directly follows the following form of the
generating function:%
\begin{equation}
\sum_{n\geq0}\frac{t^{n}}{(q)_{n}}\hat{Q}_{n}(x|a,b,q)=\frac{1}%
{(at,bt)_{\infty}}\prod_{j=0}^{\infty}v(x|tq^{j}). \label{ofdrgen}%
\end{equation}
Apart from these facts we have the following unexpected relationship proved in
\cite{Szab-bAW}, Prop 3.1(i) (again proved for the complex, mutually,
conjugate pair, but can be easily extended since we deal with polynomials).
For $\forall n\geq0$ and $ab\neq0$ we have%
\[
\hat{Q}_{n}(x|a,b,q)=(ab)^{n/2}Q_{n}\left(  \frac{a+b}{2(ab)^{1/2}}%
|\frac{x-(x^{2}-1)^{1/2}}{(ab)^{1/2}},\frac{x+(x^{2}-1)^{1/2}}{(ab)^{1/2}%
},q\right)  ,
\]
and $\hat{Q}_{n}(x|a,b,q)=b_{n}(x|q)$ if $ab=0$. Again, the above-mentioned
identity has a more friendly form, if one switches to polynomials $p_{n}$
defined by (\ref{pby Q}). Namely, we have%
\[
\hat{p}_{n}(x|y,\rho,q)\allowbreak=\allowbreak\rho^{n}p_{n}(y|x,1/\rho,q),
\]
for $\rho\neq0$ and $\hat{p}_{n}(x|y,0,q)\allowbreak=\allowbreak b_{n}(x|q)$
for $\rho\allowbreak=\allowbreak0$.

\begin{remark}
Following either (\ref{ident}) or (\ref{ofdrgen}) and setting $q\allowbreak
=\allowbreak0$, we see that
\[
\hat{Q}_{n}(x|a,b,0)=-2xD_{n}(a,b)+D_{n-1}(a,b)+D_{n+1}(a,b),
\]
where $D_{n}(a,b)=\left\{
\begin{array}
[c]{ccc}%
(a^{n}-b^{n})/(a-b) & \text{if} & a\neq b\\
na^{n-1} & \text{if} & a=b
\end{array}
\right.  $.
\end{remark}

\section{$3$-Parameter family}

\subsection{Continuous dual Hahn polynomials}

\subsubsection{3 term recurrence}

Out of the $3-$ parameter family of polynomials of the AW scheme, we will
consider only the continuous dual Hahn polynomials (briefly cdH polynomials).
Following \cite{Szab-bAW}, they satisfy the following three-term recurrence
for $n\geq1$:%
\begin{equation}
(2x-e_{n}(a,b,c|q))\psi_{n}(x|a,b,c,q)=\psi_{n+1}(x|a,b,c,q)+f_{n}%
(a,b,c|q)\psi_{n-1}(x|a,b,c,q), \label{3cdH}%
\end{equation}
where
\[
f_{n}(a,b,c|q)\allowbreak=\allowbreak(1-q^{n})(1-abq^{n-1})(1-cbq^{n-1}%
)(1-acq^{n-1}),
\]
and
\[
e_{n}(a,b,c|q)\allowbreak=\allowbreak(a+b+c)q^{n}\allowbreak+\allowbreak
abcq^{n-1}(1-q^{n}(1+q)),
\]
with $\psi_{-1}(x|a,b,c|,q)\allowbreak=\allowbreak0,$ $\psi_{0}%
(x|a,c,b,c,q)\allowbreak=\allowbreak1,$ \newline$\psi_{1}%
(x|a,c,b,c,q)\allowbreak=\allowbreak2x\allowbreak-\allowbreak(a+b+c)-abc$.

\begin{remark}
As in the case of ASC polynomials, $\psi$ depends, in fact, on parameters
$s_{1}\allowbreak=\allowbreak a+b+c+d,$ $s_{2}\allowbreak=\allowbreak
ab\allowbreak+\allowbreak ac\allowbreak+\allowbreak cb$ and $s_{3}%
\allowbreak=\allowbreak abc$. This is so we have $f_{n}(a,b,c|q)\allowbreak
=\allowbreak(1-q^{n})(1-s_{2}q^{n-1}\allowbreak+\allowbreak s_{1}s_{3}%
q^{2n-2}\allowbreak-\allowbreak s_{3}^{2}q^{3n-3})$ and $e_{n}%
(a,b,c|q)\allowbreak=\allowbreak s_{1}q^{n}\allowbreak+\allowbreak
s_{3}q^{n-1}(1-q^{n}(1+q))$. By the fundamental theorem of algebra we know
that for every triplet of parameters $(s_{1},s_{2},s_{3})$ we can find $3$
triplets of parameters $(a,b,c)$.
\end{remark}

\subsubsection{Orthogonality relations}

As before, if $\max(\left\vert a\right\vert ,\left\vert b\right\vert
,\left\vert c\right\vert )\leq1$, then the measure that makes these
polynomials orthogonal, is absolutely continuous with the density
\begin{equation}
f_{\psi}\left(  x|a,b,c,q\right)  \allowbreak=\allowbreak\left(
ab,ac,bc\right)  _{\infty}\varphi_{h}\left(  x|a,q\right)  \varphi_{h}\left(
x|b,q\right)  \varphi_{h}\left(  x|c,q\right)  f_{h}\left(  x|q\right)  .
\label{gest_cdH}%
\end{equation}
If one of the parameters, say $a>1,$ then this measure has $\#\{k:1<aq^{k}%
<a\}$ atoms at points (\ref{atoms}) with masses given by (formula (3.3.3) in
\cite{KLS-arX}).

Besides, for $\max(\left\vert a\right\vert ,\left\vert b\right\vert
,\left\vert c\right\vert )<1$ we have
\begin{equation}
\int_{\lbrack-1,1]}\psi_{n}\left(  x|a,b,c,q\right)  \psi_{m}\left(
x|a,b,c,q\right)  f_{\psi}\left(  x|a,b,c,q\right)  dx=\delta_{mn}\left(
ab,ac,bc,q\right)  _{n}. \label{psi2}%
\end{equation}

\begin{remark}
It turns out (see, e.g., \cite{Szabl-intAW}, Lemma2) that the density
$f_{\psi}$ can be expanded in the following way, when $\max(\left\vert
a\right\vert ,\left\vert b\right\vert ,\left\vert c\right\vert )\leq1$:%
\[
f_{\psi}\left(  x|a,b,c,q\right)  =f_{h}\left(  x|q\right)  \sum_{n\geq0}%
\frac{\sigma_{n}^{\left(  3\right)  }\left(  a,b,c|q\right)  }{\left(
q\right)  _{n}}h_{n}\left(  x|q\right)  ,
\]
where
\[
\sigma_{n}^{\left(  3\right)  }\left(  a,b,c|q\right)  =\sum_{j=0}^{n}%
%TCIMACRO{\QATOPD{[}{]}{n}{j}}%
%BeginExpansion
\genfrac{[}{]}{0pt}{}{n}{j}%
%EndExpansion
_{q}q^{\binom{j}{2}}\left(  -abc\right)  ^{j}S_{n-j}^{\left(  3\right)
}\left(  a,b,c|q\right)  ,
\]
and
\[
S_{n}^{\left(  3\right)  }\left(  a,b,c|q\right)  =\sum_{\substack{j,k,\geq0,
\\j+k\leq n}}\frac{(q)_{n}}{(q)_{j}(q)_{k}(q)_{n-k-j}}a^{k}b^{j}c^{n-j-k}.
\]

Following \cite{Szab-bAW}(2.13 with $d\allowbreak=\allowbreak0$) we also have%
\[
f_{\psi}\left(  x|a,b,c,q\right)  =f_{Q}(x|b,c,q)\sum_{j\geq0}\frac{a^{j}%
}{(q)_{j}}Q_{j}(x|b,c,q).
\]

\end{remark}

\subsubsection{Connection coefficients and other finite expansions}

Following \cite{AW85} (general formula for the connection coefficients between
AW polynomials with different sets of parameters), we have the following
formula for the connection coefficients between cdH and ASC polynomials.%
\begin{align}
\psi_{n}(x|a,b,c,q)  &  =\sum_{k=0}^{n}%
%TCIMACRO{\QATOPD{[}{]}{n}{k}}%
%BeginExpansion
\genfrac{[}{]}{0pt}{}{n}{k}%
%EndExpansion
_{q}(-a)^{n-k}q^{\binom{n-k}{2}}(bcq^{k})_{n-k}Q_{k}(x|b,c,q),\label{psinaQ}\\
Q_{n}(x|a,b,q)  &  =\sum_{k=0}^{n}%
%TCIMACRO{\QATOPD{[}{]}{n}{k}}%
%BeginExpansion
\genfrac{[}{]}{0pt}{}{n}{k}%
%EndExpansion
_{q}c^{n-k}(abq^{k})_{n-k}\psi_{k}(x|a,b,c,q)\text{.} \label{Qna psi}%
\end{align}

\subsubsection{Generating function}

In fact, we have at least $5$ different generating functions. The first four
are given by the formulae (3.3.13)-(3.3.16) in \cite{KLS-arX}. The fifth one
is based on DEI((\ref{psinaQ}),(\ref{psi2})). We have for $\max(\left\vert
x\right\vert ,\left\vert a\right\vert ,\left\vert b\right\vert ,\left\vert
c\right\vert ,\left\vert q\right\vert )<1$:%
\begin{equation}
\sum_{j\geq0}q^{\binom{j}{2}}\frac{(-c)^{j}}{(ac)_{j}(bc)_{j}(q)_{j}}\psi
_{j}(x|a,b,c,q)=\frac{\prod_{k=0}^{\infty}v(x|cq^{k})}{(bc)_{\infty
}(ac)_{\infty}}. \label{genpsi}%
\end{equation}
The sixth was derived first by Atakishiyeva and Atakishiyev in \cite{At-At11}
and is the following
\[
\sum_{j}\frac{a^{j}}{\left(  abcd,q\right)  _{j}}\psi_{j}\left(
x|b,c,d,q\right)  =\frac{\left(  ab,ac,ad\right)  _{\infty}}{\left(
abcd\right)  _{\infty}}\varphi_{h}\left(  x|a,q\right)  ,
\]
that is true for $\max(\left\vert a\right\vert ,\left\vert b\right\vert
,\left\vert c\right\vert ,\left\vert d\right\vert ,\left\vert q\right\vert
)<1$. However, following DEI((\ref{c_na_a}),(\ref{psi2})), we can get it
almost immediately.

\begin{remark}
Let us set $q\allowbreak=\allowbreak0$. Then
\[
f_{\psi}(x|a,b,c,0)=\frac{2\sqrt{1-x^{2}}}{\pi}\frac{1}{v(x|a)v(x|b)v(x|c)}.
\]
Now let us consider the fact that $\frac{2\sqrt{1-x^{2}}}{\pi}$ is the density
with respect to which Chebyshev polynomials of the second kind $\left\{
U_{n}\right\}  $ are orthogonal and let us follow the ideas of \cite{Szab4}.
Thus we deduce that the $n$-th polynomial of the family must be of the form of
the linear combination of the last $4$ (i.e.,$n$-th, $n-1$-th, $n-2$-th and
$n-3$-th) Chebyshev polynomials of the second kind. This is so since
$v(x|a)v(x|b)v(x|c)$ is the polynomial of order $3$. Secondly notice, that for
$n\geq2$ (\ref{3cdH}) takes the form:%
\[
2x\psi_{n}(x|a,b,c,0)=\psi_{n+1}(x|a,b,c,0)+\psi_{n-1}(x|a,b,c,0),
\]
that is the form of the three-term recurrence defining Chebyshev polynomials.
Taking into account the fact that polynomials of the form
\[
U_{n}(x)-(a+b+c)U_{n-1}(x)+(ab+bc+ac)U_{n-2}(x)-abcU_{n-3}%
\]
satisfy the above-mentioned three term recurrence for $n\geq2$. For
$n\allowbreak=\allowbreak-1,0,1,2$ these polynomials are polynomials $\psi
_{n}$ (one has to remember that $U_{-2}(x)\allowbreak=\allowbreak
-U_{0}(x)\allowbreak=\allowbreak-1)$ given above hence we deduce that
\[
\psi_{n}(x|a,b,0)\allowbreak=\allowbreak U_{n}(x)-(a+b+c)U_{n-1}%
(x)+(ab+ac+bc)U_{n-2}(x)-abcU_{n-3}(x),
\]
for all $n\geq-1$.
\end{remark}

\subsection{$q^{-1}$ Continuous dual Hahn polynomials}

They are defined as
\[
\hat{\psi}_{n}(x|a,b,c,q)\allowbreak=\allowbreak(-1)^{n}q^{\binom{n}{2}}%
\psi_{n}(x|a,b,c,q^{-1})
\]
if $q\neq0$ and $\hat{\psi}_{n}(x|a,b,c,0)\allowbreak=\allowbreak(-1)^{n}%
\psi_{n}(x|a,b,c,0)$ for $n\allowbreak=\allowbreak-1,0,1$ and $\psi
_{n}(x|a,b,c,0)\allowbreak=\allowbreak0$ for $n\geq2$.

$\left\{  \hat{\psi}_{n}\right\}  $ satisfy the following three-term
recurrence:%
\begin{gather*}
\hat{\psi}_{n+1}(x|a,b,c,q)\allowbreak=\allowbreak-(2xq^{n}-q^{n}%
e_{n}(a,b,c|q^{-1}))\hat{\psi}_{n}(x|a,b,c,q)\\
-q^{2n-1}f_{n}(a,b,c|q^{-1})\hat{\psi}_{n-1}(x|a,b,c,q),
\end{gather*}
with $\hat{\psi}_{-1}(x|a,b,c,q)\allowbreak=\allowbreak0,$ $\hat{\psi}%
_{0}(x|a,b,c,q)\allowbreak=\allowbreak1$.

\begin{remark}
Numerical simulations suggest, that the following identity holds $\forall
n\geq1:$%
\begin{equation}
\sum_{j=0}^{n}%
%TCIMACRO{\QATOPD{[}{]}{n}{j}}%
%BeginExpansion
\genfrac{[}{]}{0pt}{}{n}{j}%
%EndExpansion
_{q}\psi_{j}(x|a,b,c,d,q)\hat{\psi}_{n-j}(x|a,b,c,q)\allowbreak=\allowbreak0.
\label{ident3}%
\end{equation}
Notice that this identity is true for $q\allowbreak=\allowbreak0$ and
$n\geq0,$ since we have $\hat{\psi}_{n}(x|a,b,c,0)\allowbreak=\allowbreak
(-1)^{n},$ $n\allowbreak=\allowbreak-1,1$ and $\hat{\psi}_{n}%
(x|a,b,c,0)\allowbreak=\allowbreak0$ for $n\geq2$. In this case (\ref{ident3})
reduces itself to three-term recurrence satisfied by polynomials $\psi_{n}$,
hence is true for all $n\geq1$. Is it true in general, i.e., for all $q$?
\end{remark}

\section{$4-$Parameter family}

\subsection{Askey--Wilson polynomials}

$4-$parameter family of polynomials of the AW scheme are simply the
Askey-Wilson polynomials (briefly AW polynomials) introduced and described in
'85 in the celebrated paper \cite{AW85}.

\subsubsection{3 term recurrence}

Taking into account remarks in \cite{Szab-bAW}, they satisfy the following
three-term recurrence for $n\geq1$:%
\begin{equation}
2x\alpha_{n}\left(  x\right)  =\alpha_{n+1}\left(  x\right)  +e_{n}\left(
a,b,c,d,q\right)  \alpha_{n}(x)-f_{n}\left(  a,b,c,d,q\right)  \alpha
_{n-1}(x), \label{_aw}%
\end{equation}
with $\alpha_{-1}\left(  x\right)  \allowbreak=\allowbreak0,$ $\alpha
_{0}\left(  x\right)  \allowbreak=\allowbreak1,$ where for simplicity, we
denoted by
\begin{gather*}
f_{n}(a,b,c,d,q)=(1-q^{n})\times\\
\frac{\left(  abq^{n-1},acq^{n-1},adq^{n-1},bcq^{n-1},bdq^{n-1},cdq^{n-1}%
,abcdq^{n-2}\right)  _{1}}{(abcdq^{2n-3})_{3}(abcdq^{2n-2})_{1}},\\
e_{n}(a,b,c,d,q)=\frac{q^{n-2}}{(1-abcdq^{2n-2})\left(  1-abcdq^{2n}\right)
}\times\\
((a+b+c+d)(q^{2}-abcdq^{n}(1+q-q^{n+1}))+\\
(abc+abd+acd+bcd)(q-q^{n+2}-q^{n+1}+abcdq^{2n})).
\end{gather*}
In fact, the polynomials $\alpha_{n}$ are related to the polynomials $p_{n}$,
defined say in \cite{KLS-arX}, in the following way:%
\[
\alpha_{n}\left(  x|a,b,c,d,q\right)  \allowbreak=\allowbreak p_{n}\left(
x|a,b,c,d,q\right)  /\left(  abcdq^{n-1}\right)  _{n}.
\]
far all $n\geq-1$.

\subsubsection{Orthogonality relations}

Following \cite{Szab-bAW}, we have for $\max(\left\vert a\right\vert
,\left\vert b\right\vert ,\left\vert c\right\vert ,\left\vert d\right\vert
)\allowbreak<\allowbreak1$
\begin{gather}
\int_{\lbrack-1,1]}\alpha_{n}\left(  x|a,b,c,d,q\right)  \alpha_{m}\left(
x|a,b,c,d,q\right)  f_{AW}\left(  x|a,b,c,d,q\right)  dx\label{alfa^2}\\
=\delta_{mn}\frac{\left(  abcdq^{n-1}\right)  _{n}\left(
ab,ac,ad,bc,bd,cd,q\right)  _{n}}{\left(  abcd\right)  _{2n}},\nonumber
\end{gather}
where
\begin{gather}
f_{AW}\left(  x|a,b,c,d,q\right)  =f_{h}\left(  x|q\right)  \varphi_{h}\left(
x|a,q\right)  \varphi_{h}\left(  x|b,q\right)  \varphi_{h}\left(
x|c,q\right)  \varphi_{h}\left(  x|d,q\right)  \times\label{fAW}\\
\frac{\left(  ab,ac,ad,bc,bd,cd\right)  _{\infty}}{\left(  abcd\right)
_{\infty}}.\nonumber
\end{gather}

If one of the parameters, say $a>1,$ then, as before, there exist atoms of the
measure orthogonalizing polynomials $AW$. They are $\#\{k:1<aq^{k}<a\}$ atoms
consecrated at points (\ref{atoms}) with masses given by the formula below
(3.1.3) of \cite{KLS-arX}.

\subsubsection{Connection coefficients and other finite expansions}

Following \cite{Szab-bAW}( Corollary 2.1), we get
\begin{gather}
\alpha_{n}(x|a,b,c,d,q)=\sum_{k=0}^{n}\left[
%TCIMACRO{\QATOP{n}{k}}%
%BeginExpansion
\genfrac{}{}{0pt}{}{n}{k}%
%EndExpansion
\right]  _{q}\left(  -1\right)  ^{n-k}q^{\binom{n-k}{2}}\left(  cdq^{k}%
\right)  _{n-k}Q_{k}\left(  x|c,d,q\right) \label{aw_na_asc}\\
\times\sum_{m=0}^{n-k}\left[
%TCIMACRO{\QATOP{n-k}{m}}%
%BeginExpansion
\genfrac{}{}{0pt}{}{n-k}{m}%
%EndExpansion
\right]  _{q}q^{m(m-n+k)}a^{m}b^{n-k-m}\frac{\left(  bcq^{n-m},bdq^{n-m}%
\right)  _{m}}{\left(  abcdq^{2n-m-1}\right)  _{m}}.\nonumber\\
Q_{n}\left(  x|c,d,q\right)  \allowbreak=\sum_{j=0}^{n}\left[
%TCIMACRO{\QATOP{n}{j}}%
%BeginExpansion
\genfrac{}{}{0pt}{}{n}{j}%
%EndExpansion
\right]  _{q}\left(  cdq^{j}\right)  _{n-j}\alpha_{j}\left(
x|a,b,c,d,q\right) \label{asw_na_aw}\\
\times\sum_{m=0}^{n-j}\left[
%TCIMACRO{\QATOP{n-j}{m}}%
%BeginExpansion
\genfrac{}{}{0pt}{}{n-j}{m}%
%EndExpansion
\right]  _{q}b^{n-j-m}a^{m}\frac{\left(  bcq^{j},bdq^{j}\right)  _{m}}{\left(
abcdq^{2j}\right)  _{m}},\nonumber
\end{gather}

while, following \cite{Szab-bAW}(Lemma 2.1), we have
\begin{gather}
\alpha_{n}(x|a,b,c,d,q)\allowbreak=\allowbreak\sum_{i=0}^{n}%
%TCIMACRO{\QATOPD{[}{]}{n}{i}}%
%BeginExpansion
\genfrac{[}{]}{0pt}{}{n}{i}%
%EndExpansion
_{q}\left(  -a\right)  ^{n-i}q^{\binom{n-i}{2}}\frac{\left(  bcq^{i}%
,bdq^{i},cdq^{i}\right)  _{n-i}}{\left(  abcdq^{n+i-1}\right)  _{n-i}}\psi
_{i}\left(  x|b,c,d,q\right)  ,\label{a_na_c}\\
\psi_{n}\left(  x|b,c,d,q\right)  =\sum_{i=0}^{n}%
%TCIMACRO{\QATOPD{[}{]}{n}{i}}%
%BeginExpansion
\genfrac{[}{]}{0pt}{}{n}{i}%
%EndExpansion
_{q}a^{n-i}\frac{\left(  bcq^{i},bdq^{i},cdq^{i}\right)  _{n-i}}{\left(
abcdq^{2i}\right)  _{n-i}}\alpha_{i}\left(  x|a,b,c,d,q\right)  \text{.}
\label{c_na_a}%
\end{gather}

\begin{remark}
Taking into account identity (\ref{id_sum}), we can rewrite (\ref{asw_na_aw})
in the following form. For all $n\geq1:$
\begin{gather*}
Q_{n}\left(  x|c,d,q\right)  \allowbreak=\sum_{j=0}^{n}\left[
%TCIMACRO{\QATOP{n}{j}}%
%BeginExpansion
\genfrac{}{}{0pt}{}{n}{j}%
%EndExpansion
\right]  _{q}\left(  cdq^{j}\right)  _{n-j}\alpha_{j}\left(
x|a,b,c,d,q\right) \\
\times\frac{(ab)^{(n-j)/2}}{(abcdq^{2j})_{n-j}}Q_{n-j}(\frac{a+b}{2(ab)^{1/2}%
}|cq^{j}(ab)^{1/2},dq^{j}(ab)^{1/2},q)\text{.}%
\end{gather*}

\end{remark}

\subsubsection{Generating function}

In \cite{KLS-arX} there are 3 types of generating functions of the AW
polynomials, that can be reduced to a single one. This is so since starting
from the formula (3.1.13) of \cite{KLS-arX}, the other two one can obtain by
different choices two pairs out of $3$ possible choices of two pairs out of
$\left\{  a,b,c,d\right\}  $. One can obtain, however, another generating
function based on the ideas present in say (\ref{genpsi}). That is, we
consider DEI((\ref{a_na_c}),((\ref{alfa^2})) and get for\newline%
$\max(\left\vert x\right\vert ,\left\vert a\right\vert ,\left\vert
b\right\vert ,\left\vert c\right\vert ,\left\vert d\right\vert ,\left\vert
q\right\vert )<1:$%
\[
\frac{(abcd)_{\infty}}{(ad,bd,cd)_{\infty}}\prod_{j=0}^{\infty}v\left(
x|dq^{j}\right)  =\sum_{j\geq0}(-d)^{j}q^{\binom{j}{2}}\frac{(abcd)_{2j}%
}{(abcdq^{j-1})_{j}^{2}(ad,bd,cd,q)_{j}}\alpha_{j}(x|a,b,c,d,q)\text{.}%
\]
The other $3$ similar generating functions can be obtained by changing choices
of pairs in the denominator of right hand side of the above-mentioned formula.

\begin{remark}
For $q\allowbreak=\allowbreak0$ we have $e_{n}(a,b,c,d,0)\allowbreak
=\allowbreak0$ for $n\geq3$ and $f_{n}(a,b,c,d,0)\allowbreak=\allowbreak1$ for
$n\geq4$. That is for $n\geq4$ (\ref{_aw}) is the same as the three-term
recurrence satisfied by the Chebyshev polynomials. Besides, taking into
account results of \cite{Szab13} and \cite{Szab4} we deduce that for $n\geq4$
polynomials $\left\{  \alpha_{n}(x|a,b,c,d,0)\right\}  $ are linear
combinations of $U_{n},$ $U_{n-1},$ $U_{n-2},$ $U_{n-3},$ $U_{n-4}$. To get
the coefficients of this combination, we check the first four $\alpha^{\prime
}s$. We have
\begin{gather*}
a_{1}(x|a,b,c,d,0)\allowbreak=\allowbreak U_{1}(x)-\frac
{a+b+c+d-(abc+abd+acd+bcd)}{1-abcd},\\
a_{2}(x|a,b,c,d,0)\allowbreak=\allowbreak U_{2}(x)\allowbreak-\allowbreak
(a+b+c+d)U_{1}(x)\\
+(ab\allowbreak+\allowbreak ac\allowbreak+\allowbreak ad\allowbreak
+\allowbreak bc\allowbreak+\allowbreak bd\allowbreak+\allowbreak
cd)U_{0}(x)\allowbreak+\allowbreak abcdU_{-1}(x),\\
a_{3}(x|a,b,c,d,0)\allowbreak=\allowbreak U_{3}(x)-\allowbreak(a+b+c+d)U_{2}%
(x)\allowbreak\\
+\allowbreak(ab\allowbreak+\allowbreak ac\allowbreak+\allowbreak
ad\allowbreak+\allowbreak bc\allowbreak+\allowbreak bd\allowbreak+\allowbreak
cd)U_{1}(x)\allowbreak\\
-\allowbreak(abc+abd+acd+bcd)U_{0}(x)\allowbreak+\allowbreak abcdU_{-1}(x).
\end{gather*}

Hence, taking into account the observations from the beginning of this remark,
we deduce, that for $n\geq2$ we have%
\begin{gather*}
\alpha_{n}(x|a,b,c,d,0)=U_{n}(x)-(a+b+c+d)U_{n-1}(x)\allowbreak+\allowbreak\\
(ab\allowbreak+\allowbreak ac\allowbreak+\allowbreak ad\allowbreak+\allowbreak
bc\allowbreak+\allowbreak bd\allowbreak+\allowbreak cd)U_{n-2}(x)\allowbreak\\
-\allowbreak(abc+abd+acd+bcd)U_{n-3}(x)\allowbreak+\allowbreak abcdU_{n-4}(x).
\end{gather*}

\end{remark}

\section{Remaining formulae for connection coefficients and other useful
finite or infinite expansions including bivariate ones.\label{rem}}

This section is organized in such a way that the reference to a particular
family of polynomials will be exposed by the reference to its name. So, for
example, the connection coefficients between different families of Chebyshev
polynomials will be preceded by the heading T\&U.

\subsection{Connection coefficients}

\subsubsection{T\&U}

We have for $n\geq0$%

\begin{align}
T_{n}\left(  x\right)  \allowbreak &  =\allowbreak\left(  U_{n}\left(
x\right)  -U_{n-2}\left(  x\right)  \right)  /2,\label{TnaU}\\
U_{n}\left(  x\right)   &  =2\sum_{k=0}^{\left\lfloor n/2\right\rfloor
}T_{n-2k}\left(  x\right)  \allowbreak-\allowbreak\left(  1+\left(  -1\right)
^{n}\right)  /2. \label{UnaT}%
\end{align}

These expansions belong to common knowledge of the theory of special functions.

\subsubsection{h\&T}

Taking into account that (\ref{cH}) is equivalent to
\[
h_{n}\left(  x\right)  \allowbreak=\allowbreak\sum_{k=0}^{n}%
%TCIMACRO{\QATOPD{[}{]}{n}{k}}%
%BeginExpansion
\genfrac{[}{]}{0pt}{}{n}{k}%
%EndExpansion
_{q}\cos\left(  2k-n\right)  \theta,
\]
where $x=\cos\theta$ and (\ref{repCh}), we arrive at the following nice formula:%

\[
h_{n}\left(  x|q\right)  \allowbreak=\allowbreak\sum_{k=0}^{n}%
%TCIMACRO{\QATOPD{[}{]}{n}{k}}%
%BeginExpansion
\genfrac{[}{]}{0pt}{}{n}{k}%
%EndExpansion
_{q}T_{n-2k}\left(  x\right)  ,
\]
if one sets $T_{-n}\left(  x\right)  \allowbreak=\allowbreak T_{n}\left(
x\right)  ,$ $n\geq0$. On the other hand, taking into account (\ref{TnaU}) and
(\ref{Unah}), we get%
\[
T_{n}(x)=\frac{1}{2}\sum_{k=0}^{\left\lfloor n/2\right\rfloor }(-1)^{k}%
q^{\binom{k}{2}}(q^{k}%
%TCIMACRO{\QATOPD{[}{]}{n-k}{k}}%
%BeginExpansion
\genfrac{[}{]}{0pt}{}{n-k}{k}%
%EndExpansion
_{q}+%
%TCIMACRO{\QATOPD{[}{]}{n-k-1}{k-1}}%
%BeginExpansion
\genfrac{[}{]}{0pt}{}{n-k-1}{k-1}%
%EndExpansion
_{q})h_{n-2k}(x|q).
\]

\subsubsection{h\&h}

See formula (\ref{chbase}).

\subsubsection{b\&h}

See formula (\ref{bnah}).

\subsubsection{bh\&h}

See formula (\ref{bigh}) and (\ref{hnabh}).

\subsubsection{$q^{-1}$bh\&b}

See formula (\ref{Id1}).

\subsubsection{Q\&h}

\begin{proposition}%
\begin{gather*}
Q_{n}(x|a,b,q)=\sum_{k=0}^{\left\lfloor n/2\right\rfloor }%
%TCIMACRO{\QATOPD{[}{]}{n}{k}}%
%BeginExpansion
\genfrac{[}{]}{0pt}{}{n}{k}%
%EndExpansion
_{q}%
%TCIMACRO{\QATOPD{[}{]}{n-k}{k}}%
%BeginExpansion
\genfrac{[}{]}{0pt}{}{n-k}{k}%
%EndExpansion
_{q}(q)_{k}q^{k(k-1)}(ab)^{k}\times\\
\sum_{s=0}^{n-2k}(-q^{k})^{s}q^{\binom{s}{2}}%
%TCIMACRO{\QATOPD{[}{]}{n-2k}{s}}%
%BeginExpansion
\genfrac{[}{]}{0pt}{}{n-2k}{s}%
%EndExpansion
_{q}(ab)^{s/2}h_{n-2k-s}\left(  x|q\right)  h_{s}\left(  \frac{a+b}%
{2(ab)^{1/2}}|q\right)  .
\end{gather*}

\end{proposition}

\begin{proof}
First we use (\ref{Qbh}) and then (\ref{bnah}) obtaining
\begin{gather*}
Q_{n}(x|a,b,q)=\sum_{k=0}^{n}%
%TCIMACRO{\QATOPD{[}{]}{n}{k}}%
%BeginExpansion
\genfrac{[}{]}{0pt}{}{n}{k}%
%EndExpansion
_{q}(ab)^{k}h_{n-k}(x|q)\left(  -1\right)  ^{k}q^{\binom{k}{2}}\times\\
\sum_{s=0}^{\left\lfloor k/2\right\rfloor }%
%TCIMACRO{\QATOPD{[}{]}{k}{s}}%
%BeginExpansion
\genfrac{[}{]}{0pt}{}{k}{s}%
%EndExpansion
_{q}%
%TCIMACRO{\QATOPD{[}{]}{k-s}{s}}%
%BeginExpansion
\genfrac{[}{]}{0pt}{}{k-s}{s}%
%EndExpansion
_{q}(q)_{s}q^{s(s-k)}h_{k-2s}\left(  \frac{a+b}{2(ab)^{1/2}}|q\right)  .
\end{gather*}
Now we change the order of summation.
\end{proof}

\subsubsection{C\&C}

See formula (\ref{CnaC}).

\subsubsection{Q\&Q}

See formulae (\ref{abcd}) and (\ref{odwr}).

\subsection{Linearization formulae}

These are the formulae expressing a product of two or more polynomials of the
same type as linear combinations of polynomials of the same type as the ones
produced. We will extend the name 'linearization formulae' by relaxing the
requirement of the polynomials involved to be of the same type. Generally,
obtaining 'linearization formula ' is not simple and requires a lot of tedious calculations.

\subsubsection{h\&h}

See formulae (\ref{lihh}) and (\ref{linhhh}).

\subsubsection{h\&b}

See formula (\ref{linbh}).

\subsubsection{h\&C}

We have also the useful formula:

$\forall n,m\geq1:$%
\begin{gather}
(q)_{n}h_{m}\left(  x|q\right)  C_{n}\left(  x|\beta,q\right)  =\sum
_{\substack{k,j\geq0 \\k+j\leq(n+m)/2}}\left(  -\beta\right)  ^{k}%
%TCIMACRO{\QATOPD{[}{]}{m}{j}}%
%BeginExpansion
\genfrac{[}{]}{0pt}{}{m}{j}%
%EndExpansion
_{q}%
%TCIMACRO{\QATOPD{[}{]}{n}{k+j}}%
%BeginExpansion
\genfrac{[}{]}{0pt}{}{n}{k+j}%
%EndExpansion
_{q}\label{HRnaH}\\
\times%
%TCIMACRO{\QATOPD{[}{]}{n-k-j}{k}}%
%BeginExpansion
\genfrac{[}{]}{0pt}{}{n-k-j}{k}%
%EndExpansion
(q)_{k+j}q^{\binom{k}{2}}\left(  \beta\right)  _{n-k}h_{n+m-2k-2j}\left(
x|q\right)  ,\nonumber
\end{gather}
which was proved in \cite{ALIs88} (1.9).

\subsubsection{Q\&Q}

For the sake completeness let us mention that in \cite{Stanton08} there is a
very complicated linearization formula for Al-Salam--Chihara polynomials given
in Theorem 1.

\subsection{Useful finite sums and identities}

We have also a very useful generalization of the formula (1.12) of \cite{bms}
which was proved in \cite{Szab6} (Lemma2 assertion i)).

Let us remark that for $q\allowbreak=\allowbreak0$, (\ref{suma BH}) reduces to
three-term recurrence of polynomials $U_{n}\left(  x/2\right)  $.

Let us return to the modified version of ASC polynomials. More precisely, to
the polynomials defined by (\ref{pn}), i.e.,let us define:%
\[
p_{n}(x|y,\rho,q)=Q_{n}(x|\rho(y-\sqrt{y^{2}-1}),\rho(y+\sqrt{y^{2}-1}),q),
\]
remembering that polynomials $Q_{n}$ depend on the sum and product of
parameters $a$ and $b$.

Recently in \cite{Szab-bAW} the following identities involving ASC polynomials
$p_{n}$ were given:

i) $\forall n\geq1,0\leq k<n,z,y,t\in\mathbb{R}:$%
\[
\sum_{j=0}^{n-k}%
%TCIMACRO{\QATOPD{[}{]}{n-k}{j}}%
%BeginExpansion
\genfrac{[}{]}{0pt}{}{n-k}{j}%
%EndExpansion
_{q}\frac{p_{j}\left(  z|y,tq^{k},q\right)  }{\left(  t^{2}q^{2k}\right)
_{j}}\frac{g_{n-k-j}\left(  z|y,tq^{n-1},q\right)  }{\left(  t^{2}%
q^{n+j+k-1}\right)  _{n-k-j}}\allowbreak=\allowbreak0,
\]

ii) $\forall n\geq1,0\leq k<n,z,y,t\in\mathbb{R}:$%
\[
\sum_{m=0}^{n-k}%
%TCIMACRO{\QATOPD{[}{]}{n-k}{m}}%
%BeginExpansion
\genfrac{[}{]}{0pt}{}{n-k}{m}%
%EndExpansion
_{q}\frac{p_{n-k-m}\left(  z|y,tq^{m+k},q\right)  g_{m}(z|y,tq^{m+k-1}%
,q)}{\left(  t^{2}q^{2m+2k}\right)  _{n-k-m}\left(  t^{2}q^{m+2k-1}\right)
_{m}}\allowbreak=\allowbreak0,
\]
where polynomials $g_{n}$ are somewhat analogous to the polynomials $b_{n}$
and are defined by the formula:%
\begin{equation}
g_{n}\left(  x|y,\rho,q\right)  \allowbreak=\allowbreak\left\{
\begin{array}
[c]{ccc}%
\rho^{n}p_{n}\left(  y|x,\rho^{-1},q\right)  & \text{if} & \rho\neq0\\
b_{n}\left(  x|q\right)  & \text{if} & \rho=0
\end{array}
\right.  . \label{_g}%
\end{equation}
One showed also there, that for $n\geq-1:$
\[
g_{n}(x|y,\rho,q)\allowbreak=\allowbreak\hat{Q}_{n}(x|\rho(y-\sqrt{y^{2}%
-1}),\rho(y+\sqrt{y^{2}-1}),q),
\]
where polynomials $\hat{Q}_{n}$ are $q^{-1}-$ASC polynomials that are defined
in (\ref{odwrQ}).

Recall in this context, that we have also (\ref{ident}) and (\ref{ofdrgen}).

Exploring Carlitz paper \cite{Carlitz72} and confronting it with the above
Lemma \ref{UogCarl}, below we arrive at the following conversion Lemma.

\begin{lemma}
$\forall n,m\geq0,$ $\left\vert t\right\vert <1,$ $\theta\in(-\pi,\pi]:$
\begin{align}
&  \sum_{k=0}^{m}\sum_{l=0}^{n}%
%TCIMACRO{\QATOPD{[}{]}{m}{k}}%
%BeginExpansion
\genfrac{[}{]}{0pt}{}{m}{k}%
%EndExpansion
_{q}%
%TCIMACRO{\QATOPD{[}{]}{n}{l}}%
%BeginExpansion
\genfrac{[}{]}{0pt}{}{n}{l}%
%EndExpansion
_{q}\frac{\left(  te^{i\left(  -\theta+\eta\right)  }\right)  _{k}\left(
te^{i\left(  \theta-\eta\right)  }\right)  _{l}\left(  te^{-i\left(
\theta+\eta\right)  }\right)  _{k+l}}{\left(  t^{2}\right)  _{k+l}%
}e^{-i\left(  m-2k\right)  \theta}e^{-i(n-2l)\eta}\label{upr_Car}\\
&  =\sum_{j=0}^{n}(-1)^{j}q^{\binom{j}{2}}%
%TCIMACRO{\QATOPD{[}{]}{n}{j}}%
%BeginExpansion
\genfrac{[}{]}{0pt}{}{n}{j}%
%EndExpansion
_{q}t^{j}h_{n-j}(y|q)p_{m+j}\left(  x|y,t,q\right)  /\left(  t^{2}\right)
_{j+m},\nonumber
\end{align}
where $p_{n}(x|y,t,q)\allowbreak=\allowbreak Q_{n}(x|t(y+(y^{2}-1)^{1/2}%
),t(y-(y^{2}-1)^{1/2}),q)$ and $x\allowbreak=\allowbreak\cos\theta$ and
$y\allowbreak=\allowbreak\cos\eta$.
\end{lemma}

\begin{proof}
See \cite{Szab-bAW} Proposition 6.
\end{proof}

\subsection{Useful bivariate identities including infinite ones.}

The most important in the context of $q-$series theory is undoubtedly the
so-called Poisson--Mehler identity. It that can be obtained from
(\ref{mehler}) and (\ref{n_form}) when the new parameters $y$ and $\rho$ are
introduced. Recall that these parameters are expressed in terms of $a$ and $b$
in the following form $a+b\allowbreak=\allowbreak2\rho y$ and $ab\allowbreak
=\allowbreak\rho^{2}$. Then $\frac{(a+b)}{2(ab)^{1/2}}\allowbreak=\allowbreak
y$ so we have%
\[
\sum_{n\geq0}\frac{\rho^{n}}{(q)_{n}}h_{n}(x|q)h_{n}(y|q)\allowbreak
=\allowbreak\frac{(\rho^{2})_{\infty}}{\prod_{j=0}^{\infty}w(x,y|\rho q^{j}%
)},
\]
where by $w(x,y|r)$ is given by (\ref{wxya}). Let us denote for simplicity
\begin{equation}
V(x,y|\rho,q)=\prod_{j=0}^{\infty}w(x,y|\rho q^{j}). \label{VV}%
\end{equation}
Notice that the right-hand side of the above mentioned equality is nonnegative
for $\max(\left\vert x\right\vert ,\left\vert y\right\vert ,\left\vert
\rho\right\vert ,\left\vert q\right\vert )<1$ and also that
\[
\int_{-1}^{1}\left(  \sum_{n\geq0}\frac{\rho^{n}}{(q)_{n}}h_{n}(x|q)h_{n}%
(y|q)\right)  f_{h}(x|q)dx\allowbreak=\allowbreak1;
\]
hence the following function
\[
f_{h}(x)\frac{(\rho^{2})_{\infty}}{V(x,y|\rho,q)}%
\]
is the density.

The polynomials, that are orthogonal with respect to this density, are, in
fact, ASC polynomials $Q_{n}$ considered for specific values of parameters.
They were already introduced and named $p_{n}(x|y,\rho,q)$. A generalization
of this formula and its probabilistic applications are presented in the next
section in particular in Lemma \ref{UogCarl}.

There are other important formulae involving polynomials $\left\{
h_{n}\right\}  $ and $\left\{  b_{n}\right\}  $ considered with different
arguments. Here we will mention only the finite ones.

Let us start with the formula that has been proved by Ismail and Stanton in
\cite{IsStan200}
\begin{align}
&  \sum_{k=0}^{n}%
%TCIMACRO{\QATOPD{[}{]}{n}{k}}%
%BeginExpansion
\genfrac{[}{]}{0pt}{}{n}{k}%
%EndExpansion
_{q}q^{-k\left(  n-k\right)  /2}h_{k}\left(  x|q\right)  h_{n-k}\left(
y|\frac{1}{q}\right) \label{IsmStan}\\
&  =e^{-in\phi}\left(  -q^{\left(  1-n\right)  /2}e^{i\left(  \theta
+\phi\right)  },-q^{\left(  1-n\right)  /2}e^{i\left(  -\theta+\phi\right)
}|q\right)  _{n},\nonumber
\end{align}
where $x=\cos\theta$ and $y=\cos\phi,$ that later was simplified by
Szab\l owski in \cite{Szab08} to the following form%
\begin{equation}
\sum_{k=0}^{n}%
%TCIMACRO{\QATOPD{[}{]}{n}{k}}%
%BeginExpansion
\genfrac{[}{]}{0pt}{}{n}{k}%
%EndExpansion
_{q}q^{-k\left(  n-k\right)  /2}h_{k}\left(  x|q\right)  h_{n-k}\left(
y|\frac{1}{q}\right)  =2^{n}\allowbreak\left\{
\begin{array}
[c]{ccc}%
\prod_{j=0}^{k-1}t_{2j+1}\left(  x,y|q\right)  & \text{if} & n=2k\\
\prod_{j=0}^{k}t_{2j}\left(  x,y|q\right)  & \text{if} & n=2k+1
\end{array}
\right.  \text{,} \label{simIsm}%
\end{equation}
where
\[
t_{n}\left(  x,y|q\right)  \allowbreak=\allowbreak x^{2}+y^{2}\allowbreak
+\allowbreak xy\left(  q^{n/2}+q^{-n/2}\right)  \allowbreak+\allowbreak\left(
q^{n}+q^{-n}-2\right)  /4,
\]
for $n\geq1$ and $t_{0}\left(  x,y|q\right)  \allowbreak=\allowbreak x+y$.

Recently in \cite{SzabCheb} the following formulae have been proved
\begin{align}
d_{n}^{(2)}(\cos\theta,\cos\varphi|q)\allowbreak &  =\allowbreak\sum_{m=0}^{n}%
%TCIMACRO{\QATOPD{[}{]}{n}{m}}%
%BeginExpansion
\genfrac{[}{]}{0pt}{}{n}{m}%
%EndExpansion
_{q}b_{m}(\cos(\theta+\varphi)|q)b_{n-m}(\cos(\theta-\varphi)|q),\label{d2b}\\
f_{n}^{(2)}(\cos\theta,\cos\varphi|q)  &  =\sum_{m=0}^{n}%
%TCIMACRO{\QATOPD{[}{]}{n}{m}}%
%BeginExpansion
\genfrac{[}{]}{0pt}{}{n}{m}%
%EndExpansion
_{q}h_{m}(\cos(\theta+\varphi)|q)h_{n-m}(\cos(\theta-\varphi)|q), \label{d2h}%
\end{align}
where $x\allowbreak=\allowbreak\cos\theta$ and $y\allowbreak=\allowbreak
\cos\varphi$ and
\begin{gather}
d_{n}^{(2)}(x,y|q)=\label{con1}\\
(-1)^{n}\sum_{j=0}^{\left\lfloor n/2\right\rfloor }(-1)^{j}q^{-\binom{n-2j}%
{2}-j+\binom{j}{2}}\frac{(q)_{n}}{(q)_{j}(q)_{n-2j}}b_{n-2j}(x|q)b_{n-2j}%
(y|q),\nonumber\\
f_{n}^{(2)}(x,y|q)\allowbreak=\allowbreak\sum_{j=0}^{\left\lfloor
n/2\right\rfloor }\frac{(q)_{n}}{(q)_{j}(q)_{n-2j}}h_{n-2j}(x|q)h_{n-2j}(y|q).
\label{con2}%
\end{gather}

The r\^{o}le of the functions $d_{n}^{(2)}$ and $f_{n}^{(2)}$ can be seen in
the following relationships:%

\[
\frac{n!}{(q)_{n}}d_{n}^{(2)}(x,y|q)\allowbreak=\allowbreak\left.  \frac
{d^{n}}{d\rho^{n}}V(x,y|\rho,q)\right\vert _{\rho=0},
\]
hence
\[
V(x,y|\rho,q)=\sum_{n\geq0}\frac{\rho^{n}}{(q)_{n}}d_{n}^{(2)}(x,y|q),
\]
and
\[
\frac{n!}{(q)_{n}}f_{n}^{(2)}(x,y|q)\allowbreak=\allowbreak\left.  \frac
{d^{n}}{d\rho^{n}}V^{-1}(x,y|\rho,q)\right\vert _{\rho=0},
\]
hence
\[
V^{-1}(x,y|\rho,q)=\sum_{n\geq0}\frac{\rho^{n}}{(q)_{n}}f_{n}^{(2)}(x,y|q),
\]
where $V$ is given above by (\ref{VV}). We have also for free the following
identity:%
\[
\sum_{j=0}^{n}%
%TCIMACRO{\QATOPD{[}{]}{n}{j}}%
%BeginExpansion
\genfrac{[}{]}{0pt}{}{n}{j}%
%EndExpansion
_{q}f_{j}^{(2)}(x,y|q)d_{n-j}^{(2)}(x,y|q)\allowbreak=\allowbreak0,
\]
for $n\geq1,$ $x,y,q\in\mathbb{C}$. Of course, this identity was proved for
$\max(\left\vert x\right\vert ,\left\vert y\right\vert ,\left\vert
q\right\vert )<1$, but since both $f_{n}^{\prime}s$ and $d_{n}^{\prime}s$ are
polynomials in $x$, $y$ and $q$, the identity can be extended for all values
complex of unknowns.

It was proved recently in \cite{SzabCheb} that we also have%

\begin{equation}
\sum_{m=0}^{k}%
%TCIMACRO{\QATOPD{[}{]}{k}{m}}%
%BeginExpansion
\genfrac{[}{]}{0pt}{}{k}{m}%
%EndExpansion
_{q}d_{m}^{(2)}(x,y|q)h_{k-m}(x|q)h_{k-m}(y|q)\allowbreak=\left\{
\begin{array}
[c]{ccc}%
0 & \text{if} & k\text{ is odd}\\
(-1)^{l}q^{\binom{l}{2}}(q^{l+1})_{l} & \text{if} & k=2l
\end{array}
\right.  . \label{00k}%
\end{equation}

\part{Askey-Wilson scheme compactly supported on $[-2/\sqrt{1-q},2/\sqrt
{1-q}]$. Probabilistic interpretation \label{prob}}

This part is devoted to applications of all the mentioned above families of
polynomials that appear in the theory of probability and the theory of
stochastic processes, more precisely the Markov processes. Namely, we will
define multidimensional distributions and (or) Markov processes, whose
marginal or conditional distributions are the distributions that make
orthogonal all or some of the families of orthogonal polynomials that were
mentioned in the previous sections of the paper. It turns out that these
applications become more noticeable when the parameters defining the
above-mentioned families of polynomials are redefined and moreover, the
support of measures, that made these families of polynomials orthogonal, is
made dependent on the parameter $q$. Namely, in this section, all measures
considered will be supported on the segment
\[
S(q)=[-2/\sqrt{1-q},2/\sqrt{1-q}].
\]
Notice that when $q\rightarrow1^{-}$, $S(q)$ tends to the real line
$\mathbb{R}$.

Consequently, the densities of the considered measures and the families of
polynomials, have to be redefined. We start with the $q$-Hermite polynomials.
Now we will consider the polynomials
\begin{equation}
H_{n}\left(  x|q\right)  \allowbreak=\allowbreak(1-q)^{-n/2}h_{n}\left(
\frac{x\sqrt{1-q}}{2}|q\right)  , \label{qH}%
\end{equation}
where polynomials $h_{n}$ are the $q-$Hermite polynomials considered above.
They satisfy the following three-term recurrence:%
\begin{equation}
xH_{n}\left(  x|q\right)  =H_{n+1}\left(  x|q\right)  \allowbreak+\left[
n\right]  _{q}H_{n-1}\left(  x\right)  , \label{3trH}%
\end{equation}
for $n\geq1$ with $H_{-1}\left(  x|q\right)  \allowbreak=\allowbreak0$,
$H_{1}\left(  x|q\right)  \allowbreak=\allowbreak1$. Notice that now
polynomials $H_{n}$ are monic and also that
\[
\lim_{q\rightarrow1^{-}}H_{n}(x|q)=H_{n}(x),
\]
since $[n]_{1}\allowbreak=\allowbreak n$. The density of the probability
measure, that makes these polynomials orthogonal, is now given by%
\[
f_{N}\left(  x|q\right)  \allowbreak=\allowbreak\left\{
\begin{array}
[c]{ccc}%
\sqrt{1-q}f_{h}(x\sqrt{1-q}/2|q)/2 & \text{if} & \left\vert q\right\vert <1\\
\exp\left(  -x^{2}/2\right)  /\sqrt{2\pi} & \text{if} & q=1
\end{array}
\right.  .
\]

For the sake of completeness, let us define also polynomials $\left\{
B_{n}\left(  x|q\right)  \right\}  _{n\geq-1}$ by the following formula
(compare \cite{bms}):%
\[
B_{n}(x|q)=\left\{
\begin{array}
[c]{ll}%
i^{n}q^{n(n-2)/2}H_{n}(i\sqrt{q}\,x|q^{-1}) & \text{for }1\geq q>0\\
(-1)^{\binom{n}{2}}|q|^{n(n-2)/2}H_{n}(-\sqrt{|q|}\,x|q^{-1}) & \text{for
}-1<q<0
\end{array}
\right.  ,
\]
and satisfying the following three-term recurrence:%
\begin{equation}
B_{n+1}\left(  y|q\right)  \allowbreak=\allowbreak-q^{n}yB_{n}\left(
y|q\right)  +q^{n-1}\left[  n\right]  _{q}B_{n-1}\left(  y|q\right)  .
\label{Be}%
\end{equation}
In fact, polynomials $B_{n}$ can be also defined by
\[
B_{n}(x|q)=(1-q)^{-n/2}b_{n}\left(  \frac{\sqrt{1-q}}{2}x|q\right)  ,
\]
where $b_{n}$ is given by (\ref{bbb}).

The important identity (\ref{simpbh}) now (that is with newly defined
polynomials $H_{n}$ and $B_{n}$) takes the following form. For all $n\geq0:$%
\begin{equation}
\sum_{k=0}^{n}%
%TCIMACRO{\QATOPD{[}{]}{n}{k}}%
%BeginExpansion
\genfrac{[}{]}{0pt}{}{n}{k}%
%EndExpansion
_{q}B_{n-k}\left(  x|q\right)  H_{k+m}\left(  x|q\right)  \allowbreak=\left\{
\begin{array}
[c]{ccc}%
0 & \text{if} & n>m\\
(-1)^{n}q^{\binom{n}{2}}\frac{\left[  m\right]  _{q}!}{\left[  m-n\right]
_{q}!}H_{m-n}\left(  x|q\right)  & \text{if} & m\geq n
\end{array}
\right.  . \label{suma BH}%
\end{equation}

Notice that due to the fact that the support depends on the parameter $q,$ we
are able to include the case $q\rightarrow1^{-}$. The property that $\left[
n\right]  _{1}\allowbreak=\allowbreak n$ shows that $H_{n}(x|1)\allowbreak
=\allowbreak H_{n}(x)$ hence; it suggests, that the measure defined by the
density $f_{N}$ is a kind of generalization of the Normal or Gaussian measure.

Indeed, the fact that $f_{N}(x|q)\rightarrow\exp\left(  -x^{2}/2\right)
/\sqrt{2\pi}$ as $q\rightarrow1^{-}$ was established by Ismail years ago in
\cite{ISV87}. Thus we could have defined density $f_{N}$ for $q\allowbreak
=\allowbreak1$. We will call the measure with this density $q-$Normal or
$q-$Gaussian. Similarly, for $\forall n\geq1$ we have
\[
B_{n}(x|q)\rightarrow i^{n}H_{n}\left(  ix\right)  \text{,}%
\]
as $q\rightarrow1^{-1}$.

Let us also remark that for $\forall n\geq-1$
\[
H_{n}(x|0)\allowbreak=\allowbreak U_{n}(x/2).
\]

Notice that for $q\allowbreak=\allowbreak1$ we get from (\ref{suma BH}) the
following identity true for all nonnegative integers $n,m$
\[
\sum_{k=0}^{n}\binom{n}{k}i^{n-k}H_{n-k}\left(  ix\right)  H_{k+m}\left(
x\right)  \allowbreak=\allowbreak\left\{
\begin{array}
[c]{ccc}%
0 & \text{if} & n>m\\
(-1)^{n}\frac{m!}{(m-n)!}H_{m-n}\left(  x\right)  & \text{if} & m\geq n
\end{array}
\right.  .
\]

Let us also define the modified $q$-ultraspherical polynomials. Namely, we
will consider now polynomials $R_{n}\left(  x|\beta,q\right)  $ related to the
polynomials $C_{n}$ through the relationship:
\begin{equation}
C_{n}\left(  x|\beta,q\right)  \allowbreak=\allowbreak\left(  1-q\right)
^{n/2}R_{n}\left(  \frac{2x}{\sqrt{1-q}}|\beta,q\right)  /\left(  q\right)
_{n}, \label{q-US}%
\end{equation}
for $n\geq1$. It is not difficult to notice, that polynomials $R_{n}$ satisfy
the following three-term recurrence:%
\begin{equation}
\left(  1-\beta q^{n}\right)  xR_{n}\left(  x|\beta,q\right)  =R_{n+1}\left(
x|\beta,q\right)  +\left(  1-\beta^{2}q^{n-1}\right)  \left[  n\right]
_{q}R_{n-1}\left(  x|\beta,q\right)  \text{.} \label{R}%
\end{equation}

The density of the measure that makes polynomials $R_{n}$ orthogonal, is now
given by
\begin{equation}
f_{R}\left(  x|\beta,q\right)  =\sqrt{1-q}f_{C}(x\sqrt{1-q}/2|q)/2, \label{fR}%
\end{equation}
where $f_{C}$ is given by (\ref{fC})

Below, we analyze extreme or particular cases and we have the following observations:

\begin{proposition}
For $n\geq0$ and, remembering that polynomials $H_{n}$, $T_{n}$ and $U_{n}$
are defined in Section \ref{pom}, we have

i) $R_{n}\left(  x|0,q\right)  \allowbreak=\allowbreak H_{n}\left(
x|q\right)  $,

ii) $R_{n}\left(  x|q,q\right)  \allowbreak=\allowbreak\left(  q\right)
_{n}U_{n}\left(  x\sqrt{1-q}/2\right)  ,$

iii) $\lim_{\beta->1^{-}}\frac{R_{n}\left(  x|\beta,q\right)  }{\left(
\beta\right)  _{n}}\allowbreak=\allowbreak2\frac{T_{n}\left(  x\sqrt
{1-q}/2\right)  }{(1-q)^{n/2}}$,

iv) $R_{n}(x|\beta,1)\allowbreak=\allowbreak\left(  \frac{1+\beta}{1-\beta
}\right)  ^{n/2}H_{n}\left(  \sqrt{\frac{1-\beta}{1+\beta}}x\right)  $,

v) $R_{n}(x|\beta,0)\allowbreak=\allowbreak(1-\beta)U_{n}(x/2)-\beta
(1-\beta)U_{n-2}(x/2)$.

vi) $R_{n}(x|\beta,q)=P_{n}(x|x,\beta,q)$, where $P_{n}$ is defined by its
three-term recurrence (\ref{3Pn}), below.
\end{proposition}

\begin{proof}
i), ii), iii) were proved in \cite{Szab-rev}, while iv) and v) were shown in
\cite{Szab19}. To see vi) compare (\ref{R} and (\ref{3Pn}) and recall that
$C_{-1}\allowbreak=\allowbreak0,$ while $C_{0}\allowbreak=\allowbreak1$.
\end{proof}

Since the ideas of the probability theory and in particular of the
distribution theory are simpler to comprehend for the non-specialists, we will
define a $3-$dimensional distribution having density such that their both
marginal and conditional distributions are one of the types mentioned above,
i.e., belonging to Askey-Wilson scheme. To do this we need, firstly to
redefine parameters and instead, parameters $a$, $b$, $c$, $d$ we consider two
conjugate pairs of complex numbers. Hence now we will have%
\begin{equation}
a=\rho_{1}e^{i\theta},~b=\rho_{1}e^{-i\theta},~c=\rho_{2}e^{i\varphi}%
,~d=\rho_{2}e^{-i\varphi}\text{.} \label{n_param}%
\end{equation}

Let us denote also $\cos\theta\allowbreak=\allowbreak\sqrt{1-q}y_{1}/2$ and
$\cos\varphi\allowbreak=\allowbreak\sqrt{1-q}y_{2}/2$. Hence $y_{1},y_{2}\in
S(q)$. We have also $a+b\allowbreak=\allowbreak\rho_{1}y_{1}\sqrt{1-q}$ and
$ab\allowbreak=\allowbreak\rho_{1}^{2}$. Let us agree that, if only one pair
of parameters is used, then the related to them parameters $\rho$ and $y$ will
not have subindices. Further, for $n\geq-1$ let us denote
\begin{equation}
P_{n}\left(  x|y,\rho,q\right)  \allowbreak=\frac{1}{(1-q)^{n/2}}Q_{n}\left(
x\frac{\sqrt{1-q}}{2}|a,b,q\right)  \text{,} \label{podstawienie}%
\end{equation}
with
\[
a=\frac{\rho\sqrt{1-q}}{2}\left(  y\allowbreak-\allowbreak i\sqrt{\frac
{4}{1-q}-y^{2}}\right)  ,~b=\frac{\rho\sqrt{1-q}}{2}\left(  y\allowbreak
+\allowbreak i\sqrt{\frac{4}{1-q}-y^{2}}\right)  \text{.}%
\]
Let us also notice, that polynomials $\left\{  P_{n}\right\}  $ satisfy the
following three-term recurrence:%
\begin{equation}
P_{n+1}(x|y,\rho,q)=(x-\rho yq^{n})P_{n}(x|y,\rho,q)-(1-\rho^{2}%
q^{n-1})[n]_{q}P_{n-1}(x|y,\rho,q), \label{3Pn}%
\end{equation}
with $P_{-1}(x|y,\rho,q)\allowbreak=\allowbreak0$ and $P_{0}(x|y,\rho
,q)\allowbreak=\allowbreak1$.

Notice also, that for $q\allowbreak=\allowbreak1$ we have%
\[
P_{n+1}(x|y,\rho,1)=(x-\rho y)P_{n}(x|y,\rho,1)-n(1-\rho^{2})P_{n-1}%
(x|y,\rho,1),
\]
that is we have
\begin{equation}
P_{n}(x|y,\rho,1)=(1-\rho^{2})^{n/2}H_{n}\left(  (x-\rho y)/\sqrt{1-\rho^{2}%
}\right)  . \label{Pdlaq=1}%
\end{equation}

The important formulae (\ref{Qbh}) and (\ref{hhQ}) with new parameters, now
take the more legible forms, (given and proved originally in the form
presented below in \cite{bms}), namely, we have%

\begin{align}
P_{n}\left(  x|y,\rho,q\right)   &  =\sum_{j=0}^{n}%
%TCIMACRO{\QATOPD{[}{]}{n}{j}}%
%BeginExpansion
\genfrac{[}{]}{0pt}{}{n}{j}%
%EndExpansion
_{q}\rho^{n-j}B_{n-j}\left(  y|q\right)  H_{j}\left(  x|q\right)
,\label{PnaH}\\
H_{n}\left(  x|q\right)   &  =\sum_{j=0}^{n}%
%TCIMACRO{\QATOPD{[}{]}{n}{j}}%
%BeginExpansion
\genfrac{[}{]}{0pt}{}{n}{j}%
%EndExpansion
_{q}\rho^{n-j}H_{n-j}\left(  y|q\right)  P_{j}\left(  x|y,\rho,q\right)
\text{.} \label{HnaP}%
\end{align}

Polynomials $\{P_{n}\}_{n\geq-1}$ have many properties important for different
applications. Among others we have%

\begin{align}
P_{n}\left(  x|y,\rho,q\right)   &  =\sum_{j=0}^{n}%
%TCIMACRO{\QATOPD{[}{]}{n}{j}}%
%BeginExpansion
\genfrac{[}{]}{0pt}{}{n}{j}%
%EndExpansion
_{q}r^{n-j}P_{j}\left(  x|z,r,q\right)  P_{n-j}(z|y,\rho/r,q)\text{,}%
\label{PnaP}\\
\frac{P_{n}\left(  y|z,t,q\right)  }{(t^{2})_{n}}\allowbreak &  =\allowbreak
\sum_{j=0}^{n}(-1)^{j}q^{\binom{j}{2}}%
%TCIMACRO{\QATOPD{[}{]}{n}{j}}%
%BeginExpansion
\genfrac{[}{]}{0pt}{}{n}{j}%
%EndExpansion
_{q}t^{j}H_{n-j}\left(  y|q\right)  \frac{P_{j}\left(  z|y,t,q\right)
}{\left(  t^{2}\right)  _{j}}\text{,} \label{odwrocenie}%
\end{align}
if one extends the definition of polynomials $P_{n}$ for $\left\vert
\rho\right\vert >1$ by (\ref{PnaH}). (\ref{PnaP}) has been proved in
\cite{SzablKer}, while (\ref{odwrocenie}) is given in \cite{Szab6} Corollary
2. Besides, it follows directly from one of the infinite expansions that will
be presented in section \ref{nieskon}.

\begin{remark}
Notice that both (\ref{PnaP}) and (\ref{odwrocenie}) provide nontrivial
identities satisfied by the Hermite polynomials if one sets $q\allowbreak
=\allowbreak1,$ applies \ref{Pdlaq=1} and (\ref{q1}).
\end{remark}

Modifying formula for $f_{Q}$ and taking into account (\ref{n_form}), we end
up with the following one:
\begin{equation}
f_{CN}\left(  x|y,\rho,q\right)  =f_{N}(x|q)\frac{(\rho^{2})_{\infty}%
}{W(x,y|\rho,q)}, \label{fCN}%
\end{equation}
where
\begin{equation}
W(x,y|\rho,q)\allowbreak=\allowbreak\prod_{k=0}^{\infty}w\left(  x\sqrt
{1-q}/2,x\sqrt{1-q}/2|\rho q^{k}\right)  \allowbreak. \label{WW}%
\end{equation}
Recall that $w$ was given by (\ref{wxya}) and also that now we have
\[
w(x\sqrt{1-q}/2,x\sqrt{1-q}/2|t)=(1-t^{2})^{2}-(1-q)xyt(1+t^{2})+(1-q)t^{2}%
(x^{2}+y^{2}).
\]
Notice also, that, on the way, we took into account of the last statement of
the Remark \ref{simpl}.

Of course, modifying (\ref{QnQm}) we get
\begin{equation}
\int_{S\left(  q\right)  }P_{n}(x|y,\rho,q)P_{m}\left(  x|y,\rho,q\right)
f_{CN}\left(  x|y,\rho,q\right)  dx=\left\{
\begin{array}
[c]{ccc}%
0 & \text{if} & m\neq n\\
\left[  n\right]  _{q}!\left(  \rho^{2}\right)  _{n} & \text{if} & m=n
\end{array}
\right.  . \label{PnPm}%
\end{equation}

Moreover, one can deduce from (\ref{Pdlaq=1}), that
\begin{equation}
f_{CN}(x|y,\rho,q)\rightarrow\frac{1}{\sqrt{2\pi\left(  1-\rho^{2}\right)  }%
}\exp\left(  -\frac{\left(  x-\rho y\right)  ^{2}}{2\left(  1-\rho^{2}\right)
}\right)  , \label{limfcn}%
\end{equation}
as $q\rightarrow1^{-},$ which is the density of the conditional distribution
$(X|Y=y)$ of the $2-$ dimensional normal distribution of $(X,Y)$ with
$var(X)\allowbreak=\allowbreak var(Y)=1$ and $cov(X,Y)\allowbreak
=\allowbreak\rho$. That is why we call the distribution with the density
$f_{CN}$ $q-$conditional normal.

This is also the reason why we set as $f_{CN}(x|y,\rho,1)$ the right-hand side
of (\ref{limfcn}).

Now recall formula (\ref{mehler}) with parameters $y$ and $\rho$ instead of
$a$ and $b$. It is easy to notice that
\[
S_{n}(a,b)=\rho^{n}H_{n}(y),
\]
and consequently that
\begin{equation}
f_{CN}(x|y,\rho,q)\allowbreak=\allowbreak f_{N}(x|q)\sum_{n\geq0}\frac
{\rho^{n}}{\left[  n\right]  _{q}!}H_{n}(x|q)H_{n}(y|q), \label{P-M}%
\end{equation}
where $\left[  n\right]  _{q}!\allowbreak=\allowbreak(q)_{n}/(1-q)^{n}$. This
is a very important formula called Poisson-Mehler formula. There exist many
alternative proofs of it mentioned, e.g., in \cite{SzabP-M}.

One has to mention the following result showing that polynomials $P_{n}$ have
somehow specific properties that are not simply reflected by the properties of
polynomials $Q_{n}$. Namely, the following Lemma was proved in \cite{Szab6}.

\begin{lemma}
\label{UogCarl}For $x,y\in S\left(  q\right)  ,$ $\left\vert \rho\right\vert
<1$ let us denote for $m,k\geq0$
\[
\gamma_{m,k}\left(  x,y|\rho,q\right)  \allowbreak=\allowbreak\sum
_{j=0}^{\infty}\frac{\rho^{j}}{\left[  j\right]  _{q}!}H_{j+m}\left(
x|q\right)  H_{j+k}\left(  y|q\right)  \allowbreak.
\]
Then
\begin{equation}
\gamma_{m,k}\left(  x,y|\rho,q\right)  \allowbreak=\allowbreak\gamma
_{0,0}\left(  x,y|\rho,q\right)  \Xi_{m,k}\left(  x,y|\rho,q\right)  ,
\label{gamma_m_k}%
\end{equation}
where $\Xi_{m,k}$ is a polynomial in $x$ and $y$ of order at most $m+k$.
\newline Further, let us denote
\[
D_{n}\left(  x,y|\rho_{1},\rho_{2},\rho_{3},q\right)  \allowbreak
=\allowbreak\sum_{k=0}^{n}%
%TCIMACRO{\QATOPD{[}{]}{n}{k}}%
%BeginExpansion
\genfrac{[}{]}{0pt}{}{n}{k}%
%EndExpansion
_{q}\rho_{1}^{n-k}\rho_{2}^{k}\Xi_{n-k},_{k}\left(  x,y|\rho_{3,}q\right)  .
\]
Then, we have:

i) $\Xi_{m,k}\left(  x,y|\rho,q\right)  \allowbreak=\allowbreak\Xi
_{k,m}\left(  y,x|\rho,q\right)  $ and
\[
\Xi_{m,k}\left(  x,y|\rho,q\right)  \allowbreak=\allowbreak\sum_{s=0}%
^{k}(-1)^{s}q^{\binom{s}{2}}%
%TCIMACRO{\QATOPD{[}{]}{k}{s}}%
%BeginExpansion
\genfrac{[}{]}{0pt}{}{k}{s}%
%EndExpansion
_{q}\rho^{s}H_{k-s}\left(  y|q\right)  P_{m+s}(x|y,\rho,q)/(\rho^{2})_{m+s},
\]

ii)
\begin{equation}
D_{n}\left(  x,y|\rho_{1},\rho_{2},\rho_{3},q\right)  \allowbreak
=\allowbreak\sum_{s=0}^{n}%
%TCIMACRO{\QATOPD{[}{]}{n}{s}}%
%BeginExpansion
\genfrac{[}{]}{0pt}{}{n}{s}%
%EndExpansion
_{q}H_{n-s}\left(  y|q\right)  P_{s}\left(  x|y,\rho_{3},q\right)  \rho
_{1}^{n-s}\rho_{2}^{s}\left(  \rho_{1}\rho_{3}/\rho_{2}\right)  _{s}/\left(
\rho_{3}^{2}\right)  _{s}. \label{_C}%
\end{equation}

\end{lemma}

\begin{remark}
Recently in \cite{SzabCheb} (Proposition 3.6, iii)) the compact and more
legible form for the polynomials $\left\{  \Xi_{m,k}\right\}  $ was obtained
for $q\allowbreak=\allowbreak0$. Also a generalization of the formula
\ref{gamma_m_k} was obtained. A generalization in the sense that polynomials
$\left\{  H_{n}(x|q)\right\}  $ are replaced by polynomials $U_{n}(x/2)$ (this
refers to the case $q\allowbreak=\allowbreak0$) but also when polynomials
$\left\{  H_{n}(x|q)\right\}  $ are replaced by polynomials $T_{n}(x)$ that is
linear combinations of polynomials $\left\{  U_{n}(x)\right\}  $ and
$q\allowbreak=\allowbreak0$ (compare (\ref{TnaU})).
\end{remark}

Polynomials $D_{n}$ were used in an unsuccessful attempt to generalize the
Kibble-Slepian formula presented below. However, quite successful was an
attempt to generalize some other properties of the Normal distributions.

It is worth to mention the following formula defining polynomials $\Phi_{k,m}$
which is obtained from (\ref{gamma_m_k}) by setting $y\allowbreak=\allowbreak
x$:
\begin{equation}
\sum_{i\geq0}\frac{r^{i}}{[i]_{q}!}H_{i+k}(x|q)H_{i+m}(x|q)=\Phi
_{k,m}(x|r,q)\times\sum_{i\geq0}\frac{r^{i}}{[i]_{q}!}H_{i}(x|q)H_{i}(x|q),
\label{simp}%
\end{equation}
where polynomials $\left\{  \Phi_{k,m}(x|r,q)\right\}  $ are given by the
following formula:
\begin{equation}
\Phi_{k,m}(x|r,q)=\sum_{s=0}^{k}\frac{q^{\binom{s}{2}}(-r)^{s}(r)_{m+s}%
}{(r^{2})_{m+s}}H_{k-s}(x|q)R_{m+s}(x|r,q), \label{ww}%
\end{equation}
for $k,m\geq0$. From the definition above, it follows directly that,
$\Phi_{k,m}(x|r,q)=\Phi_{m,k}(x|r,q)\text{,}$ $k,m\geq0$.

Using formula (\ref{P-M}) one can show, as it was done in \cite{bms}, that:
\begin{equation}
\int_{S\left(  q\right)  }f_{CN}\left(  z|y,\rho_{1},q\right)  f_{CN}\left(
y|x,\rho_{2},q\right)  dy=f_{CN}\left(  x|z,\rho_{1}\rho_{2},q\right)  ,
\label{ChK}%
\end{equation}
which is nothing else but the so-called Chapman-Kolmogorov property of the
$q-\allowbreak$conditional normal distribution.

Now, let us consider polynomials
\[
A_{n}\left(  x|y,\rho_{1},z,\rho_{2},q\right)  \allowbreak=\allowbreak
\alpha_{n}\left(  x\sqrt{1-q}/2|y\sqrt{1-q}/2,\rho_{1},z\sqrt{1-q}/2,\rho
_{2},q\right)  /\left(  1-q\right)  ^{n/2},
\]
where $\alpha_{n}$ is defined by its three-term recurrence (\ref{_aw}), in
other words they are the classical Askey--Wilson polynomials. In
\cite{Szab-bAW} (4.3) the three-term recurrence satisfied by the so modified
AW\ polynomials is given. It is complicated and we will not need it. The more
important is an observation made also in \cite{Szab-bAW} (4.9) that the
modified AW density $f_{C2W}(x|y,\rho_{1},z,\rho_{2},q)$ (given originally by
(\ref{fAW}) ) can be presented in the following way:%
\[
f_{C2N}\left(  x|y,\rho_{1},z,\rho_{2},q\right)  \allowbreak=\allowbreak
\frac{f_{CN}\left(  y|x,\rho_{1},q\right)  f_{CN}\left(  x|z,\rho
_{2},q\right)  }{f_{CN}\left(  y|z,\rho_{1}\rho_{2},q\right)  }.
\]

Having this, we are ready for the first probabilistic model. Below, where we
will present the probabilistic application of the polynomials mentioned above,
we will use the following notation, being used traditionally in the
probability theory. Namely, $X\sim\ast$ means that random the variable $X$ has
a distribution that is denoted by $\ast$. For example, $X$ $\sim$
$N(m,\gamma)$ means that the random variable $X$ has a normal distribution
with expectation $m$ and variance $\gamma,$ or $X\sim f$ has a distribution
that has density $f$.

\section{Finite Markov chain}

Consider three random variables $(X_{i})_{i=1,2,3}$ forming a stationary,
finite Markov chain. More precisely, that $X_{1}$ $\sim$ $f_{N}(.|q)$,
further, let's assume that the conditional distribution of $X_{2}|X_{1}=y$ is
$q$-CN with the density $f_{CN}(.|y,\rho_{1},q)$ and $X_{3}|X_{2}=y$ is $q-$CN
with the density $f_{CN}(.|y,\rho_{2},q)$. Hence the joint density of say
$(X_{1},X_{2})$ is $f_{2}(y,x)\allowbreak=\allowbreak$ $f_{N}(y|q)f_{CN}%
(x|y,\rho,q)$. Now recall formula (\ref{fCN}) and we see that
\[
\int_{S(q)}f_{2}(x,y)dy\allowbreak=\allowbreak f_{N}(x|q).
\]
Probabilistic properties of the distribution with the density $f_{CN}%
(.|y,\rho_{2},q)$ were thoroughly analyzed in \cite{Szab5}. That is, that the
marginal distribution of $X_{2}$ is the same as that of $X_{1}$. Similarly we
show that $X_{3}$ $\sim$ $X_{2}$ $\sim X_{1}$. From (\ref{ChK}), it follows
that the joint distribution of $(X_{1},X_{3})$ has a density equal to
$f_{N}(x|q)f_{CN}(z|x,\rho_{1}\rho_{2},q)$ and of course that joint density of
$(X_{1},X_{2},X_{3})$ is given by%
\[
f_{N}(x|q)f_{CN}(y|x,\rho_{1},q)f_{CN}(z|y,\rho_{2},q).
\]
Consequently, the conditional density of $X_{2}|X_{1}=x,X_{3}=z$ is equal to
the ratio of the joint density of $(X_{1},X_{2},X_{3})$ and a marginal density
of $(X_{1},X_{3})$. Hence, consequently
\[
X_{2}|X_{1}=x,X_{3}=z\sim\frac{f_{CN}(y|x,\rho_{1},q)f_{CN}(z|y,\rho_{2}%
,q)}{f_{CN}(z|x,\rho_{1}\rho_{2},q)}=f_{C2N}(y|x,\rho_{1},z,\rho_{2}).
\]
Thus $q-$Hermite polynomials $\left\{  H_{n}(x|q)\right\}  $ are the
polynomials that are orthogonal with respect to the marginal densities of our
finite Markov chain. The modified ASC polynomials $\left\{  P_{n}(x|y,\rho
_{k},q)\right\}  _{n\geq-1}$ with $\rho_{k}$ being equal $\rho_{1},$ $\rho
_{2}$ or $\rho_{1}\rho_{2},$ are the polynomials that are orthogonal with
respect to the conditional $X_{i}|X_{j}\allowbreak=\allowbreak y$ ($i\neq j)$
densities of our chain. Finally, the modified AW $\left\{  A_{n}(x|y,\rho
_{1},z,\rho_{2})\right\}  $ polynomials are the ones that are orthogonal with
respect to the conditional densities $X_{2}|X_{1}=y,X_{2}=z$.

Of course, one can generalize this example by considering consider
continuous-time Markov processes.

In fact, the first probabilistic model where the $q-$Hermite polynomials
appeared, was defined by W. Bryc in \cite{bryc1}. It was called the stationary
Markov field (that is, time-symmetric, discrete-time, stationary Markov
process). There, it was shown that the defined in the paper, Markov field is
stationary and has marginal distribution with the density $f_{N}$ and has the
property that
\[
E(H_{m}(X_{n+k}|q)|X_{n}\allowbreak=\allowbreak x)\allowbreak=\allowbreak
\rho^{k}H_{m}(x|q),
\]
for all nonnegative integers $m,n,k$.

What has surprised R. Askey in the description of this Markov chain (i.e.,
process with discrete-time), is the fact that in the original probabilistic
description of the field there is no parameter $q$. In fact, it was defined as
some (quite nontrivial at the first sight) function of the parameters that
were used in the description of the field. In \cite{bms}, the conditional
distribution and the polynomials that are orthogonal with respect to it, were
identified. The fact that the conditional distribution $X_{2}|X_{1}=y,X_{2}=z$
is AW was shown in \cite{Szab6}.

In \cite{bryc1} and \cite{bms} it was shown, that Bryc's random field can
exist for $q\geq-1$ and $\left\vert \rho\right\vert <1$. Since $\rho
\allowbreak=\allowbreak0$ leads to the trivial case of independent random
variables we exclude this case from considerations.

For $q\in\lbrack-1,1]$ the one-dimensional distributions are uniquely defined
and depend only on $q$.

In particular, for $q\allowbreak=\allowbreak-1$ this marginal distribution is
discrete symmetric with support on $\left\{  -1,1\right\}  $. Similarly, the
conditional (transitional) distribution is also discrete supported on
$\left\{  -1,1\right\}  $ with the transitional probabilities depending on
$\rho$. These facts were signaled in \cite{bryc1} and developed and
generalized in \cite{matszab}, \cite{matszab2} and of course \cite{bms}.

For $-1<q<1$ the marginal distributions have densities and the bounded support
(these distributions are, in fact, the $q-$Normal distributions introduced by
Bo\.{z}ejko et al. in \cite{Bo}), for $q\allowbreak=\allowbreak1$ the marginal
distribution is the standard normal $N\left(  0,1\right)  $. Conditional
distributions are also uniquely defined and depend on $q$ and $\rho$. For
$-1<q\leq1$ and $\rho\in(-1,1)\backslash\left\{  0\right\}  $ these
conditional distributions have densities (see e.g., \cite{bms}). In
particular, for $q\allowbreak=\allowbreak1$ such conditional distribution of
say $X_{k+1}$ under condition $X_{k}\allowbreak=\allowbreak y$ is the normal
distribution $N\left(  \rho y,1-\rho^{2}\right)  $.

The existence of such random fields for $q>1$ was an open question. It was
answered positively in \cite{Szab08}. For $q>1$ the one-dimensional
distributions are not uniquely defined. These distributions have only known
moments of all orders. It turns out that the conditional distributions do not
exist for all pairs $\left(  q,\rho\right)  $. They might exist only if
$\rho^{2}\in\left\{  \frac{1}{q},\frac{1}{q^{2}},\ldots\right\}  $.

For $q\allowbreak>\allowbreak1$ and $\rho^{2}\allowbreak=\allowbreak\frac
{1}{q^{m-1}}$ (see \cite{bms}) the conditional distribution of $X_{n+1}%
|X_{n}=y$ is concentrated on zeros $\left\{  \chi_{j}\left(  y,q\right)
\right\}  _{j=1}^{m}$ of the Al-Salam-Chihara polynomial $P_{m}\left(
x|y,\rho,q\right)  $ defined above (more precisely by (\ref{3Pn})). Moreover,
the masses $\left\{  \lambda_{i}\right\}  _{i=1}^{m}$ assigned to these zeros,
are defined by the equalities:
\[
\sum_{j=1}^{m}\lambda_{j}=1,\sum_{j=1}^{m}\lambda_{j}P_{k}\left(  \chi
_{j}\left(  y,q\right)  |y,\rho,q\right)  \allowbreak=\allowbreak0,
\]
for $\allowbreak k\allowbreak=\allowbreak1,\ldots m-1$.

The problem was if the defined in this way, discrete conditional distributions
satisfy the Chapman--Kolmogorov equation. Following \cite{Szab08}, it turned
out that yes, they do. The proof heavily depends on the adopted to the present
setting, formulae \ref{IsmStan} and \ref{simIsm}. The referee of \cite{Szab08}
was really surprised that a very abstract formula (i.e.,\ref{IsmStan}), has
found its probabilistic application.

\section{Attempts to generalize Gaussian distributions and processes.}

In this subsection, we expose one or more important property of the Gaussian
process or distribution and indicate, so to say, a $q-$ version of the process
or distribution that has similar properties.

\subsection{Three dimensional distributions generalizing some properties of
Normal distributions.}

In this subsection, we try to generalize the following property of the
Gaussian distributions. Let $\mathbb{R}^{n}\ni\mathbf{X\sim}N(\mathbf{m,\Sigma
})$ be a random vector having Normal distribution with parameters
$\mathbf{m\allowbreak=\allowbreak}E\mathbf{X}$ and variance-covariance matrix
$\mathbf{\Sigma\allowbreak=\allowbreak}E(\mathbf{X-m)(X-m)}^{T}$. Assume that
$\mathbf{X}^{T}\allowbreak=\allowbreak(X_{1},\ldots,X_{n})$. Then, it turns
out that for every subset $I\allowbreak\subset\allowbreak\{1,2,\ldots,n\}$ and
non-negative integers $\left\{  i_{j}\right\}  _{j\in I} $ the conditional
expectation
\[
E(\prod_{j\in I}X_{j}^{i_{j}}|\{X_{k}\}_{k\notin I})
\]
is a polynomial in variables $\{X_{k}\}_{k\notin I}$ of order not exceeding
$\sum_{j\in I}i_{j}$. We will call such property PCM($n$) (polynomial
conditional moment property of $n-$dimensional distribution). So far, it is
known that only Normal distributions have this property for every integer $n.
$ Below, we will see that it is possible to construct non-Normal distribution
having this property, unfortunately, so far, only for dimension $3$.

As shown in \cite{Szab19}, the distribution that has PCM($3$) property is the
following defined for some $\rho_{12},\rho_{13},\rho_{23}\allowbreak
\in\allowbreak(-1,1)$
\begin{gather*}
f_{3D}(x,y,z|\rho_{12},\rho_{13},\rho_{23},q)=f_{N}\left(  x|q\right)
f_{N}(y|q)f_{N}(z|q)\\
\times\frac{C_{3D}\left(  \rho_{12}^{2}\right)  _{\infty}\left(  \rho_{13}%
^{2}\right)  _{\infty}\left(  \rho_{23}^{2}\right)  _{\infty}}{W(x,y|\rho
_{12},q)W(x,z|\rho_{13},q)W(y,z|\rho_{23},q)}\\
=C_{3D}f_{CN}(x|y,\rho_{12},q)f_{CN}(y|z,\rho_{23},q)f_{CN}(z|x,\rho_{13},q),
\end{gather*}
where $W(x,y|\rho,q)$ is defined by (\ref{WW}) and $C_{3D}$ is the suitably
chosen constant. To proceed further, let us denote also by $r\allowbreak
=\allowbreak\rho_{12}\rho_{23}\rho_{13}$. The first result proved in
\cite{Szab19} is the following theorem presenting marginal distributions.

\begin{theorem}
\label{margin}Let us denote for simplicity $r=\rho_{12}\rho_{13}\rho
_{23}\text{.}$ Then

i) $C_{3D}=1-r$.

ii) Two-dimensional marginals depend, in fact, on two parameters (except for
$q)$ . In the case of $f_{YZ}$ it depends on $\rho_{23}$ and $\rho_{12}%
\rho_{13}$ only. Namely, we have
\begin{gather*}
f_{YZ}(y,z|\rho_{12},\rho_{13},\rho_{23},q)=\int_{S(q)}f_{3D}(x,y,z|\rho
_{12},\rho_{13},\rho_{23},q)dx=\\
(1-r)f_{N}(y|q)f_{N}(z|q)\frac{\left(  \rho_{23}^{2}\right)  _{\infty}%
(\rho_{12}^{2}\rho_{13}^{2})_{\infty}}{\prod_{i=0}^{\infty}\omega\left(
y,z|\rho_{23}q^{i}\right)  \omega\left(  y,z|\rho_{12}\rho_{13}q^{i}\right)
}\\
=(1-r)f_{CN}(y|z,\rho_{23},q)f_{CN}(z|y,\rho_{12}\rho_{13},q)
\end{gather*}
and similarly for $f_{XZ}\text{,}$ and $f_{XY}\text{.}$

iii) Marginal one-dimensional densities $\int_{S(q)}\int_{S(q)}f_{3D}%
(x,y,z|\rho_{12},\rho_{13},\rho_{23},q)dxdy\allowbreak=\allowbreak
f_{Z}(z|\rho_{12},\rho_{13},\rho_{23},q)$ depend on the product $r=\rho
_{12}\rho_{23}\rho_{13}$ only. Moreover we have $f_{Z}(z|\rho_{12},\rho
_{13},\rho_{23},q)\allowbreak=\allowbreak f_{R}(z|r,q)$, where $f_{R}$ is a
Rogers distribution given by (\ref{fR}).
\end{theorem}

\begin{remark}
This result provides the first example of the probabilistic interpretation of
the $q-$ultraspherical density and polynomials known so far in the literature.
\end{remark}

The second result presented in \cite{Szab19} concerns the conditional moments.
Recall that in fact, we have $9$ types of conditional moments. Three of them
are of the form $E(X^{n}|Y=y,Z=z)$ (in fact, we will find $E(H_{n}%
(X)|Y=y,Z=z)$) and similarly for the remaining choices of the conditioned
random variable. Next $3$ of them are of the form $E(X^{n}Y^{m}|Z=z)$ (again
we will find $E(H_{n}(X)H_{m}(Y)|Z=z)$) and similarly the other choices of the
conditioning random variable. Finally, we have $3$ conditional moments of the
form $E(H_{n}(X)|Y=y)$ and similarly the other two choices of the variables.
We have the following results.

\begin{theorem}
\label{c2d}One-dimensional conditional moments say $E(H_{n}(Y|q)|Z=z)$ are
polynomials of order not exceeding $n$ in $Z\text{.}$ More precisely for
$n=2m+1$ we have
\[
E(H_{2m+1}(Y|q)|Z=z)=\sum_{s=0}^{m}%
%TCIMACRO{\QATOPD{[}{]}{2m+1}{s}}%
%BeginExpansion
\genfrac{[}{]}{0pt}{}{2m+1}{s}%
%EndExpansion
_{q}(\rho_{23}^{2m+1-s}+(\rho_{12}\rho_{13})^{2m+1-s})\Phi_{s,2m+1-s}%
(z|r,q)\text{,}%
\]
and for $n=2m\text{,}$ $m\geq1$ we have:
\begin{gather*}
E(H_{2m}(Y|q)|Z=z)=%
%TCIMACRO{\QATOPD{[}{]}{2m}{m}}%
%BeginExpansion
\genfrac{[}{]}{0pt}{}{2m}{m}%
%EndExpansion
_{q}r^{m}\Phi_{m,m}(z|r,q)\\
+\sum_{s=0}^{m-1}%
%TCIMACRO{\QATOPD{[}{]}{2m}{s}}%
%BeginExpansion
\genfrac{[}{]}{0pt}{}{2m}{s}%
%EndExpansion
_{q}(\rho_{23}^{2m-s}+(\rho_{12}\rho_{13})^{2m-s})\Phi_{s,2m-s}(z|r,q)\text{,}%
\end{gather*}
where polynomials $\left\{  \Phi_{m,k}(z|r,q)\right\}  _{m,k\geq0}$ are given
by (\ref{ww}).

In particular, we have:%

\begin{equation}
E(Y|Z=z)=\frac{(\rho_{23}+\rho_{12}\rho_{13})}{(1+r)}z\text{,} \label{cyz}%
\end{equation}
and%

\begin{equation}
E(Y^{2}|Z=z)=\frac{(\rho_{23}^{2}+\rho_{12}^{2}\rho_{13}^{2}%
)(1-qr)+r(1-r)(1+q)}{(1+r)(1-qr^{2})}z^{2}+\frac{1+r^{2}-\rho_{23}^{2}%
-\rho_{12}^{2}\rho_{13}^{2}}{(1-qr^{2})}\text{.} \label{cy2z}%
\end{equation}

\end{theorem}

\begin{theorem}
\label{2na1}$E(H_{n}(X|q)|Y=y,Z=z)$ assumes one of the following equivalent forms:

i)
\begin{gather}
E(H_{n}(X|q)|Y=y,Z=z)\label{1na2_1}\\
=\sum_{s=0}^{n}%
%TCIMACRO{\QATOPD{[}{]}{n}{s}}%
%BeginExpansion
\genfrac{[}{]}{0pt}{}{n}{s}%
%EndExpansion
_{q}\rho_{12}^{n-s}\rho_{13}^{s}\left(  \rho_{12}^{2}\right)  _{s}%
H_{n-s}\left(  y|q\right)  P_{s}\left(  z|y,\rho_{12}\rho_{13},q\right)
/(\rho_{12}^{2}\rho_{13}^{2})_{s}\text{,}\nonumber
\end{gather}
where $\left\{  P_{s}\left(  z|y,\rho_{1}\rho_{2},q\right)  \right\}
_{s\geq-1}$ constitute the Al-Salam--Chihara polynomials with new parameters
defined by the three-term recurrence (\ref{3Pn}).

ii)
\begin{gather}
E(H_{n}(X|q)|Y=y,Z=z)=\frac{1}{\left(  \rho_{12}^{2}\rho_{13}^{2}\right)
_{n}}\sum_{k=0}^{\left\lfloor n/2\right\rfloor }(-1)^{k}q^{\binom{k}{2}}%
%TCIMACRO{\QATOPD{[}{]}{n}{2k}}%
%BeginExpansion
\genfrac{[}{]}{0pt}{}{n}{2k}%
%EndExpansion
_{q}%
%TCIMACRO{\QATOPD{[}{]}{2k}{k}}%
%BeginExpansion
\genfrac{[}{]}{0pt}{}{2k}{k}%
%EndExpansion
_{q}\left[  k\right]  _{q}!\rho_{13}^{2k}\rho_{12}^{2k}\label{1na2_2}\\
\times\left(  \rho_{12}^{2},\rho_{13}^{2}\right)  _{k}\sum_{j=0}^{n-2k}%
%TCIMACRO{\QATOPD{[}{]}{n-2k}{j}}%
%BeginExpansion
\genfrac{[}{]}{0pt}{}{n-2k}{j}%
%EndExpansion
_{q}\left(  \rho_{12}^{2}q^{k}\right)  _{j}\left(  \rho_{13}^{2}q^{k}\right)
_{n-2k-j}\rho_{12}^{n-2k-j}\rho_{13}^{j}H_{j}\left(  z|q\right)
H_{n-2k-j}(y|q)\text{.}\nonumber
\end{gather}
iii) $E(P_{n}(X|y,\rho_{12},q)|Y=y,Z=z)=\frac{\rho_{13}^{n}(\rho_{12}^{2}%
)_{n}}{(\rho_{12}^{2}\rho_{13}^{2})_{n}}P_{n}(z|y,\rho_{12}\rho_{13},q)$.

In particular, we have
\begin{equation}
E(X|Y=y,Z=z)=\frac{y\rho_{12}(1-\rho_{13}^{2})+z\rho_{13}(1-\rho_{12}^{2}%
)}{1-\rho_{12}^{2}\rho_{13}^{2}}\text{.} \label{ex}%
\end{equation}

\end{theorem}

Finally, we have the following corollary.

\begin{corollary}
$\forall n,m\geq0:E(H_{n}(X|q)H_{m}(Y|q)|Z)$ is a polynomial of order at most
$n+m$ of the conditioning random variable $Z$.

In particular, we have:%
\begin{align}
E(XY|Z  &  =z)=z^{2}\frac{\rho_{12}(\rho_{13}^{2}+\rho_{23}^{2}%
)(1-qr)+(1-r)(\rho_{13}\rho_{23}+qr\rho_{12})}{(1+r)(1-qr^{2})}\label{cconv}\\
&  +\frac{\rho_{12}(1-\rho_{13}^{2})(1-\rho_{23}^{2})}{(1-qr^{2})}%
\text{.}\nonumber
\end{align}

\end{corollary}

Maybe it is worth to recall that the conditional densities are the following:%
\begin{gather*}
f_{X|Y,Z}(x|y,z,\rho_{12},\rho_{13},\rho_{23},q)=\frac{f_{3D}(x,y,z|\rho
_{12},\rho_{13},\rho_{23},q)}{f_{YZ}(y,z|\rho_{12},\rho_{13},\rho_{23},q)}\\
=\frac{(\rho_{12}^{2})_{\infty}\left(  \rho_{13}^{2}\right)  _{\infty}}%
{(\rho_{12}^{2}\rho_{13}^{2})_{\infty}}f_{X}(x|q)\frac{W(y,z|\rho_{12}%
\rho_{13},q)}{W(x,y|\rho_{12},q)W(x,z|\rho_{13},q)}\\
=\frac{f_{CN}(x|y,\rho_{12},q)f_{CN}(z|x,\rho_{13},q)}{f_{CN}(y|z,\rho
_{12}\rho_{13},q)}=f_{C2N}(x|y,\rho_{12},z,\rho_{13},q)\text{.}%
\end{gather*}
In other words, we have
\[
E(A_{m}(X|y,\rho_{12},z,\rho_{2},q)A_{n}\left(  X|y,\rho_{12},z,\rho
_{2},q\right)  |Y=y,Z=z)=0,
\]
for all $n,m\geq1,$ $n\neq m,$ where $A(x|y,\rho_{1},z,\rho_{2})$ is the AW
polynomial with parameters $y,\rho_{1},z,\rho_{2}$.

Similarly, as far as the other conditional densities are concerned, we have%
\begin{align*}
f_{XY|Z}(x,y|z,\rho_{12},\rho_{13},\rho_{23},q)  &  =\frac{f_{3D}%
(x,y,z|\rho_{12},\rho_{13},\rho_{23},q)}{\allowbreak f_{R}(z|r,q)}\text{,}\\
f_{X|Y}(x|y,\rho_{12},\rho_{13},\rho_{23},q)  &  =\frac{f_{XY}(x|y,\rho
_{12},\rho_{13}\rho_{23},q)}{f_{Y}(y|\rho_{12}\rho_{13}\rho_{23},q)}\text{.}%
\end{align*}

\subsection{$q-$Wiener stochastic process}

It is known that the Wiener process $(Y_{t})_{t\geq0}$ is a stochastic process
defined on $[0,\infty)$ that has i) independent increments, ii) Gaussian
marginal distributions $N(0,\sigma^{2}t)$, iii)
\begin{align}
E\left(  t^{n/2}H_{n}\left(  _{n}\frac{Y_{t}}{\sigma\sqrt{t}}\right)
|\mathcal{F}_{\leq s}^{Y}\right)  \allowbreak &  =\allowbreak s^{n/2}%
H_{n}\left(  \frac{Y_{s}}{\sigma\sqrt{s}}\right)  \text{,}\label{marG}\\
E\left(  s^{-n/2}H_{n}\left(  \frac{Y_{s}}{\sigma\sqrt{s}}\right)
|\mathcal{F}_{\geq t}^{Y}\right)   &  =t^{-n/2}H_{n}\left(  \frac{Y_{t}%
}{\sigma\sqrt{t}}\right)  \label{revmartG}%
\end{align}
a.s. for all $n\in\mathbb{N}$, $0\leq s<t$. Here $H_{n}$ denotes the Hermite
polynomial defined (\ref{_1}) and $\mathcal{F}_{\leq s}^{Y}$, $\mathcal{F}%
_{\geq t}^{Y}$ denote the $\sigma-$fields generated by the process $\left(
Y_{\tau}\right)  $ for $\tau\leq s$ or $\tau\geq t$ i.e., so to say, the past
or the future of the process $\left(  Y_{t}\right)  $ that happen before the
moment $s$ or past the moment $t$. Properties (\ref{marG}) and (\ref{revmartG}%
) are sometimes expressed in the following form. Namely, that the families
(indexed by $n$) of processes $\left\{  t^{n/2}H_{n}\left(  \frac{Y_{t}%
}{\sigma\sqrt{t}}\right)  \right\}  _{t\geq0}$ and $\left\{  s^{-n/2}%
H_{n}\left(  \frac{Y_{s}}{\sigma\sqrt{s}}\right)  \right\}  _{s\geq0}$ are
respectively called martingales and reversed martingales.

To simplify the further description let us assume $\sigma=1$. Now, in
\cite{Szab-OU-W}, for every $\left\vert q\right\vert <1,$ there has been
defined Markov stochastic process $\left(  X_{t}\right)  _{t\geq0}$ such that:

i) each of its marginal distributions, say, $X_{t}$ has density $\frac
{1}{\sqrt{\tau}}f_{N}\left(  \frac{x}{\sqrt{\tau}}|q\right)  ,$

ii) for every $n\geq1$ the following stochastic processes $\left\{
t^{n/2}H_{n}\left(  \frac{Y_{t}}{\sqrt{t}}|q\right)  \right\}  _{t\geq0}$ and
$\left\{  s^{-n/2}H_{n}\left(  \frac{Y_{s}}{\sigma\sqrt{s}}|q\right)
\right\}  _{s\geq0}$ are repetitively martingale and reversed martingale.

That is (\ref{marG}) and (\ref{revmartG}) are satisfied where the polynomials
$H_{n}(x)$ are substituted by $H_{n}(x|q)$. For the probabilists, it is
obvious that the increments of this process cannot be independent unless
$q\allowbreak=\allowbreak1$, which is when we deal with the Gaussian case.

In \cite{Szab-OU-W} there has been defined, the other, related process called
$q$-Ornstein--Uhleneck process. Both these processes provide examples of
Markov processes that have all conditional moments of order $n$ being
polynomials of the condition of the order not exceeding $n$. The theory of
such processes has been developed in the series of papers \cite{SzablPoly},
\cite{SzabStac}, \cite{Szab17}. It should be mentioned that together with the
$q-$Normal distribution the initial idea to define $q-$Wiener process can be
traced to the pioneering works of Bo\.{z}ejko et al. \cite{Bo} in 1997, W.
Bryc \cite{bryc1} in 2001, later developed in \cite{BryWes10} and
\cite{BryMaWe}.

\subsection{Generalization of Kibble--Slepian formula}

Recall that Kibble in 1949 \cite{Kibble} and independently Slepian in 1972
\cite{Slepian72} extended the Mehler's formula to higher dimensions. Namely,
they expanded the ratio of the standardized multidimensional Gaussian density
divided by the product of one-dimensional marginal densities in the multiple
sums involving only constants (correlation coefficients) and the Hermite
polynomials. The formula, in its generality, can be found in \cite{IA} (4.7.2
p.107). Since we are going to generalize its $3$-dimensional version, only
this version will be presented here.

Namely, let us consider $3$ dimensional density $f_{3D}\left(  x_{1}%
,x_{2},x_{3}|\rho_{12},\rho_{13},\rho_{23}\right)  $ of Normal random vector
\[
N\left(  \left[
\begin{array}
[c]{c}%
0\\
0\\
0
\end{array}
\right]  ,\left[
\begin{array}
[c]{ccc}%
1 & \rho_{12} & \rho_{13}\\
\rho_{12} & 1 & \rho_{23}\\
\rho_{13} & \rho_{23} & 1
\end{array}
\right]  \right)  .
\]
Of course, we must assume that the parameters $\rho_{12},$ $\rho_{13},$
$\rho_{23}$ are such that the variance-covariance matrix is positive definite
i.e., such that $\left\vert \rho_{ij}\right\vert <1,$ $i,j=1,2,3,$ $i\neq j$
and%
\begin{equation}
1+2\rho_{12}\rho_{13}\rho_{23}-\rho_{12}^{2}-\rho_{13}^{2}-\rho_{23}^{2}>0.
\label{dod}%
\end{equation}
Then, the Kibble--Slepian formula reads that%
\begin{align*}
&  \exp\left(  \frac{x_{1}^{2}+x_{2}^{2}+x_{3}^{2}}{2}\right)  f_{3D}\left(
x_{1},x_{2},x_{3}|\rho_{12},\rho_{13},\rho_{23}\right) \\
&  =\sum_{k,m,n=0}^{\infty}\frac{\rho_{12}^{k}\rho_{13}^{m}\rho_{23}^{n}%
}{k!m!n!}H_{k+m}\left(  x_{1}\right)  H_{k+n}\left(  x_{2}\right)
H_{m+n}\left(  x_{3}\right)  .
\end{align*}

Thus, the immediate generalization of this formula would be to substitute the
Hermite polynomials by the $q-$Hermite ones and the factorials by the $q- $factorials.

The question is if such a sum is positive. It turns out that not in general,
i.e., not for all $\rho_{12},$ $\rho_{13},$ $\rho_{23}$ satisfying
(\ref{dod}). Nevertheless, it is interesting to compute the sum
\begin{equation}
\sum_{k,m,n=0}^{\infty}\frac{\rho_{12}^{k}\rho_{13}^{m}\rho_{23}^{n}}{\left[
k\right]  _{q}!\left[  m\right]  _{q}!\left[  n\right]  _{q}!}H_{k+m}\left(
x_{1}|q\right)  H_{k+n}\left(  x_{2}|q\right)  H_{m+n}\left(  x_{3}|q\right)
. \label{suma}%
\end{equation}
For simplicity let us denote this sum by $g\left(  x_{1},x_{2},x_{3}|\rho
_{12},\rho_{13},\rho_{23},q\right)  $.

In \cite{Szab-KS} the following result has been formulated and proved.

\begin{theorem}
Sum $g\left(  x_{1},x_{2},x_{3}|\rho_{12},\rho_{13},\rho_{23},q\right)  $ has
one of the following two equivalent forms:

i)%
\begin{gather}
g\left(  x_{1},x_{2},x_{3}|\rho_{12},\rho_{13},\rho_{23},q\right)
=\frac{\left(  \rho_{13}^{2}\right)  _{\infty}}{\prod_{k=0}^{\infty}%
W_{q}\left(  x_{1},x_{3}|\rho_{13}q^{k}\right)  }\nonumber\\
\times\sum_{s\geq0}\frac{1}{\left[  s\right]  _{q}!}H_{s}\left(
x_{2}|q\right)  D_{s}\left(  x_{1},x_{3}|\rho_{12},\rho_{23},\rho
_{13},q\right)  ,
\end{gather}
where $D_{n}\left(  x_{1},x_{3}|\rho_{12},\rho_{23},\rho_{13},q\right)  $ is
given by either (\ref{_C}) or can be expressed in terms of polynomials $H_{n}$
in the following form:
\begin{gather*}
D_{n}\left(  x_{1},x_{3}|\rho_{12},\rho_{23},\rho_{13},q\right)  =\frac
{1}{\left(  \rho_{13}^{2}\right)  _{n}}\\
\times\sum_{k=0}^{\left\lfloor n/2\right\rfloor }(-1)^{k}q^{\binom{k}{2}}%
%TCIMACRO{\QATOPD{[}{]}{n}{2k}}%
%BeginExpansion
\genfrac{[}{]}{0pt}{}{n}{2k}%
%EndExpansion
_{q}%
%TCIMACRO{\QATOPD{[}{]}{2k}{k}}%
%BeginExpansion
\genfrac{[}{]}{0pt}{}{2k}{k}%
%EndExpansion
_{q}\left[  k\right]  _{q}!\rho_{12}^{k}\rho_{13}^{k}\rho_{23}^{k}\left(
\frac{\rho_{12}\rho_{13}}{\rho_{23}}\right)  _{k}\left(  \frac{\rho_{13}%
\rho_{23}}{\rho_{12}}\right)  _{k}\allowbreak\\
\sum_{j=0}^{n-2k}%
%TCIMACRO{\QATOPD{[}{]}{n-2k}{j}}%
%BeginExpansion
\genfrac{[}{]}{0pt}{}{n-2k}{j}%
%EndExpansion
_{q}\rho_{23}^{j}\left(  \frac{\rho_{12}\rho_{13}}{\rho_{23}}q^{k}\right)
_{k}\rho_{12}^{n-j-2k}\left(  \frac{\rho_{13}\rho_{23}}{\rho_{12}}%
q^{k}\right)  _{n-2k-j}H_{j}\left(  x_{1}|q\right)  H_{n-2k-j}\left(
x_{3}|q\right)  ,
\end{gather*}
similarly for other pairs $\left(  1,3\right)  $ and $\left(  2,3\right)  ,$

ii)
\begin{gather}
g\left(  x_{1},x_{2},x_{3}|\rho_{12},\rho_{13},\rho_{23},q\right)
=\frac{\left(  \rho_{13}^{2},\rho_{23}^{2}\right)  _{\infty}}{\prod
_{k=0}^{\infty}W_{q}\left(  x_{1},x_{3}|\rho_{13}q^{k}\right)  W_{q}%
(x_{3},x_{2}|\rho_{23}q^{k})}\label{exp_in_ASC}\\
\times\sum_{s=0}^{\infty}\frac{\rho_{12}^{s}\left(  \rho_{13}\rho_{23}%
/\rho_{12}\right)  _{s}}{\left[  s\right]  _{q}!\left(  \rho_{13}^{2}\right)
_{s}\left(  \rho_{23}^{2}\right)  _{s}}P_{s}\left(  x_{1}|x_{3},\rho
_{13},q\right)  P_{s}\left(  x_{2}|x_{3},\rho_{23},q\right)  ,\nonumber
\end{gather}
similarly for other pairs $\left(  1,3\right)  $ and $\left(  2,3\right)  $.
\end{theorem}

Unfortunately, as shown in \cite{Szab-KS}, one can find such $\rho_{12},$
$\rho_{13},$ $\rho_{23}$ that function the $g$ with these parameters assumes
negative values for some $x_{j}\in S\left(  q\right)  $, $j\allowbreak
=\allowbreak1,2,3$; hence consequently, $g\left(  x_{1},x_{2},x_{3}|\rho
_{12},\rho_{13},\rho_{23},q\right)  \prod_{j=0}^{3}f_{N}\left(  x_{j}%
|q\right)  $ with these values of parameters is not a density of a probability distribution.

\section{Infinite expansions\label{nieskon}}

\subsection{Kernels}

In the literature, there is a small confusion concerning terminology.
Sometimes the expression of the form
\[
\sum_{n\geq0}r_{n}A_{n}\left(  x\right)  B_{n}\left(  y\right)  ,
\]
where $\left\{  A_{n}\right\}  $ and $\left\{  B_{n}\right\}  $ are the
families of polynomials, is called a kernel (like in \cite{suslov96})) or even
sometimes 'bilinear generating function' (see e.g., \cite{Rahman97})) or also
a Poisson kernel. In general, if $A_{n}$ and $B_{n}$ are different we talk
about non-symmetric kernels, if $A_{n}(x)\allowbreak=\allowbreak B_{n}(x)$
then the kernels are called symmetric or simply kernels.

The process of expressing these sums in a closed-form is then called 'summing
of kernels'.

Summing the kernel expansions is, in general, a difficult thing to do. The
proving positivity of the kernels is another difficult problem. Only some are
known and have relatively simple forms. In most cases, sums are in the form of
a complex finite sum of the so-called basic hypergeometric functions.

If such a kernel is nonnegative for all $x$ and $y$ that belong to the
supports of measures $\mu$ and $\nu$ that make orthogonal the following
families of polynomials respectively $\left\{  p_{n}\right\}  $ and $\left\{
q_{n}\right\}  ,$ then such a kernel is called (at least among probabilists) a
Lancaster kernel.

Here, below we have another, the probabilistic application of $q-$series theory.

Properties and applications in the theory of probability of such kernels were
described in the series of papers of H.G. Lancaster \cite{Lancaster58},
\cite{Lancaster63(1)}, \cite{Lancaster63(2)}, \cite{Lancaster75}.

If we do not deal with Lancaster kernels, then their importance in
applications stems directly from Mercer's theorem and the theory following it.
Anyway, summing kernels is an important and ambitious task.

In the case of Lancaster kernels, it turns out that in many cases the number
sequence $\left\{  r_{n}\right\}  $ has to be a moment sequence. It is the
necessary condition in case of unbounded supports of the measures $\mu$ and
$\nu$.

We start with the famous Poisson--Mehler expansion of $f_{CN}\left(
x|y,\rho,q\right)  /f_{N}\left(  x|q\right)  $ in an infinite series of
Mercer's type (compare e.g., \cite{Mercier09}). Namely, the following fact is true:

\begin{theorem}
\label{Mehler}For all $\forall\left\vert q\right\vert ,\left\vert
\rho\right\vert <1,~x,y\in S\left(  q\right)  $ we have%
\begin{gather}
\frac{(\rho^{2})_{\infty}}{\prod_{k=0}^{\infty}W_{q}\left(  x,y|\rho
q^{k}\right)  }\label{PM}\\
=\sum_{n=0}^{\infty}\frac{\rho^{n}}{[n]_{q}!}H_{n}(x|q)H_{n}(y|q).\nonumber
\end{gather}

For $q\allowbreak=\allowbreak1,$ $x,y\in\mathbb{R}$ we have%
\begin{equation}
\frac{\exp\left(  \frac{x^{2}+y^{2}}{2}\right)  }{\sqrt{1-\rho^{2}}}%
\exp\left(  -\frac{x^{2}+y^{2}-2\rho xy}{2(1-\rho^{2})}\right)  \allowbreak
=\allowbreak\sum_{n=0}^{\infty}\frac{\rho^{n}}{n!}H_{n}(x)H_{n}(y).
\label{PMq=1}%
\end{equation}

\end{theorem}

\begin{proof}
There exist many proofs of both formulae (see e.g., \cite{IA}, \cite{bressoud}%
). One of the shortest, exploiting connection coefficients, given in
(\ref{PnaH}) is given in \cite{Szab4}. The simplest seems to be based on
(\ref{Mehler}), (\ref{n_form}) and then introducing parameters \ref{n_param}.
\end{proof}

\begin{corollary}
For all $\forall\left\vert q\right\vert ,\left\vert \rho\right\vert <1,~x\in
S\left(  q\right)  $ we have%
\[
\sum_{k\geq0}\frac{\rho^{k}\left(  \rho q^{k+1}\right)  _{\infty}}{\left[
k\right]  _{q}!}H_{2k}\left(  x|q\right)  \allowbreak=\allowbreak\frac{\left(
\rho^{2}\right)  _{\infty}}{\left(  \rho\right)  _{\infty}}\prod_{k=0}%
^{\infty}L_{q}^{-1}\left(  x|\rho q^{k}\right)  .
\]

\end{corollary}

\begin{proof}
We put $y\allowbreak=\allowbreak x$ in (\ref{PM}), then we apply a modified
version of (\ref{lihh}), change the order of summation and finally apply
formulae $\frac{1}{\left(  \rho\right)  _{j+1}}\allowbreak=\allowbreak
\sum_{k\geq0}%
%TCIMACRO{\QATOPD{[}{]}{j+k}{k}}%
%BeginExpansion
\genfrac{[}{]}{0pt}{}{j+k}{k}%
%EndExpansion
_{q}\rho^{k}$ and $\frac{\left(  \rho\right)  _{\infty}}{\left(  \rho\right)
_{j+1}}\allowbreak=\allowbreak\left(  q^{j+1}\rho\right)  _{\infty}$.
\end{proof}

\begin{remark}
This formula has been obtained by other means in the first section compare
with (\ref{genh2})
\end{remark}

We will call the expression of the form on the right-hand side of (\ref{PM})
the kernel expansion, while the expressions from the left-hand side of
(\ref{PM}) kernels. The name refers to Mercer's theorem and the fact that for
example
\[
\int_{S\left(  q\right)  }k\left(  x,y|\rho,q\right)  H_{n}\left(  x|q\right)
f_{N}\left(  x|q\right)  dx=\rho^{n}H_{n}\left(  y|q\right)  f_{N}\left(
y|q\right)  ,
\]
where we denoted by $k\left(  x,y|\rho,q\right)  $ the left-hand side of
(\ref{PM}). Hence we see that $k$ is a kernel, while functions $H_{n}\left(
x|q\right)  f_{N}\left(  x|q\right)  $ are eigenfunctions of the kernel $k$
with $\rho^{n}$ being an eigenvalue related to an eigenfunction $H_{n}\left(
x|q\right)  f_{N}\left(  x|q\right)  $. Such kernels and kernel expansions are
very important in the analysis or in quantum physics in the analysis of
different models of harmonic oscillators.

Below we will present several of them, mostly the ones involving the big
$q-$Hermite, Al-Salam--Chihara and $q-$ultraspherical polynomials.

To present more complicated sums, we will need the following definition of the
basic hypergeometric function, namely
\begin{equation}
_{j}\phi_{k}\left[
\begin{array}
[c]{cccc}%
a_{1} & a_{2} & \ldots & a_{j}\\
b_{1} & b_{2} & \ldots & b_{k}%
\end{array}
;q,x\right]  =\sum_{n=0}^{\infty}\frac{\left(  a_{1},\ldots,a_{j}\right)
_{n}}{\left(  q,b_{1},\ldots,b_{k}\right)  _{n}}\left(  \left(  -1\right)
^{n}q^{\binom{n}{2}}\right)  ^{1+k-j}x^{n}, \label{jfk}%
\end{equation}

\begin{gather}
_{2m}W_{2m-1}\left(  a,a_{1},\ldots,a_{2m-3};q,x\right)  =\label{2mW2m-1}\\
_{2m}\phi_{2m-1}\left[
\begin{array}
[c]{ccccccc}%
a & q\sqrt{a} & -q\sqrt{a} & a_{1} & a_{2} & \ldots & a_{2m-3}\\
\sqrt{a} & -\sqrt{a} & \frac{qa}{a_{1}} & \frac{qa}{a_{2}} & \ldots &
\frac{qa}{a_{2m-3}} &
\end{array}
;q,x\right]  .\nonumber
\end{gather}

We will present now the kernels built of families of polynomials that are
discussed here and their sums.

\begin{theorem}
\label{Kernels}i) For all $\left\vert t\right\vert <1,\left\vert x\right\vert
,\left\vert y\right\vert <2:$
\[
\sum_{n=0}^{\infty}t^{n}U_{n}\left(  x/2\right)  U_{n}\left(  y/2\right)
\allowbreak=\allowbreak\frac{\left(  1-t^{2}\right)  }{\left(  \left(
1-t^{2}\right)  ^{2}-t\left(  1+t^{2}\right)  xy+t^{2}(x^{2}+y^{2})\right)
}.
\]

ii) For all $\left\vert t\right\vert <1,\left\vert x\right\vert ,\left\vert
y\right\vert <1:$%
\begin{gather*}
\sum_{n=0}^{\infty}\frac{\left(  1-\beta q^{n}\right)  \left(  q\right)  _{n}%
}{\left(  1-\beta\right)  \left(  \beta^{2}\right)  _{n}}t^{n}C_{n}\left(
x|\beta,q\right)  C_{n}\left(  y|\beta,q\right)  \allowbreak=\\
\frac{\left(  \beta q\right)  _{\infty}^{2}}{\left(  \beta^{2}\right)
_{\infty}\left(  \beta t^{2}\right)  _{\infty}}\prod_{n=0}^{\infty}%
\frac{w\left(  x,y|t\beta q^{n}\right)  }{w\left(  x,y|tq^{n}\right)  }%
\times\\
~_{8}W_{7}\left(  \frac{\beta t^{2}}{q},\frac{\beta}{q},te^{i(\theta+\phi
)},te^{-i\left(  \theta+\phi\right)  },te^{i\left(  \theta-\phi\right)
},te^{-i\left(  \theta-\phi\right)  };q,\beta q\right)  ,
\end{gather*}
\newline where $x\allowbreak=\allowbreak\cos\theta,$ $y\allowbreak
=\allowbreak\cos\phi$.

iii) For all $\left\vert x\right\vert ,\left\vert y\right\vert ,\left\vert
t\right\vert ,\left\vert tb/a\right\vert \leq1$:%
\begin{gather}
\sum_{n\geq0}\frac{\left(  tb/a\right)  ^{n}}{\left(  q\right)  _{n}}%
h_{n}\left(  x|a,q\right)  h_{n}\left(  y|b,q\right)  \allowbreak
=\allowbreak\left(  \frac{b^{2}t^{2}}{a^{2}}\right)  _{\infty}\prod
_{k=0}^{\infty}\frac{v\left(  x|tbq^{k}\right)  }{w\left(  x,y|t\frac{b}%
{a}q^{k}\right)  }\times\label{kernel_bigH}\\
_{3}\phi_{2}\left(
\begin{array}
[c]{ccc}%
t & bte^{i\left(  \theta+\phi\right)  }/a & bte^{i\left(  -\theta+\phi\right)
}/a\\
b^{2}t^{2}/a^{2} & bte^{i\phi} &
\end{array}
;q,be^{-i\phi}\right)  ,\nonumber
\end{gather}
with $x\allowbreak=\allowbreak\cos\theta$ and $y\allowbreak=\allowbreak
\cos\phi$.

iv) For all $\left\vert t\right\vert <1,x,y\in S\left(  q\right)
,ab=\alpha\beta:$
\begin{gather*}
\sum_{n\geq0}\frac{\left(  t\alpha/a\right)  ^{n}}{\left(  q\right)
_{n}\left(  ab\right)  _{n}}Q_{n}\left(  x|a,b,q\right)  Q_{n}\left(
y|\alpha,\beta,q\right)  \allowbreak=\allowbreak\\
\frac{\left(  \frac{\alpha^{2}t^{2}}{a},\frac{\alpha^{2}t}{a}e^{i\theta
},be^{-i\theta},bte^{i\theta},\alpha te^{-i\phi},\alpha te^{i\phi}\right)
_{\infty}}{\left(  ab,\frac{\alpha^{2}t^{2}}{a}e^{i\theta}\right)  _{\infty
}\prod_{k=0}^{\infty}w\left(  x,y|\frac{\alpha t}{a}q^{k}\right)  }\times\\
_{8}W_{7}\left(  \frac{\alpha^{2}t^{2}e^{i\theta}}{aq},t,\frac{\alpha t}%
{\beta},ae^{i\theta},\frac{\alpha t}{a}e^{i\left(  \theta+\phi\right)  }%
,\frac{\alpha t}{a}e^{i\left(  \theta-\phi\right)  };q,be^{-i\theta}\right)  ,
\end{gather*}
where as before $x\allowbreak=\allowbreak\cos\theta$ and $y\allowbreak
=\allowbreak\cos\phi$ and \newline%
\begin{gather*}
\sum_{n\geq0}\frac{t^{n}}{\left(  q\right)  _{n}\left(  ab\right)  _{n}}%
Q_{n}\left(  x|a,b,q\right)  Q_{n}\left(  y|\alpha,\beta,q\right)
\allowbreak=\allowbreak\\
\frac{\left(  \frac{\beta t}{a}\right)  _{\infty}}{\left(  \alpha at\right)
_{\infty}}\prod_{k=0}^{\infty}\frac{(1+\alpha^{2}t^{2}q^{2k})^{2}-2\alpha
tq^{k}\left(  x+y\right)  \left(  1+\alpha^{2}t^{2}q^{2k}\right)  +4\alpha
^{2}xyt^{2}q^{2k}}{w\left(  x,y|tq^{k}\right)  }\\
~_{8}W_{7}\left(  \frac{\alpha at}{q},\frac{\alpha t}{b},ae^{i\theta
},ae^{-i\theta},\alpha e^{i\phi},\alpha e^{-i\phi};q;\frac{\beta t}{a}\right)
.
\end{gather*}

v) For all $\left\vert \rho_{1}\right\vert ,\left\vert \rho_{2}\right\vert
,\left\vert q\right\vert <1,$ $x,y\in S\left(  q\right)  $
\begin{equation}
0\leq\sum_{n\geq0}\frac{\rho_{1}^{n}}{\left[  n\right]  _{q}!\left(  \rho
_{2}^{2}\right)  _{n}}P_{n}\left(  x|y,\rho_{2},q\right)  P_{n}\left(
z|y,\frac{\rho_{2}}{\rho_{1}},q\right)  =\frac{\left(  \rho_{1}^{2}\right)
_{\infty}}{\left(  \rho_{2}^{2}\right)  _{\infty}}\prod_{k=0}^{\infty}%
\frac{W_{q}\left(  x,z|\rho_{2}q^{k}\right)  }{W_{q}\left(  x,y|\rho_{1}%
q^{k}\right)  }. \label{kerASC}%
\end{equation}

\end{theorem}

\begin{proof}
[Remarks concerning the proof]i) We set $q\allowbreak=\allowbreak0$ in
(\ref{PM}) and use the fact that $H_{n}(x|0)\allowbreak=\allowbreak
U_{n}(x/2)$. ii) It is formula (1.7) in \cite{Rahman97} based on
\cite{GasRah}. iii) it is formula (14.14) in \cite{suslov96}. iv) these are
formulae (14.5) and (14.8) of \cite{suslov96}. v) Notice that it cannot be
derived from assertion iv) since the condition $ab\allowbreak=\allowbreak
\alpha\beta$ is not satisfied. Recall that (see (\ref{podstawienie}))
$ab\allowbreak=\allowbreak\rho_{2}^{2}$ while $\alpha\beta\allowbreak
=\allowbreak\rho_{1}^{2}$. For the proof, recall the idea of expansion of
ratio of densities presented in \cite{Szab4}, use formulae (\ref{PnaP}) and
(\ref{PnPm}) and finally notice that $f_{CN}\left(  x|y,\rho_{1},q\right)
/f_{CN}\left(  x|z,\rho_{2},q\right)  \allowbreak=\allowbreak\frac{\left(
\rho_{1}^{2}\right)  _{\infty}}{\left(  \rho_{2}^{2}\right)  _{\infty}}%
\prod_{k=0}^{\infty}\frac{W_{q}\left(  x,z|\rho_{2}q^{k}\right)  }%
{W_{q}\left(  x,y|\rho_{1}q^{k}\right)  }$.
\end{proof}

\begin{corollary}
For all $\left\vert a\right\vert >\left\vert b\right\vert ,$ $x,y\in S\left(
q\right)  :$%
\[
0\leq\sum_{n\geq0}\frac{b^{n}}{\left[  n\right]  _{q}!a^{n}}H_{n}\left(
x|a,q\right)  H_{n}\left(  y|b,q\right)  \allowbreak=\allowbreak\left(
\frac{b^{2}}{a^{2}}\right)  _{\infty}\prod_{k=0}^{\infty}\frac{V_{q}\left(
x|bq^{k}\right)  }{W_{q}\left(  x,y|\frac{b}{a}q^{k}\right)  }.
\]

\end{corollary}

\begin{proof}
We set $t\allowbreak=\allowbreak0$ in (\ref{kernel_bigH}) and assume
$\left\vert b\right\vert <\left\vert a\right\vert $. For an alternative simple
proof see \cite{SzablKer}.
\end{proof}

\subsection{Expansions of kernel's reciprocals}

We have the following infinite expansions:

\begin{theorem}
i) For $\left\vert q\right\vert ,\left\vert \rho\right\vert <1,x,y\in S\left(
q\right)  :$%
\[
1/\sum_{n=0}^{\infty}\frac{\rho^{n}}{\left[  n\right]  _{q}!}H_{n}\left(
x|q\right)  H_{n}\left(  y|q\right)  \allowbreak=\allowbreak\sum_{n=0}%
^{\infty}\frac{\rho^{n}}{\left(  \rho^{2}\right)  _{n}[n]_{q}!}B_{n}\left(
y|q\right)  P_{n}\left(  x|y,\rho,q\right)  .
\]

ii) For $x,y\in\mathbb{R}$ and $\rho^{2}<1/2$%
\[
1/\sum_{n=0}^{\infty}\frac{\rho^{n}}{n!}H_{n}\left(  x\right)  H_{n}\left(
y\right)  \allowbreak=\allowbreak\sum_{n=0}^{\infty}\frac{\rho^{n}i^{n}%
}{n!\left(  1-\rho^{2}\right)  ^{n/2}}H_{n}\left(  ix\right)  H_{n}\left(
\frac{(x-\rho y)}{\sqrt{1-\rho^{2}}}\right)  .
\]

iii) For $\left\vert q\right\vert <1,\left\vert a\right\vert <\left\vert
b\right\vert ,x,y\in S\left(  q\right)  :$
\[
1/\sum_{n\geq0}\frac{a^{n}}{\left[  n\right]  _{q}!b^{n}}H_{n}\left(
x|a,q\right)  H_{n}\left(  y|b,q\right)  =\sum_{n\geq0}\frac{a^{n}}{\left[
n\right]  _{q}!b^{n}\left(  a^{2}/b^{2}\right)  _{n}}B_{n}\left(
y|b,q\right)  P_{n}\left(  x|y,a/b,q\right)  .
\]

iv) For $\left\vert \rho_{1}\right\vert ,\left\vert \rho_{2}\right\vert
,\left\vert q\right\vert <1,$ $x,y\in S\left(  q\right)  :$%
\begin{align*}
&  1/\sum_{n\geq0}\frac{\rho_{1}^{n}}{\left[  n\right]  _{q}!\left(  \rho
_{2}^{2}\right)  _{n}}P_{n}\left(  x|y,\rho_{2},q\right)  P_{n}\left(
z|y,\frac{\rho_{2}}{\rho_{1}},q\right) \\
&  =\sum_{n\geq0}\frac{\rho_{2}^{n}}{\left[  n\right]  _{q}!\left(  \rho
_{1}^{2}\right)  _{n}}P_{n}\left(  x|z,\rho_{1},q\right)  P_{n}\left(
y|z,\frac{\rho_{1}}{\rho_{2}},q\right)  .
\end{align*}

\end{theorem}

\begin{proof}
[Remarks concerning the proof]i) and ii) are proved in \cite{Szab4}. iii) is
proved in \cite{SzablKer}. iv) directly follows (\ref{kerASC}).
\end{proof}

\end{document}